\documentclass[11pt]{amsart}
\usepackage[a4paper,margin=1in]{geometry}
\usepackage[T1]{fontenc}
\usepackage{lmodern,microtype,amsmath,amssymb,amsthm,mathtools,booktabs,needspace}
\usepackage{enumitem}
\usepackage{tikz}
\usetikzlibrary{arrows.meta,calc,decorations.pathreplacing}
\usepackage{aliascnt}
\usepackage[colorlinks=true,linkcolor=blue,citecolor=blue,urlcolor=blue]{hyperref}
\usepackage[nameinlink,capitalise,noabbrev]{cleveref}
\allowdisplaybreaks
\numberwithin{equation}{section}
\newtheorem{theorem}{Theorem}[section]
\newaliascnt{proposition}{theorem}
\newtheorem{proposition}[proposition]{Proposition}
\aliascntresetthe{proposition}
\newaliascnt{lemma}{theorem}
\newtheorem{lemma}[lemma]{Lemma}
\aliascntresetthe{lemma}
\newaliascnt{corollary}{theorem}
\newtheorem{corollary}[corollary]{Corollary}
\aliascntresetthe{corollary}
\theoremstyle{definition}
\newaliascnt{definition}{theorem}
\newtheorem{definition}[definition]{Definition}
\aliascntresetthe{definition}
\theoremstyle{remark}
\newaliascnt{remark}{theorem}
\newtheorem{remark}[remark]{Remark}
\aliascntresetthe{remark}
\crefname{theorem}{Theorem}{Theorems}
\Crefname{theorem}{Theorem}{Theorems}
\crefname{proposition}{Proposition}{Propositions}
\Crefname{proposition}{Proposition}{Propositions}
\crefname{lemma}{Lemma}{Lemmas}
\Crefname{lemma}{Lemma}{Lemmas}
\crefname{corollary}{Corollary}{Corollaries}
\Crefname{corollary}{Corollary}{Corollaries}
\crefname{definition}{Definition}{Definitions}
\Crefname{definition}{Definition}{Definitions}
\crefname{remark}{Remark}{Remarks}
\Crefname{remark}{Remark}{Remarks}

\usepackage{color}

\hypersetup{pdftitle={Condensation and Non-condensation in the Quasilinear 4-wave Kinetic Equation},pdfauthor={Siwei Luo and Jian-Guo Liu}}
\title[Condensation and Non-condensation in the Quasilinear 4-wave Kinetic Equation]
{Condensation and Non-condensation\\
in the Quasilinear 4-wave Kinetic Equation}
\author{Siwei Luo}
\address{School of the Gifted Young, University of Science and Technology
of China, Hefei 230026, China}
\email{luosw@mail.ustc.edu.cn}
\author{Jian-Guo Liu}
\address{Department of Mathematics and Department of Physics,
Duke University, Durham, NC 27708, USA}
\email{jian-guo.liu@duke.edu}
\date{October 2026}

\begin{document}
\begin{abstract}
We study finite-time condensation at zero frequency for the isotropic quasilinear
4-wave kinetic equation in three dimensions. With Schr\"odinger dispersion and a product interaction coefficient with exponent \(\beta\in (0,3/4]\), we construct global measure solutions \((g_t)_{t\geq 0}\) for every nonnegative initial measure \(g_0\) with finite wave action and energy. We then establish the sharp threshold \(\beta=1/4\) for finite-time condensation of the constructed solutions. For \(\beta\in [1/4,3/4],\) the constructed solutions satisfy \(g_t(\{0\})=g_0(\{0\})\) for any \(t\geq 0\). For \(\beta\in (0,1/4)\),
every constructed solution with \((g_0)_{\rm ac}\neq 0\) satisfies
\[\|g_t-\mathcal N_0\delta_0\|_{\mathrm{TV}}
\le C\mathcal N_0(1+\mathcal N_0^2t)^{-(1-4\beta)/40}\]
for all \(t\ge0\), where
\(\mathcal N_0=g_0([0,\infty))\) and
\(C=C(\beta,g_0/\mathcal N_0)>0\).
In particular, for every \(\zeta\in(0,1)\), there exists
\(T_\zeta<\infty\) such that
\(g_t(\{0\})\ge(1-\zeta)\mathcal N_0\)
for all \(t\ge T_\zeta\).
\end{abstract}
\maketitle
\tableofcontents

\section{Introduction}\label{sec:model-main}

The wave kinetic equation (WKE) describes the evolution of
the wave spectrum in the weak-turbulence approximation to the
cubic nonlinear Schr\"odinger equation \cite{DNPZ92}.
Resonant interactions can transfer wave action towards low
frequencies and energy towards high frequencies, giving rise
to an inverse wave-action cascade and a direct energy cascade
\cite{ZLF92,Nazarenko11}.
The accumulation of a positive amount of wave action at zero
frequency can form a condensate, while at the other end of the spectrum,
the transfer of a positive amount of energy to infinite frequency
leads to possible energy loss from the kinetic solution in finite time.

Wave kinetic equations can also describe more general Hamiltonian
wave systems
\cite{ZLF92,MMT97}.
In these systems, the coupling strength can depend on the
participating wave vectors through an interaction coefficient.
In this paper, we study how this dependence affects finite-time condensate
formation.
We consider a family of 4-wave kinetic equations
in three dimensions with Schr\"odinger dispersion \(\omega( k)=| k|^2\) and a product
interaction coefficient depending on an exponent $\beta \in (0,3/4]$. Varying $\beta$ can change the interaction strength for different frequency scales and lead to possible phase transition of action condensation.
Indeed, we construct global isotropic measure solutions and then prove
a sharp threshold in $\beta$ for finite-time condensate formation
for the constructed solutions.

\subsection{The model}
We consider the three-dimensional quasilinear Schr\"odinger equation
\begin{equation}\label{eq:parent-pde}
 i\partial_tu+\Delta u
 =\kappa|\nabla|^\beta
 \left(\bigl||\nabla|^\beta u\bigr|^2|\nabla|^\beta u\right),
 \qquad  x\in\mathbb R^3,
\end{equation}
where $\kappa\in\mathbb R\setminus\{0\}$, $\beta\in(0,3/4]$, and
$|\nabla|^\beta=(-\Delta)^{\beta/2}$.
It is generated by the Hamiltonian
\begin{equation}\label{eq:hamiltonian}
 \mathcal H(u)
 =\int_{\mathbb R^3}|\nabla u|^2\,d x
 +\frac{\kappa}{2}\int_{\mathbb R^3}
       \bigl||\nabla|^\beta u\bigr|^4\,d x,
\end{equation}
through $i\partial_tu=\delta\mathcal H/\delta\overline u$.
At $\beta=0$, this equation reduces to the cubic
nonlinear Schr\"odinger equation.

Now consider the corresponding wave kinetic equation of \eqref{eq:parent-pde}.
For a spatially homogeneous random wave field, let
$n(t, k)$ denote its wave-action spectrum.
The standard weak-turbulence theory formally leads to a
4-wave kinetic equation \cite{ZLF92,Nazarenko11}.
After absorbing the prefactor proportional to $\kappa^2$
into time, we obtain
\begin{align}\label{eq:kinetic-k}
\partial_tn_4 &=\mathcal C_\beta[n]( k_4),\qquad n_j=n(t, k_j),\quad \omega_j=\omega(k_j)=|k_j|^2,\quad j=1,2,3,4.\\
\mathcal C_\beta[n]( k_4)
 &={}\int_{\mathbb R^9}
 |T_{1234}|^2
  \delta^{(3)}(\Sigma)\delta(\Omega)
 \bigl[n_1n_2(n_3+n_4)-n_3n_4(n_1+n_2)\bigr]
 \,d k_1\,d k_2\,d k_3.
\end{align}
Here
$\Sigma= k_1+ k_2
- k_3- k_4$ and
$\Omega=\omega_1+\omega_2-\omega_3-\omega_4$. Since $|\nabla|^\beta$ has Fourier symbol $| k|^\beta$,
the quartic Hamiltonian \eqref{eq:hamiltonian} gives the interaction coefficient
$T_{1234}=(| k_1|| k_2|
| k_3|| k_4|)^\beta$. The Dirac delta functions impose the resonance relations
$\Sigma=0$ and $\Omega=0$ in \eqref{eq:kinetic-k}.

The periodic version of \eqref{eq:parent-pde} with $\kappa=1$ and its associated kinetic equation are discussed by Ampatzoglou and L\'eger \cite{AL25}. Zhu \cite{Zhu26} then considers the same product interaction coefficient with power-law dispersion $\omega( k)=| k|^a$, $1\le a\le5$. The derivative nonlinearity also appears in the one-dimensional family introduced by Majda, McLaughlin, and Tabak \cite{MMT97} to study dispersive wave turbulence, allowing both the dispersion relation and the interaction coefficient to vary. Here we fix the Schr\"odinger dispersion and study how the interaction exponent $\beta$ affects condensate formation.

Define
the total wave action and the total energy of the kinetic equation
\begin{equation}
     \mathcal N(t)=\int_{\mathbb R^3}
n(t, k)\,d k,\qquad \mathcal E(t)=\int_{\mathbb R^3}
| k|^2n(t, k)\,d k.
\end{equation}
For isotropic spectra, we write $n(t,\omega)$ for
$n(t,k)$, where $\omega=|k|^2$.
The change of variable gives the wave-action
measure $g_t(d\omega)=g(t,\omega)\,d\omega$, with
$g(t,\omega)=2\pi\sqrt{\omega}\,n(t,\omega)$.
Thus
\begin{equation}
     \mathcal N(t)=\int_{[0,\infty)}dg_t,\qquad  \mathcal E(t)=\int_{[0,\infty)}\omega\,dg_t.
\end{equation}
Note that we allow $g_t$ to be a measure on $[0,\infty)$ instead of a function to include possible condensates. We now state the radial equation in weak form of \eqref{eq:kinetic-k} with respect to \(g_t\).

Throughout the paper, we write $a\wedge b:=\min\{a,b\}$ and
$a\vee b:=\max\{a,b\}$.

Set
$\omega_4=\omega_1+\omega_2-\omega_3$. For $\omega_1,\omega_2,\omega_3>0$, define
\begin{equation}\label{eq:kernel}
 K_\beta(\omega_1,\omega_2,\omega_3)
 =(\omega_1\omega_2\omega_3\omega_4)^\beta
 \dfrac{\min\{\sqrt{\omega_1},\sqrt{\omega_2},\sqrt{\omega_3},\sqrt{\omega_4}\}}
       {\sqrt{\omega_1\omega_2\omega_3}}
\end{equation}
when $\omega_4>0$. We set \(K_\beta=0\) whenever one of the four frequencies is nonpositive.

Let $\mathcal M_1$ be the Banach space of finite signed Borel
measures $g$ on $[0,\infty)$ satisfying
\begin{equation}
\|g\|_{\mathcal M_1}:=\int_{[0,\infty)}(1+\omega)\,d|g|<\infty.
\end{equation}
We also write
\begin{equation}
\operatorname{Lip}_c([0,\infty))
=\{\phi\text{ Lipschitz on }[0,\infty):
\phi\text{ has compact support}\}.
\end{equation}

\begin{definition}[Isotropic measure solution]\label{def:weak-form}
    Let \(g_0\in\mathcal M_1\) be nonnegative. A narrowly continuous curve \((g_t)_{t\geq0}\) of nonnegative measures in \(\mathcal M_1\) is called an \emph{isotropic measure solution} of \eqref{eq:kinetic-k} if \(g_{t=0}=g_0\),
    \(
    \sup_{0\leq t\leq T}\bigl(\mathcal N(t)+\mathcal E(t)\bigr)<\infty
    \) for every \(T<\infty\), and for every \(\phi\in \operatorname{Lip_c}([0,\infty))\) and \(T<\infty\),
    \begin{equation}
    \int_0^T\iiint_{[0,\infty)^3}  K_\beta(\omega_1,\omega_2,\omega_3)|\Delta\phi|\, dg_\tau^{\otimes 3}\,d\tau<\infty,
    \end{equation}
    and for \(0\leq s\leq t<\infty\),
    \begin{equation}\label{eq:weak}
        \int_{[0,\infty)}\phi\, dg_t-\int_{[0,\infty)}\phi\, dg_s=\dfrac{1}{4}\int_s^t\iiint_{[0,\infty)^3} K_\beta(\omega_1,\omega_2,\omega_3)\Delta\phi(\omega_1,\omega_2,\omega_3)\, dg_\tau^{\otimes 3}d\tau,
    \end{equation}
    where \(dg_\tau^{\otimes 3}=dg_\tau(\omega_1)dg_\tau(\omega_2)dg_\tau(\omega_3)\) and \(\Delta\phi=\phi(\omega_3)+\phi(\omega_4)-\phi(\omega_1)-\phi(\omega_2)\) when \(\omega_4>0\). We set \(\Delta\phi=0\) when \(\omega_4\le0\).
\end{definition}

For $\beta=0$, the corresponding weak formulation with a
non-interacting condensate is given in \cite{EV15},
and a derivation with general interaction coefficients is given
in \cite{MerinoAceituno16}. For the product coefficient, one can refer to \cite{AL25} for the details.

We now define finite-time condensation. A \emph{condensate} is represented by a positive atom of $g_t$ at
$\omega=0$, i.e. \(g_t(\{0\})>0\) for some \(t<\infty\) \cite{EV15,ST26}. 
We say that \emph{finite-time condensate formation} occurs if
$g_0(\{0\})=0$ and $g_T(\{0\})>0$ for some $0<T<\infty$.

Our results identify the condensation threshold in the range $\beta\in (0,3/4]$:
condensate formation can occur precisely below $\beta=1/4$. For \(\beta>3/4\), finite wave action and energy no longer guarantee absolute convergence of the collision integral on the right-hand side of \eqref{eq:weak}, and this range is therefore not considered in the present paper.

\begin{remark}[Non-interacting condensate]
     We use a weak formulation with a non-interacting condensate here. Since the collision kernel \(K_\beta\) vanishes whenever one of the four frequencies is zero, a condensate atom already present at zero does not contribute directly to the collision integral.
\end{remark}
\begin{remark}[The coefficient \(1/4\) in the weak formulation]
     With the normalization in \eqref{eq:kinetic-k} and
\(g(t,\omega)=2\pi\sqrt{\omega}\,n(t,\omega)\),
the radial weak formulation has the coefficient \(1/4\)
in \eqref{eq:weak}.
The calculation is given in \cite{AL25}.
\end{remark}
\begin{remark}[The boundary \(\omega_4=0\)]
\label{rem:collision-boundary}
The zero extension of \(\Delta\phi\) in \cref{def:weak-form} may be discontinuous at
\(\omega_4=0\).
However, for the passage to the limit in the weak equation in \cref{sec:measure}, we use
a symmetrized integrand \(F_\beta[\phi]\).
By \cref{cor:boundary-removal}, this function is continuous on
\(\mathcal D_{\mathrm{ord}}\), and
\(F_{\beta,N}[\phi]\to F_\beta[\phi]\) locally uniformly there.
\end{remark}
\subsection{Main results} For a nonnegative initial measure $g_0\in\mathcal M_1$, write
\begin{equation}
    {\mathcal N}_0=\int_{[0,\infty)} dg_0,\qquad {\mathcal E}_0=\int_{[0,\infty)}\omega\,dg_0.
\end{equation}
In \cref{sec:measure}, we will construct isotropic measure solutions in the sense of \cref{def:weak-form} by introducing symmetric cutoff equations and passing to the limit as the cutoff is removed. Let \(g^N=(g_t^N)_{t\ge0}\) denote the cutoff solution with initial value \(g_0\).

\begin{theorem}[Global measure solutions]\label{thm:measure-main}
Let $\beta\in (0,3/4]$ and let $g_0\in\mathcal M_1$ be nonnegative. There exists a global isotropic measure solution \((g_t)_{t\geq0}\) of \eqref{eq:weak} in the sense of \cref{def:weak-form} with initial value \(g_0\) satisfying \begin{equation}
\mathcal N(t)=\mathcal N_0,\qquad \mathcal E(t)\leq \mathcal E(s)\leq \mathcal E_0, \qquad 0\leq s\leq t<\infty.
\end{equation}
\end{theorem}

We now state the threshold \(\beta=1/4\) for finite-time action condensation of the solutions constructed in \cref{thm:measure-main}.

\begin{theorem}[Condensation and non-condensation threshold]\label{thm:action}
Let \(\beta\in (0,3/4]\). The following statements hold.
\begin{itemize}
    \item If $\beta\in [1/4,3/4]$, then every solution constructed in
\cref{thm:measure-main} and \cref{sec:measure} satisfies
$g_t(\{0\})=g_0(\{0\})$ for every $t\ge0$. In particular, if \(g_0(\{0\})=0\), no condensation forms in finite time.

\item If $\beta\in(0,1/4)$, let $g_0\in\mathcal M_1$ be nonnegative
and have a nonzero absolutely continuous part with respect to
Lebesgue measure. For every $0<\zeta<1$, there is a finite time
$T=T(\beta,g_0,\zeta)$ such that every solution constructed in
\cref{thm:measure-main} and \cref{sec:measure} with initial measure
$g_0$ satisfies
\begin{equation}\label{eq:main-condensation}
 g_t(\{0\})\ge(1-\zeta)\mathcal N_0,
 \qquad t\ge T.
\end{equation}
Moreover, there is a constant $C$, depending only on $\beta$ and
$g_0/\mathcal N_0$, such that all these solutions satisfy
\begin{equation}\label{eq:main-condensation-rate}
 \|g_t-\mathcal N_0\delta_0\|_{\mathrm{TV}}
 \le C\mathcal N_0(1+\mathcal N_0^2t)^{-(1-4\beta)/40},
 \qquad t\ge0.
\end{equation}
\end{itemize}
\end{theorem}

The phase transition results for all $\beta\in(0,3/4]$ are
summarized in \cref{tab:transition}.

\begin{table}[t]
\centering
\small
\caption{Condensation threshold.}
\label{tab:transition}
\begin{tabular}{@{}p{.19\textwidth}p{.73\textwidth}@{}}
\toprule
Parameter \(\beta\) & Condensate formation\\
\midrule
$(0,1/4)$
& Every datum with a nonzero absolutely continuous part
and no atom at zero condenses in finite time.\\[4pt]

$[1/4,3/4]$
& The condensate mass is constant.\\
\bottomrule
\end{tabular}
\end{table}

\begin{remark}[Energy conservation and loss]
\label{rem:energy-threshold}
The energy inequality in \cref{thm:measure-main} allows for
energy loss through transfer to infinite frequency.
In a forthcoming work, we will address the threshold $\beta=2/3$
for the solutions constructed here: energy conservation for
$0<\beta\le2/3$, and finite-time energy loss for suitable initial
data when $2/3<\beta\le3/4$.
\end{remark}

\subsection{Related work}

Wave turbulence theory has its origins in the study of lattice
vibrations, water waves, and random dispersive waves, including the
classical works of Peierls \cite{Peierls29}, Phillips
\cite{Phillips60}, Hasselmann \cite{Hasselmann62}, Benney and
Saffman \cite{BS66}, Benney and Newell \cite{BN69}, and Zakharov
\cite{Zakharov65,Zakharov68}.
It provides a kinetic description of spectral transfer in weakly
nonlinear wave systems, with applications to fluid dynamics,
plasmas, Bose gases, and nonlinear optics.
General accounts of the theory and its physical applications can
be found in the books of Zakharov, Lvov, and Falkovich
\cite{ZLF92} and Nazarenko \cite{Nazarenko11}.

For the 4-wave kinetic equation with a constant interaction coefficient associated with the cubic nonlinear Schr\"odinger equation, Escobedo and Vel\'azquez \cite{EV15}
developed a theory of isotropic weak and mild solutions, including
local and global well-posedness, finite-time condensation, and long-time asymptotics.
Germain, Ionescu, and Tran \cite{GIT20} then established local
well-posedness in weighted $L^\infty$ spaces for a general class of radial
dispersion relations, and in weighted $L^2$ spaces for Schr\"odinger
dispersion, without assuming isotropy of the initial data.
Recently, for a class of convex radial dispersion relations, Staffilani and
Tran \cite{STUV26} constructed global mild radial solutions and
proved transfer of energy to high frequencies as $t\to\infty$
under suitable assumptions on the initial data.
They also established sufficient conditions for finite-time
condensation and estimates on non-condensation times \cite{ST26}.

Equations with frequency-dependent interaction coefficients have
also been studied.
The Majda--McLaughlin--Tabak model was introduced as a family of
one-dimensional equations for studying dispersive wave turbulence
\cite{MMT97}, and Germain, La, and Zhang \cite{GLZ25} established
local well-posedness for its kinetic equation.
Ampatzoglou and L\'eger \cite{AL25} studied the same family of
spatially homogeneous 4-wave kinetic equations considered here,
with quadratic dispersion and interaction weight
$|T_{1234}|^2=(\omega_1\omega_2\omega_3\omega_4)^\beta
=(|k_1||k_2||k_3||k_4|)^{2\beta}$,
associated with a family of quasilinear Schr\"odinger equations.
They proved local well-posedness in suitable weighted $L^\infty$
spaces for $0\le\beta\le1/4$ and ill-posedness in these spaces
for $1/4<\beta\le1$. They also identified $\beta=1/4$ as the boundary between finite
and infinite capacity for the inverse wave-action cascade
\cite[Remark~1.5]{AL25}.
For power-law dispersion $\omega(k)=|k|^a$, $1\le a\le5$,
Zhu \cite{Zhu26} obtained sharp local well-posedness and
ill-posedness criteria in weighted $L^\infty$ spaces, depending on
the interaction exponent and the decay of the initial data.

Condensation has also been studied for the closely related
Boltzmann--Nordheim equation for Bose gases.
A foundational theory of global isotropic measure solutions and
their condensation behavior was developed by Lu
\cite{Lu04,Lu13}.
A central contribution was made by Escobedo and Vel\'azquez
\cite{EV14,EVN15}, who proved finite-time blow-up for classical
solutions and finite-time formation of a condensate for weak
solutions, including the supercritical regime in which the
corresponding Bose--Einstein equilibrium contains a Dirac mass.
Cai and Lu \cite{CL19,CL26} subsequently studied convergence to
Bose--Einstein equilibrium and the influence of the interaction
potential on condensation.
More recently, under scattering assumptions with $0\le\eta<1$,
Luo and Liu \cite{LL26BEC} proved that every admissible initial
measure in the full subcritical regime
$\overline T/\overline T_c<1$ admits a conservative isotropic
solution that develops a condensate in finite time, without any
initial low-energy concentration assumption.

In addition to 4-wave equations, 3-wave kinetic models have also been
studied in several settings, including phonon interactions,
stratified ocean flows, capillary and acoustic waves, and Bose
gases
\cite{Spohn06,GST20,NT18,CTSB20,SoTr18,SoTr20}.
Related existence, well-posedness, condensation, and energy-transfer
phenomena have been investigated in these works.
More recently, Staffilani and Tran \cite{STMixed25} studied a
kinetic model containing both 3-wave and 4-wave interaction terms
and proved finite-time and immediate energy cascades for suitable
classes of initial data.

The rigorous derivation of wave kinetic equations from nonlinear
dispersive equations has also been studied as part of Hilbert's
sixth problem.
Early results concerned equilibrium correlations for the lattice
nonlinear Schr\"odinger equation \cite{LS11} and kinetic
approximations for the cubic nonlinear Schr\"odinger equation
on time scales shorter than the kinetic time
\cite{BGHS21,DH21,CGHom,CGLong}.
Deng and Hani \cite{DH23} derived the homogeneous 4-wave
kinetic equation from the cubic nonlinear Schr\"odinger equation
with random initial data in dimensions $d\ge3$, up to a
sufficiently small multiple of the kinetic time.
They considered the joint large-box and weak-nonlinearity limit
under a particular scaling law.
Their subsequent work covered the full range of admissible
scaling laws \cite{DH26Scaling} and extended the approximation
to every compact interval of rescaled time within the lifespan
of the smooth kinetic solution \cite{DHLong23}.
Staffilani and Tran \cite{STKdV21} justified a 3-wave kinetic
limit, in the resonance broadening sense, for a stochastic
Zakharov--Kuznetsov equation on a lattice.
Related work treats spatially inhomogeneous models
\cite{ACGInhom,HRSTInhom} and damped--driven nonlinear
Schr\"odinger equations \cite{DKKinetic,GHDriven}.

\subsection{Organization of the paper} The rest of the paper is organized as follows. In \cref{sec:measure}, we develop the framework for
measure solutions and prove \cref{thm:measure-main}.
\cref{sec:action-threshold} proves that the condensate
mass remains constant for $1/4\le\beta\le3/4$.
\cref{sec:density-estimates} establishes density inequalities and a
comparison principle for $0<\beta<1/4$.
Using these estimates, \cref{sec:condensation} proves finite-time condensation and
convergence to $\mathcal N_0\delta_0$ for initial measures
with a nonzero absolutely continuous part when $0<\beta<1/4$,
completing the proof of \cref{thm:action}.

\section{Measure solutions}\label{sec:measure}

The main purpose of this section is to introduce the framework
for measure solutions and prove \cref{thm:measure-main}.
The measure-solution framework used here is due to
Escobedo and Vel\'azquez \cite{EV15}.
We adapt their regularization, symmetrization, convex-test
estimates, and compactness argument to the product kernel $K_\beta$.
Throughout this section, $0<\beta\le3/4$ and
$g_0\in\mathcal M_1$ is nonnegative. We construct the solutions
of \cref{def:weak-form} by symmetric cutoff approximation,
with every cutoff equation starting from $g_0$.

\subsection{The cutoff equation}

For \(\omega_1,\omega_2,\omega_3,\omega_4>0\), define
\begin{equation}
 S_\beta(\omega_1,\omega_2,\omega_3,\omega_4)
 =(\omega_1\omega_2\omega_3\omega_4)^\beta
   \min_{1\le j\le4}\sqrt{\omega_j}.
\end{equation}
We extend it by zero whenever one of its arguments is nonpositive.
For each integer $N\ge1$, let
$\ell_N(\omega)=1 \wedge (2N\omega-1)_+$. This function is
continuous and nondecreasing, vanishes on $[0,(2N)^{-1}]$,
and equals one on $[N^{-1},\infty)$. Set
\begin{equation}\label{eq:cutoff}
 \begin{aligned}
 S_{\beta,N}(\omega_1,\omega_2,\omega_3,\omega_4)
   &=(S_\beta(\omega_1,\omega_2,\omega_3,\omega_4)\wedge N)
       \prod_{j=1}^4\ell_N(\omega_j),\\
 K_{\beta,N}(\omega_1,\omega_2,\omega_3)
   &=\frac{S_{\beta,N}(\omega_1,\omega_2,\omega_3,\omega_4)}
           {\sqrt{\omega_1\omega_2\omega_3}},
 \qquad \omega_4=\omega_1+\omega_2-\omega_3.
 \end{aligned}
\end{equation}
As for $K_\beta$, \(K_{\beta,N}\) is defined to be zero
if any of the four frequencies is nonpositive. Hence
$0\le K_{\beta,N}\uparrow K_\beta$ for every frequency argument.
Symmetric cutoffs of isotropic collision kernels are also used in
\cite{Lu14} and \cite{EV15}.

For $N\ge1$, we consider the cutoff equation
\begin{equation}\label{eq:cutoff-weak}
\int_{[0,\infty)}\phi\,dg_t^N
-\int_{[0,\infty)}\phi\,dg_s^N
=\frac14\int_s^t\iiint_{[0,\infty)^3}
K_{\beta,N}\Delta\phi\,d(g_\tau^N)^{\otimes3}\,d\tau,
\end{equation}
for $0\le s\le t$ and $\phi\in\operatorname{Lip}_c([0,\infty))$,
with initial value $g_0^N=g_0$.

\begin{proposition}\label{prop:integrated-cutoff-existence}
Let $g_0\in\mathcal M_1$ be nonnegative.
For every $N\ge1$, equation \eqref{eq:cutoff-weak} with initial
value $g_0$ has a unique global nonnegative solution
$g^N\in C^1([0,\infty);\mathcal M_1)$.
For every $t\ge0$,
\begin{equation}
    \int_{[0,\infty)}dg_t^N=\mathcal N_0,\qquad \int_{[0,\infty)}\omega\,dg_t^N=\mathcal E_0,\qquad g_t^N(\{0\})=g_0(\{0\}).
\end{equation}
Moreover, \eqref{eq:cutoff-weak} holds for every Borel function
$\phi:[0,\infty)\to\mathbb R$ satisfying
$|\phi(\omega)|\le C_\phi(1+\omega)$.
\end{proposition}
\begin{proof}
Fix $N\ge1$ and set $C_N=N(2N)^{3/2}$.

\smallskip
\noindent\textit{Step 1. Estimates for the collision operator.}
Since $K_{\beta,N}$ vanishes whenever
$\min_{1\le j\le4}\omega_j\le(2N)^{-1}$ and \(S_{\beta,N}\leq N\), we have
\begin{equation}
0\le K_{\beta,N}(\omega_1,\omega_2,\omega_3)
\le N(2N)^{3/2}=C_N.
\end{equation}
Set \begin{equation}\mathcal D=\{(\omega_1,\omega_2,\omega_3)\in[0,\infty)^3:
\omega_1+\omega_2\ge\omega_3\}.\end{equation}
For each signed measure $g\in\mathcal M_1$, define the finite signed
measure $\Lambda_N[g]$ on $\mathcal D$ by
$\Lambda_N[g]=K_{\beta,N}(g^{\otimes3}|_{\mathcal D})/4$.
For $j=1,2,3,4$, let $\pi_j:\mathcal D\to[0,\infty)$ be given by
$\pi_j(\omega_1,\omega_2,\omega_3)=\omega_j$.
Define
\begin{equation}
\mathcal I_N(g)
=(\pi_3)_\#\Lambda_N[g]+(\pi_4)_\#\Lambda_N[g]
 -(\pi_1)_\#\Lambda_N[g]-(\pi_2)_\#\Lambda_N[g].
\end{equation}
We claim that \(\mathcal I_N(g)\in\mathcal M_1\). Indeed, since $ \sum_{j=1}^4(1+\omega_j)
=4+2\omega_1+2\omega_2$ on $\mathcal D$, we have
\begin{align}
\|\mathcal I_N(g)\|_{\mathcal M_1}
&\le \sum_{j=1}^4\int_{[0,\infty)}
(1+\omega)\,d|(\pi_j)_\#\Lambda_N[g]|(\omega)\notag\\
&\le \sum_{j=1}^4\iiint_{\mathcal D}
(1+\omega_j)\,d|\Lambda_N[g]|
=\frac14\iiint_{\mathcal D}
(4+2\omega_1+2\omega_2)K_{\beta,N}\,d|g|^{\otimes3}\notag\\
&\le \frac{C_N}{4}\iiint_{[0,\infty)^3}
(4+2\omega_1+2\omega_2)\,d|g|^{\otimes3}
=C_N\|g\|_{\mathrm{TV}}^2\|g\|_{\mathcal M_1}.
\end{align}
For $g,h\in\mathcal M_1$, since
\[
g^{\otimes3}-h^{\otimes3}
=(g-h)\otimes g\otimes g
 +h\otimes(g-h)\otimes g
 +h\otimes h\otimes(g-h),
\]
we have
\begin{align*}
\|\mathcal I_N(g)-\mathcal I_N(h)\|_{\mathcal M_1}
&\le \frac14\iiint_{\mathcal D}
(4+2\omega_1+2\omega_2)K_{\beta,N}
\,d|g^{\otimes3}-h^{\otimes3}|\\
&\le C_N\iiint_{[0,\infty)^3}
\prod_{j=1}^3(1+\omega_j)
\,d|g^{\otimes3}-h^{\otimes3}|\\
&\le C_N\bigl(
\|g\|_{\mathcal M_1}^2
+\|g\|_{\mathcal M_1}\|h\|_{\mathcal M_1}
+\|h\|_{\mathcal M_1}^2
\bigr)\|g-h\|_{\mathcal M_1}.
\end{align*}
Thus
\begin{equation}\label{eq:cutoff-local-lipschitz}
\|\mathcal I_N(g)-\mathcal I_N(h)\|_{\mathcal M_1}
\le C_N(\|g\|_{\mathcal M_1}+\|h\|_{\mathcal M_1})^2
\|g-h\|_{\mathcal M_1},
\end{equation}
so $\mathcal I_N$ is locally Lipschitz on $\mathcal M_1$.

\smallskip
\noindent\textit{Step 2. Local existence and nonnegativity.}
For $g\ge0$, define
\begin{equation}
\mathcal I_N^+(g)
=(\pi_3)_\#\Lambda_N[g]+(\pi_4)_\#\Lambda_N[g],
\qquad
L_N[g](\omega_1)
=\frac12\iint_{[0,\infty)^2}
K_{\beta,N}(\omega_1,\omega_2,\omega_3)
\,dg(\omega_2)\,dg(\omega_3).
\end{equation}
Note that both $\mathcal I_N^+(g)$ and $L_N[g]$ are nonnegative.
By the definition of $\Lambda_N[g]$ and the symmetry of
$K_{\beta,N}$ in $\omega_1,\omega_2$, we have
\begin{align}
\mathcal I_N(g)=\mathcal I_N^+(g)
-(\pi_1)_\#\Lambda_N[g]-(\pi_2)_\#\Lambda_N[g]=\mathcal I_N^+(g)-2(\pi_1)_\#\Lambda_N[g]
=\mathcal I_N^+(g)-L_N[g]g.
\end{align}
The boundedness of $K_{\beta,N}$ then implies
$0\le L_N[g](\omega_1)\le C_N\|g\|_{\mathrm{TV}}^2/2$
for every $\omega_1\ge0$.

Choose $B>\|g_0\|_{\mathcal M_1}$ and set $A=C_NB^2$.
For every nonnegative $g$ with $\|g\|_{\rm TV}\leq \|g\|_{\mathcal M_1}\le B$,
\begin{equation}
\mathcal I_N(g)+Ag
=\mathcal I_N^+(g)+(A-L_N[g])g\ge0.
\end{equation}
For $T>0$, consider
\begin{equation}
\mathcal X_{B,T}
=\left\{\nu\in C([0,T];\mathcal M_1):
\nu_0=g_0,\ \nu_t\ge0,\
\sup_{0\le t\le T}\|\nu_t\|_{\mathcal M_1}\le B\right\}.
\end{equation}
Since the nonnegative cone is closed in $\mathcal M_1$, $\mathcal X_{B,T}$ is a nonempty closed subset of
$C([0,T];\mathcal M_1)$ and is complete for the supremum norm.
Define
\begin{equation}
(\mathcal T\nu)_t
=e^{-At}g_0+\int_0^t e^{-A(t-s)}
\bigl(\mathcal I_N(\nu_s)+A\nu_s\bigr)\,ds,
\qquad 0\le t\le T,
\end{equation}
where the integral is a Bochner integral in $\mathcal M_1$.
Since the integrand is continuous in $s$ and takes values in the
nonnegative cone, $(\mathcal T\nu)_t\ge0$ and
$(\mathcal T\nu)_0=g_0$.
Moreover,
\begin{equation}
\sup_{0\le t\le T}\|(\mathcal T\nu)_t\|_{\mathcal M_1}
\le\|g_0\|_{\mathcal M_1}+T(C_NB^3+AB).
\end{equation}
For $\nu,\mu\in\mathcal X_{B,T}$,
\eqref{eq:cutoff-local-lipschitz} gives
\begin{equation}
\sup_{0\le t\le T}
\|(\mathcal T\nu)_t-(\mathcal T\mu)_t\|_{\mathcal M_1}
\le T(4C_NB^2+A)
\sup_{0\le t\le T}\|\nu_t-\mu_t\|_{\mathcal M_1}.
\end{equation}
Choose $T>0$ such that
$T(C_NB^3+AB)\le B-\|g_0\|_{\mathcal M_1}$ and
$T(4C_NB^2+A)<1$.
Then $\mathcal T$ maps $\mathcal X_{B,T}$ into itself and is a
contraction. The Banach fixed-point theorem gives a nonnegative
fixed point $g^N\in\mathcal X_{B,T}$.
Multiplying the fixed-point identity by $e^{At}$, we obtain
\[
e^{At}g_t^N
=g_0+\int_0^t e^{As}
\bigl(\mathcal I_N(g_s^N)+Ag_s^N\bigr)\,ds.
\]
Since $g^N\in C([0,T];\mathcal M_1)$ and $\mathcal I_N$ is
locally Lipschitz, the integrand is continuous in $\mathcal M_1$.
The fundamental theorem of calculus for Bochner integrals therefore
gives $t\mapsto e^{At}g_t^N\in C^1([0,T];\mathcal M_1)$.
Hence $g^N\in C^1([0,T];\mathcal M_1)$, and
\begin{align*}
\partial_tg_t^N=-Ag_t^N+e^{-At}\partial_t(e^{At}g_t^N)=-Ag_t^N+\mathcal I_N(g_t^N)+Ag_t^N
=\mathcal I_N(g_t^N).
\end{align*}

\smallskip
\noindent\textit{Step 3. Borel test functions and conservation laws.}
Let $\phi:[0,\infty)\to\mathbb R$ be Borel measurable and satisfy
$|\phi(\omega)|\le C_\phi(1+\omega)$.
Hence $ g\mapsto\int\phi\,dg$ is a bounded linear
functional on $\mathcal M_1$.
The estimate in Step~1 then gives, for every $g\in\mathcal M_1$,
\begin{equation}
\frac14\iiint_{[0,\infty)^3}K_{\beta,N}|\Delta\phi|
\,d|g|^{\otimes3}
\le C_\phi C_N\|g\|_{\mathrm{TV}}^2\|g\|_{\mathcal M_1}.
\end{equation}
Thus the collision integral is integrable.
Applying the functional to $\partial_tg_t^N=\mathcal I_N(g_t^N)$
and using the definition of the pushforwards, we obtain
\begin{align}
\dfrac{d}{dt}\int_{[0,\infty)}\phi\,dg_t^N
=\int_{[0,\infty)}\phi\,d\mathcal I_N(g_t^N)
=\frac14\iiint_{[0,\infty)^3}
K_{\beta,N}\Delta\phi\,d(g_t^N)^{\otimes3}. 
\end{align}
Since $t\mapsto\mathcal I_N(g_t^N)$ is continuous in
$\mathcal M_1$, the derivative above is continuous in $t$.
Integrating in time proves \eqref{eq:cutoff-weak} for every
such $\phi$ and $0\le s\le t\le T$.

Taking $s=0$ and testing \eqref{eq:cutoff-weak} with
$\phi(\omega)=1$ and $\phi(\omega)=\omega$,
we obtain
\begin{equation}
\int_{[0,\infty)}dg_t^N
=\int_{[0,\infty)}dg_0=\mathcal N_0,
\qquad
\int_{[0,\infty)}\omega\,dg_t^N
=\int_{[0,\infty)}\omega\,dg_0=\mathcal E_0.
\end{equation}
Moreover, testing \eqref{eq:cutoff-weak} with
$\phi=\mathbf1_{\{0\}}$ 
therefore yields
\begin{equation}
g_t^N(\{0\})-g_0(\{0\})
=\frac14\int_0^t\iiint_{[0,\infty)^3}
K_{\beta,N}\Delta\mathbf1_{\{0\}}\,
d(g_\tau^N)^{\otimes3}\,d\tau=0.
\end{equation}
Since $g_t^N\ge0$, it follows that, for every $t\in[0,T]$,
\begin{equation}
\|g_t^N\|_{\mathcal M_1}
=\int_{[0,\infty)}(1+\omega)\,dg_t^N
=\mathcal N_0+\mathcal E_0
=\|g_0\|_{\mathcal M_1}.
\end{equation}

\smallskip
\noindent\textit{Step 4. Global existence and uniqueness.}
Fix $B>\mathcal N_0+\mathcal E_0$ and choose $T>0$ as in
Step~2. This choice depends only on $N$, $B$, and
$\mathcal N_0+\mathcal E_0$.
The estimates in Step~2 then give a nonnegative solution on
$[0,T]$ for every initial measure $\mu\ge0$ satisfying
$\|\mu\|_{\mathcal M_1}=\mathcal N_0+\mathcal E_0$.
By Step~3, the solution has the same \(\mathcal M_1\) norm throughout this interval.
We may therefore apply this construction successively on
$[kT,(k+1)T]$, $k=0,1,\ldots$, with initial value $g_{kT}^N$.
The resulting curve belongs to
$C([0,\infty);\mathcal M_1)$, and adding the integral equations
on these intervals gives
\begin{equation}
g_t^N=g_0+\int_0^t\mathcal I_N(g_s^N)\,ds,
\qquad t\ge0.
\end{equation}
Since $t\mapsto\mathcal I_N(g_t^N)$ is continuous in
$\mathcal M_1$, this identity implies
$g^N\in C^1([0,\infty);\mathcal M_1)$ and
$\partial_tg_t^N=\mathcal I_N(g_t^N)$.
The conservation identities in Step~3 hold on each interval
$[kT,(k+1)T]$ and hence for every $t\ge0$.

Let $\widetilde g^N\in C^1([0,\infty);\mathcal M_1)$ be another
nonnegative solution of \eqref{eq:cutoff-weak} with initial value
$g_0$. For every $\phi\in\operatorname{Lip}_c([0,\infty))$,
differentiating \eqref{eq:cutoff-weak} in time gives
\begin{align*}
\int_{[0,\infty)}\phi\,d(\partial_t\widetilde g_t^N)=\dfrac{d}{dt}\int_{[0,\infty)}\phi\,d\widetilde g_t^N=\frac14\iiint_{[0,\infty)^3}
K_{\beta,N}\Delta\phi\,d(\widetilde g_t^N)^{\otimes3}
=\int_{[0,\infty)}\phi\,d\mathcal I_N(\widetilde g_t^N).
\end{align*}
Since compactly supported Lipschitz functions separate finite
signed measures, we obtain
$\partial_t\widetilde g_t^N=\mathcal I_N(\widetilde g_t^N)$
in $\mathcal M_1$. The argument of Step~3 therefore applies to
$\widetilde g^N$ and yields
$\|\widetilde g_t^N\|_{\mathcal M_1}
=\mathcal N_0+\mathcal E_0=\|g_t^N\|_{\mathcal M_1}$.
Subtracting the two integral equations and using
\eqref{eq:cutoff-local-lipschitz}, we find
\begin{align}
\|g_t^N-\widetilde g_t^N\|_{\mathcal M_1}
&=\left\|\int_0^t
\bigl(\mathcal I_N(g_s^N)-\mathcal I_N(\widetilde g_s^N)\bigr)
\,ds\right\|_{\mathcal M_1}\notag\\
&\le\int_0^t
\|\mathcal I_N(g_s^N)-\mathcal I_N(\widetilde g_s^N)\|_{\mathcal M_1}
\,ds\notag\\
&\le C_N\int_0^t
\bigl(\|g_s^N\|_{\mathcal M_1}
+\|\widetilde g_s^N\|_{\mathcal M_1}\bigr)^2
\|g_s^N-\widetilde g_s^N\|_{\mathcal M_1}\,ds\notag\\
&=4C_N(\mathcal N_0+\mathcal E_0)^2
\int_0^t\|g_s^N-\widetilde g_s^N\|_{\mathcal M_1}\,ds.
\end{align}
Gronwall's inequality gives
$\|g_t^N-\widetilde g_t^N\|_{\mathcal M_1}=0$ for every $t\ge0$,
which proves uniqueness.
\end{proof}

\subsection{Symmetrization and estimates for test functions}
We now state the symmetrization and estimates used to pass to the limit. Their proofs involve technical, elementary estimates and are
given in \cref{app:measure-estimates}.

Set
$\mathcal D_{\mathrm{ord}}=\{(x,y,z)\in[0,\infty)^3:x\le y\le z\}$,
and define
$\operatorname{sort}:[0,\infty)^3\to\mathcal D_{\mathrm{ord}}$
by
\begin{equation}
(x,y,z)=\operatorname{sort}(\omega_1,\omega_2,\omega_3)
=\left(
\min_{1\le j\le3}\omega_j,\;
\operatorname{med}_{1\leq j\leq 3} \omega_j,\;
\max_{1\le j\le3}\omega_j
\right).
\end{equation}
Each component is continuous, so $\operatorname{sort}$ is
continuous. For $0<x\le y\le z$, choosing $z$, $y$, or $x$
as the third argument of the collision kernel gives the fourth
frequency $x+y-z$, $x+z-y$, or $y+z-x$, respectively.
For a Borel function $\phi:[0,\infty)\to\mathbb R$, define
\begin{align}
\Delta_1\phi
&=\phi(z)+\phi(x+y-z)-\phi(x)-\phi(y),\qquad x+y>z,\\
\Delta_2\phi
&=\phi(y)+\phi(x+z-y)-\phi(x)-\phi(z),\\
\Delta_3\phi
&=\phi(x)+\phi(y+z-x)-\phi(y)-\phi(z).
\end{align}
We set $\Delta_1\phi=0$ when $x+y\le z$.
By the symmetry of $S_{\beta,N}$, define
\begin{align}
G_{1,N}(x,y,z)
&=K_{\beta,N}(x,y,z)
=\frac{S_{\beta,N}(x,y,z,x+y-z)}{\sqrt{xyz}},\\
G_{2,N}(x,y,z)
&=K_{\beta,N}(x,z,y)
=\frac{S_{\beta,N}(x,y,z,x+z-y)}{\sqrt{xyz}},\\
G_{3,N}(x,y,z)
&=K_{\beta,N}(y,z,x)
=\frac{S_{\beta,N}(x,y,z,y+z-x)}{\sqrt{xyz}}.
\end{align}
For \(K_\beta\), the same permutations give
\begin{align}
G_1(x,y,z)=K_\beta(x,y,z)&=(xyz)^{\beta-1/2}(x+y-z)_+^{\beta+1/2},\\
G_2(x,y,z)=K_\beta(x,z,y)&=x^\beta(yz)^{\beta-1/2}(x+z-y)^\beta,\\
G_3(x,y,z)=K_\beta(y,z,x)&=x^\beta(yz)^{\beta-1/2}(y+z-x)^\beta.
\end{align}
Indeed, when $x+y>z$, the smallest frequency in the first
quartet is $x+y-z$. In the other two quartets, the smallest
frequency is $x$. Define
\begin{equation}
F_{\beta,N}[\phi]=\sum_{i=1}^3G_{i,N}\Delta_i\phi,
\qquad
F_\beta[\phi]=\sum_{i=1}^3G_i\Delta_i\phi
\end{equation}
on $\mathcal D_{\mathrm{ord}}$ with
all coefficients set to zero on $\{x=0\}$.

\begin{lemma}\label{lem:ordered}
Let $N\ge1$. The following statements hold.
\begin{enumerate}[label=\textup{(\roman*)},leftmargin=*]
\item \textit{Symmetrization.}
Let $g^N$ be the solution constructed in
\cref{prop:integrated-cutoff-existence}.
For every Borel function $\phi:[0,\infty)\to\mathbb R$
satisfying $|\phi(\omega)|\le C_\phi(1+\omega)$,
\begin{equation}\label{eq:ordered}
\dfrac{d}{dt}\int_{[0,\infty)}\phi\,dg_t^N
=\frac1{12}\int_{\mathcal D_{\mathrm{ord}}}
F_{\beta,N}[\phi](x,y,z)\,
d\bigl(\operatorname{sort}_\#(g_t^N)^{\otimes3}\bigr),
\qquad t\ge0.
\end{equation}
For every nonnegative $g\in\mathcal M_1$ and every Borel
function $\phi:[0,\infty)\to\mathbb R$ satisfying
$\iiint_{[0,\infty)^3}
K_\beta|\Delta\phi|\,dg^{\otimes3}<\infty$, we have
\begin{equation}
\frac14\iiint_{[0,\infty)^3}
K_\beta\Delta\phi\,dg^{\otimes3}
=\frac1{12}\int_{\mathcal D_{\mathrm{ord}}}
F_\beta[\phi](x,y,z)\,
d\bigl(\operatorname{sort}_\#g^{\otimes3}\bigr).
\end{equation}

\item \textit{Convex test functions.}
For every convex function $\phi:[0,\infty)\to\mathbb R$,
\begin{equation}\label{eq:convex}
\begin{aligned}
0\le F_{\beta,N}[\phi]
&\le G_{1,N}\Delta_1\phi+G_{3,N}\Delta_3\phi,\\
0\le F_\beta[\phi]
&\le G_1\Delta_1\phi+G_3\Delta_3\phi
\end{aligned}
\qquad\text{on }\mathcal D_{\mathrm{ord}}.
\end{equation}
\end{enumerate}
\end{lemma}

\begin{lemma}\label{lem:compact}
Let $\phi\in\operatorname{Lip}_c([0,\infty))$, and set
$p=0 \vee (2\beta-1/2)$, $q=0 \vee (3\beta-1)/2$.
Then $0\le q\le p\le1$ and $q\le5/8$.
For every $(x,y,z)\in\mathcal D_{\mathrm{ord}}$ and $N\ge1$,
\begin{align}
\sum_{i=1}^3G_{i,N}|\Delta_i\phi|
&\le C_{\phi,\beta}(1+x)^p(1+y)^p(1+z)^p,
\label{eq:raw-test-bound}\\
|F_{\beta,N}[\phi]|
&\le C_{\phi,\beta}(1+x)^q(1+y)^q(1+z)^q.
\label{eq:compact-envelope}
\end{align}
If $0\le x\le1$, then
\begin{equation}\label{eq:compact-small-frequency}
|F_{\beta,N}[\phi]|
\le C_{\phi,\beta}x^\beta(1+y)^q(1+z)^q.
\end{equation}
The same estimates hold with $G_i$ and $F_\beta$ in place of
$G_{i,N}$ and $F_{\beta,N}$.
\end{lemma}

\begin{corollary}\label{cor:boundary-removal}
Let $\phi\in\operatorname{Lip}_c([0,\infty))$, $p$ and $q$ be as in \cref{lem:compact}.
Then $F_\beta[\phi]$ is continuous on $\mathcal D_{\mathrm{ord}}$,
and $F_{\beta,N}[\phi]\to F_\beta[\phi]$ locally uniformly on $\mathcal D_{\mathrm{ord}}$.
For every nonnegative $g\in\mathcal M_1$ with
$\|g\|_{\mathcal M_1}\le M$,
\begin{equation}\label{eq:absolute-collision-bound}
\iiint_{[0,\infty)^3}K_\beta|\Delta\phi|\,dg^{\otimes3}
\le C_{\phi,\beta}M^3.
\end{equation}
If $\mu=\operatorname{sort}_\#g^{\otimes3}$, then for $R\ge1$,
\begin{equation}\label{eq:compact-high-tail}
\int_{\mathcal D_{\mathrm{ord}}\cap\{z>R\}}
|F_\beta[\phi]|\,d\mu
\le C_{\phi,\beta}M^3(1+R)^{q-1}.
\end{equation}
Both integral estimates hold with the same constants for \(K_{\beta,N}\) and \(F_{\beta,N}[\phi]\).
\end{corollary}

\subsection{Proof of
\texorpdfstring{\cref{thm:measure-main}}
{Theorem \getrefnumber{thm:measure-main}}}

For nonnegative $g\in\mathcal M_1$ and
$\phi\in\operatorname{Lip}_c([0,\infty))$, \cref{cor:boundary-removal,lem:ordered} allow us to define
\begin{equation}\label{eq:collision-functional}
\mathcal I_\beta[\phi;g]
=\frac14\iiint_{[0,\infty)^3}K_\beta\Delta\phi\,dg^{\otimes3}
=\frac1{12}\int_{\mathcal D_{\mathrm{ord}}}F_\beta[\phi]
\,d\bigl(\operatorname{sort}_\#g^{\otimes3}\bigr).
\end{equation}
For the cutoff coefficients, set
\begin{equation}
\mathcal I_{\beta,N}[\phi;g]
=\frac14\iiint_{[0,\infty)^3}K_{\beta,N}\Delta\phi\,dg^{\otimes3}
=\frac1{12}\int_{\mathcal D_{\mathrm{ord}}}F_{\beta,N}[\phi]
\,d\bigl(\operatorname{sort}_\#g^{\otimes3}\bigr)
=\int_{[0,\infty)}\phi\,d\mathcal I_N(g).
\end{equation}
\begin{proof}[Proof of \cref{thm:measure-main}]
If $\mathcal N_0=0$, then $g_0=0$, and $g_t=0$ gives the
required solution. We therefore assume $\mathcal N_0>0$.
Set $M=\mathcal N_0+\mathcal E_0$, and let $g^N$ be the
solutions given by \cref{prop:integrated-cutoff-existence}.
In particular,
\begin{equation}
g_t^N\ge0,\qquad
\int_{[0,\infty)}dg_t^N=\mathcal N_0,\qquad
\int_{[0,\infty)}\omega\,dg_t^N=\mathcal E_0,
\qquad t\ge0.
\end{equation}

\smallskip
\noindent\textit{Step 1. Compactness.}
Define
\begin{equation}
\mathcal K
=\left\{\nu\in\mathcal M_1:
\nu\ge0,\quad
\int_{[0,\infty)}d\nu=\mathcal N_0,\quad
\int_{[0,\infty)}\omega\,d\nu\le\mathcal E_0
\right\}.
\end{equation}
For $\nu\in\mathcal K$ and $R>0$, we have
\begin{equation}
\nu((R,\infty))
=\int_{(R,\infty)}1\,d\nu
\le\frac1R\int_{(R,\infty)}\omega\,d\nu
\le\frac{\mathcal E_0}{R}.
\end{equation}
Given $\delta>0$, choose $R>\mathcal E_0/\delta$.
Then $[0,R]$ is compact and
$\nu([0,\infty)\setminus[0,R])<\delta$ for every
$\nu\in\mathcal K$. Thus $\mathcal K$ is uniformly tight.

We next verify that a narrow limit of measures in
$\mathcal K$ belongs to $\mathcal K$.
Suppose $\nu_n\in\mathcal K$ and $\nu_n$ converges narrowly
to a nonnegative finite measure $\nu$.
Testing the convergence with the constant function $1$ gives
\begin{equation}
\int_{[0,\infty)}d\nu
=\lim_{n\to\infty}\int_{[0,\infty)}d\nu_n
=\mathcal N_0.
\end{equation}
For every $L>0$, since $\omega\mapsto\omega\wedge L$
is bounded and continuous,
\begin{equation}
\int_{[0,\infty)}(\omega\wedge L)\,d\nu
=\lim_{n\to\infty}
\int_{[0,\infty)}(\omega\wedge L)\,d\nu_n
\le\mathcal E_0.
\end{equation}
Since $\omega\wedge L\uparrow\omega$ as $L\to\infty$,
the monotone convergence theorem yields
\begin{equation}
\int_{[0,\infty)}\omega\,d\nu
=\lim_{L\to\infty}
\int_{[0,\infty)}(\omega\wedge L)\,d\nu
\le\mathcal E_0.
\end{equation}
Therefore $\nu\in\mathcal K$.

Now let $(\nu_n)$ be any sequence in $\mathcal K$.
Hence $\nu_n/\mathcal N_0$ are probability measures
on the Polish space $[0,\infty)$, and
\[
\sup_n
\frac{\nu_n((R,\infty))}{\mathcal N_0}
\le\frac{\mathcal E_0}{\mathcal N_0R}
\longrightarrow0
\qquad\text{as }R\to\infty.
\]
Prokhorov's theorem then gives a subsequence
$\nu_{n_k}/\mathcal N_0$ converging narrowly to a probability
measure $\pi$.
Thus $\nu_{n_k}$ converges narrowly to
$\nu=\mathcal N_0\pi$, and the preceding argument shows
that $\nu\in\mathcal K$.
Hence every sequence in $\mathcal K$ has a
narrowly convergent subsequence with limit in $\mathcal K$.

\smallskip
\noindent\textit{Step 2. Equicontinuity and extraction of a limit.}
Choose a sequence
$(\varphi_j)_{j\ge1}\subset\operatorname{Lip}_c([0,\infty))$
that is dense in the unit ball of $C_0([0,\infty))$ for
the uniform norm, with $\|\varphi_j\|_\infty\le1$.
For $\nu,\widetilde\nu\in\mathcal K$, define
\begin{equation}
d(\nu,\widetilde\nu)
=\sum_{j=1}^\infty 2^{-j}
\left|
\int_{[0,\infty)}\varphi_j\,d(\nu-\widetilde\nu)
\right|.
\end{equation}
Since each summand is bounded by $2^{1-j}\mathcal N_0$, the
series converges uniformly on $\mathcal K\times\mathcal K$.
If $d(\nu,\widetilde\nu)=0$, the two measures have the same
integrals against every $\varphi_j$.
Density and
\[
\left|
\int_{[0,\infty)}f\,d(\nu-\widetilde\nu)
\right|
\le2\mathcal N_0\|f\|_\infty
\]
then give equality of their integrals against every
$f\in C_0([0,\infty))$.
Hence $\nu=\widetilde\nu$, and $d$ is a metric.

We now show the relation between $d$ and narrow convergence.
If $\nu_n$ converges narrowly to $\nu$ in $\mathcal K$,
then, for every integer $J\ge1$,
\begin{equation}
d(\nu_n,\nu)
\le
\sum_{j=1}^J2^{-j}
\left|\int_{[0,\infty)}\varphi_j\,d(\nu_n-\nu)\right|
+2^{1-J}\mathcal N_0.
\end{equation}
Letting $n\to\infty$ and then $J\to\infty$ gives
$d(\nu_n,\nu)\to0$.

Conversely, suppose $d(\nu_n,\nu)\to0$.
By Step~1, every subsequence of $(\nu_n)$ has a further
subsequence converging narrowly to some
$\widetilde\nu\in\mathcal K$.
Along that subsequence, the preceding implication gives
$d(\nu_n,\widetilde\nu)\to0$.
The triangle inequality yields
$d(\nu,\widetilde\nu)\leq d(\nu,\nu_n)+d(\widetilde \nu,\nu_n)\to0$ and \(d(\nu,\widetilde\nu)=0\), so $\widetilde\nu=\nu$.
Therefore $\nu_n$ converges narrowly to $\nu$.
Hence $d$ metrizes narrow convergence on
$\mathcal K$. Step~1 also shows that $(\mathcal K,d)$
is a compact metric space.

For $\phi\in\operatorname{Lip}_c([0,\infty))$,
\eqref{eq:cutoff-weak} and
\cref{cor:boundary-removal} give
\begin{equation}\label{eq:compact-time-bound}
\begin{aligned}
\left|
\int_{[0,\infty)}\phi\,dg_t^N
-\int_{[0,\infty)}\phi\,dg_s^N
\right|
&\le
\frac14\int_s^t\iiint_{[0,\infty)^3}
K_{\beta,N}|\Delta\phi|\,
d(g_\tau^N)^{\otimes3}\,d\tau\\
&\le C_{\phi,\beta}M^3(t-s),
\qquad 0\le s\le t.
\end{aligned}
\end{equation}
Consequently, for every $J\ge1$ and $s,t\ge0$,
\begin{equation}
d(g_t^N,g_s^N)
\le
M^3|t-s|
\sum_{j=1}^J2^{-j}C_{\varphi_j,\beta}
+2^{1-J}\mathcal N_0.
\end{equation}
Given $\delta>0$, choose $J$ such that
$2^{1-J}\mathcal N_0<\delta/2$.
Then choose $h>0$ such that
\[
M^3h\sum_{j=1}^J2^{-j}C_{\varphi_j,\beta}<\delta/2.
\]
These choices give
$d(g_t^N,g_s^N)<\delta$ whenever $|t-s|<h$,
uniformly in $N$.
Thus the curves $g^N$ are equicontinuous in
$(\mathcal K,d)$.

Since $g_t^N\in\mathcal K$, Arzel\`a--Ascoli gives a
uniformly convergent subsequence on each interval $[0,k]$.
A diagonal extraction over the integers $k\ge1$ produces
a single subsequence, still indexed by $N$, and a continuous
curve $g:[0,\infty)\to(\mathcal K,d)$ such that
\begin{equation}
\sup_{0\le t\le T}d(g_t^N,g_t)\longrightarrow0,
\qquad T<\infty.
\end{equation}
For every $f\in C_b([0,\infty))$,
$\nu\mapsto\int f\,d\nu$ is continuous on
$(\mathcal K,d)$ because $d$ metrizes narrow convergence.
It is uniformly continuous since $\mathcal K$ is compact.
Therefore
\begin{equation}\label{eq:uniform-narrow-convergence}
\sup_{0\le t\le T}
\left|
\int_{[0,\infty)}f\,d(g_t^N-g_t)
\right|
\longrightarrow0,
\qquad f\in C_b([0,\infty)),\quad T<\infty.
\end{equation}
Hence $g$ is narrowly continuous.
Moreover, $g_0^N=g_0$ and $g_t\in\mathcal K$ imply
\begin{equation}
g_{t=0}=g_0,\qquad
g_t\ge0,\qquad
\mathcal N(t)=\mathcal N_0,\qquad
\mathcal E(t)\le\mathcal E_0,\qquad
\|g_t\|_{\mathcal M_1}\le M.
\end{equation}

\smallskip
\noindent\textit{Step 3. Passage to the limit in the weak equation.}
We first establish convergence of the product measures
against compactly supported continuous functions.
Fix $t\ge0$.
For $f_1,f_2,f_3\in C_c([0,\infty))$,
\eqref{eq:uniform-narrow-convergence} gives
\begin{equation}
\iiint_{[0,\infty)^3}
\prod_{j=1}^3f_j(\omega_j)\,d(g_t^N)^{\otimes3}
=
\prod_{j=1}^3\int_{[0,\infty)}f_j\,dg_t^N
\longrightarrow
\prod_{j=1}^3\int_{[0,\infty)}f_j\,dg_t.
\end{equation}
The same convergence holds for finite sums of such products.

Let $H\in C_c([0,\infty)^3)$ and $\delta>0$.
Note that finite sums of products $\prod_{j=1}^3f_j(\omega_j)$ with $f_j\in C_c([0,\infty))$, form a subalgebra of
$C_0([0,\infty)^3)$, which separates points, and
at each point some element of the algebra is nonzero.
By the Stone--Weierstrass theorem for locally compact spaces, there are
$f_{\ell,j}\in C_c([0,\infty))$ such that
\begin{equation}
P(\omega_1,\omega_2,\omega_3)
=\sum_{\ell=1}^m\prod_{j=1}^3f_{\ell,j}(\omega_j),
\qquad
\|H-P\|_\infty<\delta.
\end{equation}
Hence
\begin{align*}
\left|
\iiint_{[0,\infty)^3}
H\,d\bigl((g_t^N)^{\otimes3}-g_t^{\otimes3}\bigr)
\right|\le2\mathcal N_0^3\delta+
\left|
\iiint_{[0,\infty)^3}
P\,d\bigl((g_t^N)^{\otimes3}-g_t^{\otimes3}\bigr)
\right|.
\end{align*}
Letting $N\to\infty$ and then $\delta\downarrow0$ proves
\begin{equation}
\iiint_{[0,\infty)^3}H\,d(g_t^N)^{\otimes3}
\longrightarrow
\iiint_{[0,\infty)^3}H\,dg_t^{\otimes3},
\qquad H\in C_c([0,\infty)^3).
\end{equation}

Now fix $\phi\in\operatorname{Lip}_c([0,\infty))$ and set
\begin{equation}
q=0 \vee (3\beta-1)/2,\qquad
\mu_t^N=\operatorname{sort}_\#(g_t^N)^{\otimes3},
\qquad
\mu_t=\operatorname{sort}_\#g_t^{\otimes3}.
\end{equation}
For $R\ge1$, define
$\eta_R(\omega)=1 \wedge (2-\omega/R)_+$.
Then $\eta_R=1$ on $[0,R]$, $\eta_R=0$ on $[2R,\infty)$,
and $0\le\eta_R\le1$.

Since $\operatorname{sort}$ is continuous, by \cref{cor:boundary-removal}, 
$(\eta_R(z)F_\beta[\phi])\circ\operatorname{sort}$
is continuous on $[0,\infty)^3$.
It vanishes when $\max_j\omega_j\ge2R$, so it belongs to
$C_c([0,\infty)^3)$. Letting \(H=(\eta_R(z)F_\beta[\phi])\circ\operatorname{sort}\) gives
\begin{equation}
\int_{\mathcal D_{\mathrm{ord}}}
\eta_R(z)F_\beta[\phi]\,d\mu_t^N
\longrightarrow
\int_{\mathcal D_{\mathrm{ord}}}
\eta_R(z)F_\beta[\phi]\,d\mu_t.
\end{equation}
Using also the local uniform convergence in
\cref{cor:boundary-removal}, we obtain
\begin{align}
&\left|
\int_{\mathcal D_{\mathrm{ord}}}
\eta_R(z)F_{\beta,N}[\phi]\,d\mu_t^N
-\int_{\mathcal D_{\mathrm{ord}}}
\eta_R(z)F_\beta[\phi]\,d\mu_t
\right|\notag\\
&\quad\le
\mathcal N_0^3
\sup_{0\le x\le y\le z\le2R}
|F_{\beta,N}[\phi]-F_\beta[\phi]|+
\left|
\int_{\mathcal D_{\mathrm{ord}}}
\eta_R(z)F_\beta[\phi]\,d(\mu_t^N-\mu_t)
\right|
\longrightarrow0.
\end{align}

To remove $\eta_R$, use
$0\le1-\eta_R(z)\le\mathbf1_{\{z>R\}}$.
Since $\|g_t^N\|_{\mathcal M_1},\|g_t\|_{\mathcal M_1}\le M$,
\eqref{eq:compact-high-tail} gives
\begin{align}
\int_{\mathcal D_{\mathrm{ord}}}
(1-\eta_R(z))|F_{\beta,N}[\phi]|\,d\mu_t^N
+\int_{\mathcal D_{\mathrm{ord}}}
(1-\eta_R(z))|F_\beta[\phi]|\,d\mu_t\le C_{\phi,\beta}M^3(1+R)^{q-1}.
\end{align}
Thus
\begin{align*}
\left|
\mathcal I_{\beta,N}[\phi;g_t^N]
-\mathcal I_\beta[\phi;g_t]
\right|
&\le\frac1{12}\left|
\int_{\mathcal D_{\mathrm{ord}}}
\eta_R(z)F_{\beta,N}[\phi]\,d\mu_t^N
-\int_{\mathcal D_{\mathrm{ord}}}
\eta_R(z)F_\beta[\phi]\,d\mu_t
\right|\\
&\quad+C_{\phi,\beta}M^3(1+R)^{q-1}.
\end{align*}
For fixed $R$, letting $N\to\infty$ gives
\begin{equation}
\limsup_{N\to\infty}
\left|
\mathcal I_{\beta,N}[\phi;g_t^N]
-\mathcal I_\beta[\phi;g_t]
\right|
\le C_{\phi,\beta}M^3(1+R)^{q-1}.
\end{equation}
Since $q\le5/8<1$, the right-hand side tends to zero as
$R\to\infty$. Therefore
\begin{equation}
\mathcal I_{\beta,N}[\phi;g_t^N]
\longrightarrow\mathcal I_\beta[\phi;g_t],
\qquad t\ge0.
\end{equation}

For each $N$, since
$t\mapsto\mathcal I_{\beta,N}[\phi;g_t^N]$
is continuous by
\cref{prop:integrated-cutoff-existence},
its pointwise limit
$t\mapsto\mathcal I_\beta[\phi;g_t]$ is therefore measurable.
Moreover, \eqref{eq:absolute-collision-bound} and its cutoff
version imply
\begin{equation}
|\mathcal I_{\beta,N}[\phi;g_t^N]|
+|\mathcal I_\beta[\phi;g_t]|
\le C_{\phi,\beta}M^3,
\qquad t\ge0,\quad N\ge1.
\end{equation}
For every $T<\infty$, dominated convergence gives
\begin{equation}
\int_0^T
\left|
\mathcal I_{\beta,N}[\phi;g_\tau^N]
-\mathcal I_\beta[\phi;g_\tau]
\right|\,d\tau
\longrightarrow0.
\end{equation}
Hence, for $0\le s\le t\le T$,
\begin{equation}
\int_s^t\mathcal I_{\beta,N}[\phi;g_\tau^N]\,d\tau
\longrightarrow
\int_s^t\mathcal I_\beta[\phi;g_\tau]\,d\tau.
\end{equation}
Passing to the limit in \eqref{eq:cutoff-weak} and using
\eqref{eq:uniform-narrow-convergence}, we obtain
\begin{equation}\label{eq:integrated-paired-weak}
\begin{aligned}
\int_{[0,\infty)}\phi\,dg_t
-\int_{[0,\infty)}\phi\,dg_s=\int_s^t\mathcal I_\beta[\phi;g_\tau]\,d\tau=\frac14\int_s^t\iiint_{[0,\infty)^3}
K_\beta\Delta\phi\,dg_\tau^{\otimes3}\,d\tau.
\end{aligned}
\end{equation}

Finally, since $K_\beta|\Delta\phi|$ is a nonnegative Borel
function, narrow continuity of $g$ proves measurability of
$\tau\mapsto\iiint K_\beta|\Delta\phi|\,dg_\tau^{\otimes3}$.
Applying \eqref{eq:absolute-collision-bound} at each time gives
\begin{equation}
\int_0^T\iiint_{[0,\infty)^3}
K_\beta|\Delta\phi|\,dg_\tau^{\otimes3}\,d\tau
\le
C_{\phi,\beta}\int_0^T
\|g_\tau\|_{\mathcal M_1}^3\,d\tau
\le C_{\phi,\beta}TM^3<\infty.
\end{equation}
Together with Step~2, this verifies all the requirements of
\cref{def:weak-form}.

\smallskip
\noindent\textit{Step 4. Monotonicity of the energy.}
For $R>0$, since
$\omega\mapsto(R-\omega)_+=(R-\omega) \vee 0$
is convex and belongs to
$\operatorname{Lip}_c([0,\infty))$,
\cref{lem:ordered} gives
$F_\beta[(R-\cdot)_+]\ge0$.
Applying \eqref{eq:integrated-paired-weak} and
\cref{lem:ordered}, we obtain
\begin{equation}
\int_{[0,\infty)}(R-\omega)_+\,dg_t
-\int_{[0,\infty)}(R-\omega)_+\,dg_s
=
\frac1{12}\int_s^t
\int_{\mathcal D_{\mathrm{ord}}}
F_\beta[(R-\cdot)_+]\,d\mu_\tau\,d\tau
\ge0.
\end{equation}
Since $\omega\wedge R=R-(R-\omega)_+$ and
$\mathcal N(t)=\mathcal N(s)=\mathcal N_0$,
\begin{align}
\int_{[0,\infty)}(\omega\wedge R)\,dg_t
&=R\mathcal N_0
-\int_{[0,\infty)}(R-\omega)_+\,dg_t\notag\\
&\le R\mathcal N_0
-\int_{[0,\infty)}(R-\omega)_+\,dg_s\notag\\
&=\int_{[0,\infty)}(\omega\wedge R)\,dg_s.
\end{align}
For each fixed time, $\omega\wedge R\uparrow\omega$ as
$R\to\infty$. Applying monotone convergence to both integrals
gives
$\mathcal E(t)\le\mathcal E(s)$ for $0\le s\le t$.
Since $g_{t=0}=g_0$, we conclude that
$\mathcal E(t)\le\mathcal E(s)\le\mathcal E_0$
for every $0\le s\le t<\infty$.
The same argument applies to every solution constructed in this section.
\end{proof}

\section{Absence of condensation}\label{sec:action-threshold}

In this section, we prove the non-condensation assertion in
\cref{thm:action}.

\begin{lemma}\label{lem:critical-kernel}
Let $1/4\le\beta\le3/4$. For every
$\omega_1,\omega_2,\omega_3\ge0$,
\begin{equation}\label{eq:noncondensation-kernel-bound}
0\le K_{1/4}(\omega_1,\omega_2,\omega_3)\le1,
\qquad
0\le K_\beta(\omega_1,\omega_2,\omega_3)
\le\prod_{j=1}^3(1+\omega_j).
\end{equation}
\end{lemma}
\begin{proof}
If one of the four frequencies is not positive, \eqref{eq:noncondensation-kernel-bound} is direct.  We therefore assume that all four frequencies
are positive and set
$(x,y,z)=\operatorname{sort}(\omega_1,\omega_2,\omega_3)$.

From \cref{sec:measure}, the three possible values of $K_{1/4}$ are $G_1$, $G_2$ and $G_3$, evaluated at $\beta=1/4$.
Hence
\begin{align}
G_1^4
&=\frac{(x+y-z)^3}{xyz}
\le\frac{x^3}{xyz}\le1,
\qquad x+y>z,\\
G_2^4
&=\frac{x(x+z-y)}{yz}
\le\frac{x(y+z-x)}{yz}
=G_3^4
=1-\frac{(y-x)(z-x)}{yz}\le1.
\end{align}
When $x+y\le z$, we have $G_1=0$. Since \(G_1,G_2,G_3\geq 0\), $K_{1/4}\le1$.

Set $a=\beta-1/4\in [0,1/2]$.
By the definition,
$K_\beta=K_{1/4}(\omega_1\omega_2\omega_3\omega_4)^a$.
The resonance relation gives
$\omega_4\le\omega_1+\omega_2
\le(1+\omega_1)(1+\omega_2)$.
Consequently,
\begin{align}
K_\beta\le(\omega_1\omega_2\omega_3\omega_4)^a\le(1+\omega_1)^{2a}(1+\omega_2)^{2a}
      (1+\omega_3)^a\le\prod_{j=1}^3(1+\omega_j),
\end{align}
where the last inequality uses $2a\le1$ and $a\le1$.
This proves \eqref{eq:noncondensation-kernel-bound}.
\end{proof}

\begin{proposition}\label{thm:action-preservation}
Let $1/4\le\beta\le3/4$ and let $g_0\in\mathcal M_1$ be
nonnegative. Every isotropic measure solution with initial
measure $g_0$ satisfies $g_t(\{0\})=g_0(\{0\})$ for all
$t\ge0$.
Moreover, for every $T<\infty$, the cutoff solutions with
initial measure $g_0$ in \cref{sec:measure} satisfy
\begin{equation}\label{eq:uniform-positive-near-zero}
\lim_{r\downarrow0}\sup_{N\ge1}\sup_{0\le t\le T}
g_t^N((0,r])=0.
\end{equation}
\end{proposition}
\begin{proof}
\noindent\textit{Step 1. Preservation of the atom at zero.}
Let $g$ be an isotropic measure solution and fix $T<\infty$.
By \cref{def:weak-form},
$M_T:=\sup_{0\le\tau\le T}\|g_\tau\|_{\mathcal M_1}<\infty$.
\cref{lem:critical-kernel} gives, for every
$\tau\in[0,T]$,
\begin{equation}
\iiint_{[0,\infty)^3}K_\beta\,dg_\tau^{\otimes3}
\le\iiint_{[0,\infty)^3}\prod_{j=1}^3(1+\omega_j)
\,dg_\tau^{\otimes3}
=\|g_\tau\|_{\mathcal M_1}^3\le M_T^3.
\end{equation}

For $0<r\le1$, take $\phi_r(\omega)=(1-\omega/r)_+$ in
\eqref{eq:weak}. Then
$\phi_r\in\operatorname{Lip}_c([0,\infty))$,
$0\le\phi_r\le1$, and $|\Delta\phi_r|\le2$. If $\min_{1\le j\le4}\omega_j\le0$, then $K_\beta=0$,
so $K_\beta\Delta\phi_r=0$.
If $\min_{1\le j\le4}\omega_j>0$, then
$0<r<\min_{1\le j\le4}\omega_j$ implies
$\phi_r(\omega_j)=(1-\omega_j/r)_+=0$ for $j=1,2,3,4$.
Thus $\Delta\phi_r=0$ for every such $r$.
Consequently,
\begin{equation}
\lim_{r\downarrow0}K_\beta\Delta\phi_r=0,
\qquad
(\omega_1,\omega_2,\omega_3)\in[0,\infty)^3.
\end{equation}

For $0\le s\le t\le T$, we have
$|K_\beta\Delta\phi_r|\le2K_\beta$ and
\begin{equation}
\int_s^t\iiint_{[0,\infty)^3}
2K_\beta\,dg_\tau^{\otimes3}\,d\tau
\le2\int_s^t\|g_\tau\|_{\mathcal M_1}^3\,d\tau
\le2(t-s)M_T^3<\infty.
\end{equation}
For each fixed $\tau\in[s,t]$, since
$|K_\beta\Delta\phi_r|\le2K_\beta$,  dominated convergence in the three
frequencies implies
$\displaystyle\iiint K_\beta\Delta\phi_r\,dg_\tau^{\otimes3}
\to0$ as $r\downarrow0$.
Since the absolute value of this integral is bounded by
$2M_T^3$, uniformly in $r$ and $\tau$,
dominated convergence in time then gives
\begin{equation}
\lim_{r\downarrow0}
\int_s^t\iiint_{[0,\infty)^3}
K_\beta\Delta\phi_r\,dg_\tau^{\otimes3}\,d\tau=0.
\end{equation}

Moreover,
$\lim_{r\downarrow0}\phi_r(\omega)
=\mathbf1_{\{0\}}(\omega)$ for every $\omega\ge0$.
Since
$|g_t-g_s|([0,\infty))
\le\|g_t\|_{\mathcal M_1}+\|g_s\|_{\mathcal M_1}
\le2M_T$,
dominated convergence also gives
\begin{equation}
\lim_{r\downarrow0}
\int_{[0,\infty)}\phi_r\,d(g_t-g_s)
=g_t(\{0\})-g_s(\{0\}).
\end{equation}
Passing to the limit in \eqref{eq:weak}, we obtain
$g_t(\{0\})-g_s(\{0\})=0$ for $0\le s\le t\le T$.
For any $t\ge0$, choosing $T\ge t$ and $s=0$ gives
$g_t(\{0\})=g_0(\{0\})$.

\smallskip
\noindent\textit{Step 2. Uniform control near zero.}
Fix $T<\infty$ and write $a_0=g_0(\{0\})$. By contradiction and
suppose that \eqref{eq:uniform-positive-near-zero} fails.
Then there are $\varepsilon>0$, a sequence $r_k\downarrow0$,
integers $N_k\ge1$, and times $t_k\in[0,T]$ such that
\begin{equation}\label{eq:noncondensation-contradiction}
g_{t_k}^{N_k}((0,r_k])\ge\varepsilon
\qquad\text{for every }k.
\end{equation}
After passing to a subsequence, we may assume that
$t_k\to t_*$ for some $t_*\in[0,T]$.
We now show that a further subsequence of $g_{t_k}^{N_k}$
converges narrowly to a measure $\nu$ with
$\nu(\{0\})=a_0$.

If $(N_k)$ has a bounded subsequence, we may take a further
subsequence on which $N_k=N_*$ is constant.
The continuity of $g^{N_*}$ in $\mathcal M_1$ gives
$g_{t_k}^{N_*}\to g_{t_*}^{N_*}$ in $\mathcal M_1$ and
hence narrowly. In this case, take $\nu=g_{t_*}^{N_*}$.
\cref{prop:integrated-cutoff-existence} gives
$\nu(\{0\})=a_0$.

Otherwise, $N_k\to\infty$. Applying Steps~1--3
in the proof of \cref{thm:measure-main} to the sequence $(g^{N_k})$, and extracting
a diagonal subsequence on bounded time intervals, we obtain
a narrowly continuous limit $g$ with initial measure $g_0$.
The passage to the limit in Step~3 shows that $g$ is an
isotropic measure solution. Along this subsequence,
\eqref{eq:uniform-narrow-convergence} holds, so for every
$f\in C_b([0,\infty))$,
\begin{equation}
\sup_{0\le t\le T}
\left|\int_{[0,\infty)}f\,d(g_t^{N_k}-g_t)\right|
\longrightarrow0.
\end{equation}
Since $g$ is narrowly continuous and $t_k\to t_*$,
\begin{align}
\left|\int_{[0,\infty)}f\,d(g_{t_k}^{N_k}-g_{t_*})\right|\le\sup_{0\le t\le T}
\left|\int_{[0,\infty)}f\,d(g_t^{N_k}-g_t)\right|+
\left|\int_{[0,\infty)}f\,d(g_{t_k}-g_{t_*})\right|
\longrightarrow0.
\end{align}
Thus $g_{t_k}^{N_k}$ converges narrowly to
$\nu=g_{t_*}$. By Step~1, $\nu(\{0\})=a_0$.

In either case, fix $r>0$.
For all sufficiently large $k$, we have $r_k\le r$.
Since the cutoff solutions preserve the atom $a_0$ at zero,
\eqref{eq:noncondensation-contradiction} implies
$g_{t_k}^{N_k}([0,r])\ge a_0+\varepsilon$.
Since $[0,r]$ is closed, the Portmanteau theorem gives
\begin{equation}
\nu([0,r])
\ge\limsup_{k\to\infty}g_{t_k}^{N_k}([0,r])
\ge a_0+\varepsilon.
\end{equation}
The sets $[0,1/m]$ decrease to $\{0\}$ as
$m\to\infty$, and $\nu([0,1])<\infty$.
Continuity for finite measures therefore gives
$\displaystyle\nu(\{0\})
=\lim_{m\to\infty}\nu([0,1/m])\ge a_0+\varepsilon$,
which contradicts
$\nu(\{0\})=a_0$.
\end{proof}

\section{Density estimates}\label{sec:density-estimates}
Throughout this section, $0<\beta<1/4$, $d=1-4\beta>0$, and
$g_0\in\mathcal M_1$ is nonnegative with $\mathcal N_0=1$.
Let $g^N$ be the cutoff solutions from
\cref{prop:integrated-cutoff-existence}; thus $g_t^N([0,\infty))=1$.
\subsection{Kernel estimates}\label{subsec:cond-kernel}

We use $(x,y,z)=\operatorname{sort}(\omega_1,\omega_2,\omega_3)$
as in \cref{sec:measure}. For $0<x\le y\le z$, recall that
\begin{align}
 G_1(x,y,z)=K_\beta(x,y,z)
 &=(xyz)^{\beta-1/2}(x+y-z)_+^{\beta+1/2},
 \label{eq:cond-G1}\\
 G_2(x,y,z)=K_\beta(x,z,y)
 &=x^\beta(yz)^{\beta-1/2}(x+z-y)^\beta,
 \label{eq:cond-G2}\\
 G_3(x,y,z)=K_\beta(y,z,x)
 &=x^\beta(yz)^{\beta-1/2}(y+z-x)^\beta.
 \label{eq:cond-G3}
\end{align}

\begin{lemma}\label{lem:cond-kernel-upper}
\begin{enumerate}[label=\textup{(\roman*)},leftmargin=*]
\item Let $\omega_1,\omega_2,\omega_3>0$ and
$y=\operatorname{med}\{\omega_1,\omega_2,\omega_3\}$.
For every $N\ge1$,
\begin{equation}\label{eq:condensation-kernel-bound}
 0\le K_{\beta,N}(\omega_1,\omega_2,\omega_3)
 \le K_\beta(\omega_1,\omega_2,\omega_3)
 \le2^\beta y^{-d}.
\end{equation}

\item Suppose that $\omega_1,\omega_2,\omega_3>0$ and
$\omega_4=\omega_1+\omega_2-\omega_3>0$.
For distinct $i,j\in\{1,2,3\}$,
\begin{equation}\label{eq:cond-negative-moment-kernel}
 \begin{aligned}
 K_\beta(\omega_1,\omega_2,\omega_3)
 &\le2^\beta(\omega_i\omega_j)^{-d/2},\\
 K_\beta(\omega_1,\omega_2,\omega_3)\omega_4^{-d/2}
 &\le2(\omega_1\omega_2\omega_3)^{-d/2}.
 \end{aligned}
\end{equation}

\item Let $\omega,\omega_1,\omega_2>0$ satisfy
$\omega_1 \vee \omega_2\ge\omega/2$. Then
\begin{equation}\label{eq:cond-potential-negative-kernel}
 K_\beta(\omega,\omega_1,\omega_2)
 +K_\beta(\omega,\omega_2,\omega_1)
 \le2^{3/2}\omega^{-d}\min\{1,\omega_1/\omega,\omega_2/\omega\}^\beta.
\end{equation}
\end{enumerate}
Estimates \eqref{eq:cond-negative-moment-kernel} and
\eqref{eq:cond-potential-negative-kernel} also hold with
$K_{\beta,N}$ in place of $K_\beta$, for every $N\ge1$.
\end{lemma}

\begin{lemma}\label{lem:kernel}
Set $L=[2(2^\beta+1)]^{1/\beta}>2$.
\begin{enumerate}[label=\textup{(\roman*)},leftmargin=*]
\item If $0<2x\le y\le z$, then
\begin{equation}\label{eq:cond-kernel-gain-difference}
 G_3(x,y,z)-G_2(x,y,z)
 \ge \beta2^{\beta-1}x^\beta y^{\beta+1/2}z^{2\beta-3/2}.
\end{equation}

\item If $0<Lx\le y\le z$, then
\begin{equation}\label{eq:cond-high-potential}
 G_3(x,y,z)-G_2(x,y,z)-G_1(x,y,z)
 \ge \beta2^{\beta-1}x^\beta y^{\beta+1/2}z^{2\beta-3/2}.
\end{equation}
\end{enumerate}
\end{lemma}

\begin{lemma}\label{lem:cond-kernel-small-inputs}
Let $0<x\le y<z/2$.
\begin{enumerate}[label=\textup{(\roman*)},leftmargin=*]
\item For every $N\ge1$,
\begin{equation}\label{eq:cond-kernel-pairing}
 \begin{aligned}
 \sqrt z\,K_{\beta,N}(z,y,x)
 &=\sqrt{z+y-x}\,K_{\beta,N}(z+y-x,x,y),\\
 \sqrt z\,K_{\beta,N}(z,x,y)
 &=\sqrt{z+x-y}\,K_{\beta,N}(z+x-y,y,x).
 \end{aligned}
\end{equation}

\item If $N\ge2/z \vee (3z/2)^{4\beta+1/2}$, then
\begin{equation}\label{eq:cond-kernel-small-inputs}
 \begin{aligned}
 K_{\beta,N}(z,y,x)
 &=z^{2\beta-1/2}\left(1+\frac{y-x}{z}\right)^\beta
 x^\beta y^{\beta-1/2}\ell_N(x)\ell_N(y),\\
 K_{\beta,N}(z,x,y)
 &=z^{2\beta-1/2}\left(1-\frac{y-x}{z}\right)^\beta
 x^\beta y^{\beta-1/2}\ell_N(x)\ell_N(y).
 \end{aligned}
\end{equation}

\item
\begin{align}
 x^\beta y^{\beta-1/2}(y-x)
 &\le x^\beta y^{\beta+1/2}\le z^{2\beta+1/2},
 \label{eq:cond-kernel-first-difference}\\
 x^\beta y^{\beta-1/2}(y-x)^2
 &\le x^\beta y^{\beta+3/2}\le z^{2\beta+3/2}.
 \label{eq:cond-kernel-second-difference}
\end{align}
\end{enumerate}
\end{lemma}

The proofs of \cref{lem:cond-kernel-upper,lem:kernel,lem:cond-kernel-small-inputs}
are given in \cref{app:measure-estimates}.

\subsection{Density inequalities}
\label{subsec:cond-density-inequalities}

Define $f_N$ by
\begin{equation}\label{eq:cond-density-definition}
 (g_t^N)_{\mathrm{ac}}(d\omega)
 =\sqrt\omega\,f_N(t,\omega) d\omega.
\end{equation}
Since $\mu\mapsto\mu_{\mathrm{ac}}$ is linear and satisfies
$\|\mu_{\mathrm{ac}}\|_{\mathcal M_1}\le\|\mu\|_{\mathcal M_1}$, \cref{prop:integrated-cutoff-existence} gives
\begin{equation}\label{eq:cond-density-time-regularity}
 f_N\in C^1([0,\infty);L^1((0,\infty),\sqrt\omega d\omega)).
\end{equation}

In this and the next section, we extend $f_N(t,\cdot)$ and
$u$, $\phi$ introduced below by zero from $(0,\infty)$ to $\mathbb R$,
without changing notation. Inequalities involving these functions are
understood almost everywhere in $\omega$.

Consider \((\omega_1,\omega_2)\to(\omega_3,\omega_4)\), where
\(\omega_1+\omega_2=\omega_3+\omega_4\). For
$u\in L^1((0,\infty),\sqrt{\omega_4} d\omega_4)$, define
\begin{equation}\label{eq:cond-jump-operator}
 L_N(t)[u](\omega_4)=\frac12\iint_{(0,\omega_4/2)^2}
 \frac{S_{\beta,N}(\omega_1,\omega_2,\omega_3,\omega_4)}
      {\sqrt{\omega_2\omega_3\omega_4}}
 [u(\omega_1)-u(\omega_4)]
  g_t^N(d\omega_2)g_t^N(d\omega_3).
\end{equation}
Also define
\begin{equation}\label{eq:cond-potential}
 \begin{aligned}
 V_N(t,\omega_4)
 ={}&\frac14\iint_{(0,\infty)^2}
 \frac{S_{\beta,N}(\omega_1,\omega_2,\omega_3,\omega_4)}
      {\sqrt{\omega_1\omega_2\omega_4}}
  g_t^N(d\omega_1)g_t^N(d\omega_2)\\
 &-\frac12\iint_{(0,\infty)^2\setminus(0,\omega_4/2)^2}
 \frac{S_{\beta,N}(\omega_1,\omega_2,\omega_3,\omega_4)}
      {\sqrt{\omega_2\omega_3\omega_4}}
  g_t^N(d\omega_2)g_t^N(d\omega_3),
 \end{aligned}
\end{equation}
and
\[
 \mathcal G_N(t,\omega_4)
 :=\frac1{4\sqrt{\omega_4}}
 \iint_{(\omega_4/2,\infty)^2}
 S_{\beta,N}(\omega_1,\omega_2,\omega_3,\omega_4)
 f_N(t,\omega_1)f_N(t,\omega_2)f_N(t,\omega_3)
 d\omega_1d\omega_2\ge0.
\]
\begin{lemma}\label{lem:cond-density-comparison}
For every $N\ge1$, $t\ge0$, and a.e.\ $\omega_4>0$,
\begin{equation}\label{eq:cond-output-gain}
 \partial_t f_N(t,\omega_4)
 \ge L_N(t)[f_N(t,\cdot)](\omega_4)
 +V_N(t,\omega_4)f_N(t,\omega_4)
 +\mathcal G_N(t,\omega_4).
\end{equation}
For every $\omega_4>0$,
\(
 V_N(t,\omega_4)\ge-\omega_4^{-d}\ell_N(\omega_4).
\)
Consequently, for a.e.\ $\omega_4>0$,
\begin{equation}\label{eq:cond-density-forced}
 \partial_t f_N(t,\omega_4)
 \ge L_N(t)[f_N(t,\cdot)](\omega_4)
 -\omega_4^{-d}\ell_N(\omega_4)f_N(t,\omega_4)
 +\mathcal G_N(t,\omega_4).
\end{equation}
In particular,
\begin{equation}\label{eq:cond-density-homogeneous}
 \partial_t f_N(t,\omega_4)
 \ge L_N(t)[f_N(t,\cdot)](\omega_4)
 -\omega_4^{-d}\ell_N(\omega_4)f_N(t,\omega_4).
\end{equation}
\end{lemma}
\begin{proof}
Fix $N\ge1$. Since $g^N\in C^1([0,\infty);\mathcal M_1)$ by \cref{prop:integrated-cutoff-existence},
the bounded linear map $\mu\mapsto\mu_{\mathrm{ac}}$ gives
\begin{equation}\label{eq:cond-density-differentiation}
 \begin{aligned}\left\|\frac{(g_{t+h}^N)_{\mathrm{ac}}-(g_t^N)_{\mathrm{ac}}}{h}
       -(\partial_tg_t^N)_{\mathrm{ac}}\right\|_{\mathcal M_1}\le
 \left\|\frac{g_{t+h}^N-g_t^N}{h}
       -\partial_tg_t^N\right\|_{\mathcal M_1}
 \longrightarrow0
 \end{aligned}
\end{equation}
as \(h\to 0\) with \(t+h\geq 0.\) Thus $\partial_t(g_t^N)_{\mathrm{ac}}
=(\partial_tg_t^N)_{\mathrm{ac}}$.

Recall from the proof of \cref{prop:integrated-cutoff-existence}
that $\Lambda_N[g]=K_{\beta,N}(g^{\otimes3}|_{\mathcal D})/4$,
where $\mathcal D=\{(\omega_1,\omega_2,\omega_3)\in[0,\infty)^3:
\omega_1+\omega_2\ge\omega_3\}$ and $\pi_j(\omega_1,\omega_2,\omega_3)=\omega_j$,
for \(j=1,2,3,4\) with $\omega_4=\omega_1+\omega_2-\omega_3$. The symmetry in $\omega_1,\omega_2$ gives
$(\pi_2)_\#\Lambda_N[g_t^N]=(\pi_1)_\#\Lambda_N[g_t^N]$.
By \cref{prop:integrated-cutoff-existence} and
\eqref{eq:cond-density-definition},
\begin{equation}\label{eq:cond-density-measure-derivative}
 \begin{aligned}
 \partial_t f_N(t,\omega_4)
 &=\frac1{\sqrt{\omega_4}}
   \frac{d[\partial_t(g_t^N)_{\mathrm{ac}}]}{d\omega_4}
 =\frac1{\sqrt{\omega_4}}
   \frac{d(\partial_tg_t^N)_{\mathrm{ac}}}{d\omega_4}\\
 &=\frac1{\sqrt{\omega_4}}
   \frac{d[(\pi_3)_\#\Lambda_N[g_t^N]]_{\mathrm{ac}}}{d\omega_4}
  +\frac1{\sqrt{\omega_4}}
   \frac{d[(\pi_4)_\#\Lambda_N[g_t^N]]_{\mathrm{ac}}}{d\omega_4}-\frac2{\sqrt{\omega_4}}
   \frac{d[(\pi_1)_\#\Lambda_N[g_t^N]]_{\mathrm{ac}}}{d\omega_4}.
 \end{aligned}
\end{equation}

Let $E\subset(0,\infty)$ be a Borel set.
By the resonance relation and symmetry of \(S_{\beta,N}\),
\begin{equation}\label{eq:cond-output-frequency-symmetry}
 S_{\beta,N}(\omega_1,\omega_2,\omega_3,\omega_4)
 =S_{\beta,N}(\omega_1,\omega_2,\omega_4,
             \omega_1+\omega_2-\omega_4).
\end{equation}
Using \eqref{eq:cutoff}, we obtain
\begin{equation}\label{eq:cond-third-projection}
 \begin{aligned}
 ((\pi_3)_\#\Lambda_N[g_t^N])(E)&=\frac14\iiint_{\mathcal D}\mathbf1_E(\omega_3)
 K_{\beta,N}(\omega_1,\omega_2,\omega_3)
 g_t^N(d\omega_1)g_t^N(d\omega_2)g_t^N(d\omega_3)\\
 &=\frac14\int_E
 \left[\iint_{(0,\infty)^2}
 \frac{S_{\beta,N}(\omega_1,\omega_2,\omega_3,
                   \omega_1+\omega_2-\omega_3)}
      {\sqrt{\omega_1\omega_2\omega_3}}
 g_t^N(d\omega_1)g_t^N(d\omega_2)\right]g_t^N(d\omega_3)\\
 &=\frac14\int_E
 \left[\iint_{(0,\infty)^2}
 \frac{S_{\beta,N}(\omega_1,\omega_2,\omega_4,
                   \omega_1+\omega_2-\omega_4)}
      {\sqrt{\omega_1\omega_2\omega_4}}
 g_t^N(d\omega_1)g_t^N(d\omega_2)\right]g_t^N(d\omega_4)\\
 &=\frac14\int_E
 \left[\iint_{(0,\infty)^2}
 \frac{S_{\beta,N}(\omega_1,\omega_2,\omega_3,\omega_4)}
      {\sqrt{\omega_1\omega_2\omega_4}}
 g_t^N(d\omega_1)g_t^N(d\omega_2)\right]g_t^N(d\omega_4).
 \end{aligned}
\end{equation}
For the first projection, 
\begin{equation}\label{eq:cond-first-projection}
 \begin{aligned}
 ((\pi_1)_\#\Lambda_N[g_t^N])(E)
 &=\frac14\iiint_{\mathcal D}\mathbf1_E(\omega_1)
 K_{\beta,N}(\omega_1,\omega_2,\omega_3)
 g_t^N(d\omega_1)g_t^N(d\omega_2)g_t^N(d\omega_3)\\
 &=\frac14\int_E
 \left[\iint_{(0,\infty)^2}
 \frac{S_{\beta,N}(\omega_1,\omega_2,\omega_3,
                   \omega_1+\omega_2-\omega_3)}
      {\sqrt{\omega_1\omega_2\omega_3}}
 g_t^N(d\omega_2)g_t^N(d\omega_3)\right]g_t^N(d\omega_1)\\
 &=\frac14\int_E
 \left[\iint_{(0,\infty)^2}
 \frac{S_{\beta,N}(\omega_4,\omega_3,\omega_2,
                   \omega_4+\omega_3-\omega_2)}
      {\sqrt{\omega_4\omega_3\omega_2}}
 g_t^N(d\omega_2)g_t^N(d\omega_3)\right]g_t^N(d\omega_4)\\
 &=\frac14\int_E
 \left[\iint_{(0,\infty)^2}
 \frac{S_{\beta,N}(\omega_3+\omega_4-\omega_2,
                   \omega_2,\omega_3,\omega_4)}
      {\sqrt{\omega_2\omega_3\omega_4}}
 g_t^N(d\omega_2)g_t^N(d\omega_3)\right]g_t^N(d\omega_4)\\
 &=\frac14\int_E
 \left[\iint_{(0,\infty)^2}
 \frac{S_{\beta,N}(\omega_1,\omega_2,\omega_3,\omega_4)}
      {\sqrt{\omega_2\omega_3\omega_4}}
 g_t^N(d\omega_2)g_t^N(d\omega_3)\right]g_t^N(d\omega_4).
 \end{aligned}
\end{equation}
Taking absolutely continuous parts and using
\eqref{eq:cond-density-definition} yields
\begin{equation}\label{eq:cond-projection-densities}
 \begin{aligned}
 \frac1{\sqrt{\omega_4}}
 \frac{d[(\pi_3)_\#\Lambda_N[g_t^N]]_{\mathrm{ac}}}{d\omega_4}
 &=\frac{f_N(t,\omega_4)}4\iint_{(0,\infty)^2}
 \frac{S_{\beta,N}(\omega_1,\omega_2,\omega_3,\omega_4)}
      {\sqrt{\omega_1\omega_2\omega_4}}
 g_t^N(d\omega_1)g_t^N(d\omega_2),\\
 \frac2{\sqrt{\omega_4}}
 \frac{d[(\pi_1)_\#\Lambda_N[g_t^N]]_{\mathrm{ac}}}{d\omega_4}
 &=\frac{f_N(t,\omega_4)}2\iint_{(0,\infty)^2}
 \frac{S_{\beta,N}(\omega_1,\omega_2,\omega_3,\omega_4)}
      {\sqrt{\omega_2\omega_3\omega_4}}
 g_t^N(d\omega_2)g_t^N(d\omega_3).
 \end{aligned}
\end{equation}
Substituting \eqref{eq:cond-projection-densities} into
\eqref{eq:cond-density-measure-derivative} and then using
\eqref{eq:cond-potential}, we find
\begin{equation}\label{eq:cond-density-identity}
 \begin{aligned}
 &\partial_t f_N(t,\omega_4)
 =\frac{f_N(t,\omega_4)}4\iint_{(0,\infty)^2}
 \frac{S_{\beta,N}(\omega_1,\omega_2,\omega_3,\omega_4)}
      {\sqrt{\omega_1\omega_2\omega_4}}
 g_t^N(d\omega_1)g_t^N(d\omega_2)\\
 &\quad-\frac{f_N(t,\omega_4)}2\iint_{(0,\infty)^2}
 \frac{S_{\beta,N}(\omega_1,\omega_2,\omega_3,\omega_4)}
      {\sqrt{\omega_2\omega_3\omega_4}}
 g_t^N(d\omega_2)g_t^N(d\omega_3)
 +\frac1{\sqrt{\omega_4}}
 \frac{d[(\pi_4)_\#\Lambda_N[g_t^N]]_{\mathrm{ac}}}{d\omega_4}\\
 &=V_N(t,\omega_4)f_N(t,\omega_4)
 +\frac1{\sqrt{\omega_4}}
 \frac{d[(\pi_4)_\#\Lambda_N[g_t^N]]_{\mathrm{ac}}}{d\omega_4}\\
 &\quad-\frac{f_N(t,\omega_4)}2\iint_{(0,\omega_4/2)^2}
 \frac{S_{\beta,N}(\omega_1,\omega_2,\omega_3,\omega_4)}
      {\sqrt{\omega_2\omega_3\omega_4}}
 g_t^N(d\omega_2)g_t^N(d\omega_3).
 \end{aligned}
\end{equation}

Now we estimate $(\pi_4)_\#\Lambda_N[g_t^N]$.
We first need the following inequality.
\begin{equation}\label{eq:cond-disjoint-gain-regions}
 \mathbf1_{\{0<\omega_2,\omega_3<\omega_4/2\}}
 +\mathbf1_{\{0<\omega_1,\omega_3<\omega_4/2\}}
 +\mathbf1_{\{\omega_1,\omega_2>\omega_4/2\}}
 \le1\qquad\text{on }\mathcal D.
\end{equation}
Indeed, the first two sets in \eqref{eq:cond-disjoint-gain-regions}
are disjoint on $\mathcal D$, since their intersection would imply
$\omega_1+\omega_2<\omega_4\le\omega_3+\omega_4
=\omega_1+\omega_2$.
The third set requires $\omega_1,\omega_2>\omega_4/2$,
so it is disjoint from each of the first two sets. 

Since $g_t^N(d\omega_j)\ge
\sqrt{\omega_j}f_N(t,\omega_j)d\omega_j$ for $j=1,2,3$,
\eqref{eq:cond-disjoint-gain-regions} gives
\begin{equation}\label{eq:cond-restricted-output}
 \begin{aligned}
 &((\pi_4)_\#\Lambda_N[g_t^N])(E)\\
 &\quad=\frac14\iiint_{\mathcal D}\mathbf1_E(\omega_4)
 K_{\beta,N}(\omega_1,\omega_2,\omega_3)
 g_t^N(d\omega_1)g_t^N(d\omega_2)g_t^N(d\omega_3)\\
 &\quad\ge\frac14
 \iiint_{\substack{\mathcal D\\0<\omega_2,\omega_3<\omega_4/2}}
 \mathbf1_E(\omega_4)
 \frac{S_{\beta,N}(\omega_1,\omega_2,\omega_3,\omega_4)}
      {\sqrt{\omega_2\omega_3}}
 f_N(t,\omega_1)d\omega_1g_t^N(d\omega_2)g_t^N(d\omega_3)\\
 &\quad+\frac14
 \iiint_{\substack{\mathcal D\\0<\omega_1,\omega_3<\omega_4/2}}
 \mathbf1_E(\omega_4)
 \frac{S_{\beta,N}(\omega_1,\omega_2,\omega_3,\omega_4)}
      {\sqrt{\omega_1\omega_3}}
 f_N(t,\omega_2)g_t^N(d\omega_1)d\omega_2g_t^N(d\omega_3)\\
 &\quad+\frac14
 \iiint_{\substack{\mathcal D\\\omega_1,\omega_2>\omega_4/2}}
 \mathbf1_E(\omega_4)S_{\beta,N}(\omega_1,\omega_2,\omega_3,\omega_4)
 f_N(t,\omega_1)f_N(t,\omega_2)f_N(t,\omega_3)
 d\omega_1d\omega_2d\omega_3.
 \end{aligned}
\end{equation}
The first two terms on the right of \eqref{eq:cond-restricted-output}
are equal under relabeling $(\omega_1,\omega_2)\mapsto(\omega_2,\omega_1)$, i.e. 
\begin{equation}\label{eq:cond-output-symmetry}
 \begin{aligned}
 &\iiint_{\substack{\mathcal D\\0<\omega_1,\omega_3<\omega_4/2}}
 \mathbf1_E(\omega_4)
 \frac{S_{\beta,N}(\omega_1,\omega_2,\omega_3,\omega_4)}
      {\sqrt{\omega_1\omega_3}}
 f_N(t,\omega_2)
 g_t^N(d\omega_1)d\omega_2g_t^N(d\omega_3)\\
 &\quad=\iiint_{\substack{\mathcal D\\0<\omega_2,\omega_3<\omega_4/2}}
 \mathbf1_E(\omega_4)
 \frac{S_{\beta,N}(\omega_1,\omega_2,\omega_3,\omega_4)}
      {\sqrt{\omega_2\omega_3}}
 f_N(t,\omega_1)
 d\omega_1g_t^N(d\omega_2)g_t^N(d\omega_3).
 \end{aligned}
\end{equation}
For fixed $\omega_2,\omega_3>0$,
$\omega_4>2(\omega_2 \vee \omega_3)\iff\omega_1>2(\omega_2 \vee \omega_3)+\omega_3-\omega_2$. Tonelli's theorem then implies
\begin{equation}\label{eq:cond-small-output-change}
 \begin{aligned}
 &\frac12
 \iiint_{\substack{\mathcal D\\0<\omega_2,\omega_3<\omega_4/2}}
 \mathbf1_E(\omega_4)
 \frac{S_{\beta,N}(\omega_1,\omega_2,\omega_3,\omega_4)}
      {\sqrt{\omega_2\omega_3}}
 f_N(t,\omega_1)
 d\omega_1g_t^N(d\omega_2)g_t^N(d\omega_3)\\
 &=\frac12\iint_{(0,\infty)^2}
 \int_{2(\omega_2\vee\omega_3)+\omega_3-\omega_2}^{\infty}
 \mathbf1_E(\omega_1+\omega_2-\omega_3)
 \frac{S_{\beta,N}(\omega_1,\omega_2,\omega_3,
                   \omega_1+\omega_2-\omega_3)}
      {\sqrt{\omega_2\omega_3}}\\[-0.3ex]
 &\qquad\times f_N(t,\omega_1)
 d\omega_1g_t^N(d\omega_2)g_t^N(d\omega_3)\\
 &=\frac12\iint_{(0,\infty)^2}
 \int_{2(\omega_2\vee\omega_3)}^{\infty}
 \mathbf1_E(\omega_4)
 \frac{S_{\beta,N}(\omega_3+\omega_4-\omega_2,
                   \omega_2,\omega_3,\omega_4)}
      {\sqrt{\omega_2\omega_3}}f_N(t,\omega_3+\omega_4-\omega_2)
 d\omega_4g_t^N(d\omega_2)g_t^N(d\omega_3)\\
 &=\frac12\int_E\iint_{(0,\omega_4/2)^2}
 \frac{S_{\beta,N}(\omega_3+\omega_4-\omega_2,
                   \omega_2,\omega_3,\omega_4)}
      {\sqrt{\omega_2\omega_3}}f_N(t,\omega_3+\omega_4-\omega_2)
 g_t^N(d\omega_2)g_t^N(d\omega_3)d\omega_4.
 \end{aligned}
\end{equation}
For fixed $\omega_1,\omega_2>0$, 
$0<\omega_4<2(\omega_1 \wedge \omega_2)\iff|\omega_1-\omega_2|<\omega_3<\omega_1+\omega_2$.
Since $E\subset(0,\infty)$, the third term in
\eqref{eq:cond-restricted-output} satisfies
\begin{equation}\label{eq:cond-cubic-output-change}
 \begin{aligned}
 &\frac14
 \iiint_{\substack{\mathcal D\\\omega_1,\omega_2>\omega_4/2}}
 \mathbf1_E(\omega_4)
 S_{\beta,N}(\omega_1,\omega_2,\omega_3,\omega_4)
 f_N(t,\omega_1)f_N(t,\omega_2)f_N(t,\omega_3)
 d\omega_1d\omega_2d\omega_3\\
 &=\frac14\iint_{(0,\infty)^2}
 \int_{|\omega_1-\omega_2|}^{\omega_1+\omega_2}
 \mathbf1_E(\omega_1+\omega_2-\omega_3)
 S_{\beta,N}(\omega_1,\omega_2,\omega_3,
             \omega_1+\omega_2-\omega_3)\\[-0.3ex]
 &\qquad\times f_N(t,\omega_1)f_N(t,\omega_2)f_N(t,\omega_3)
 d\omega_3d\omega_1d\omega_2\\
 &=\frac14\iint_{(0,\infty)^2}
 \int_0^{2(\omega_1\wedge\omega_2)}
 \mathbf1_E(\omega_4)
 S_{\beta,N}(\omega_1,\omega_2,
             \omega_1+\omega_2-\omega_4,\omega_4)\\[-0.3ex]
 &\qquad\times f_N(t,\omega_1)f_N(t,\omega_2)
 f_N(t,\omega_1+\omega_2-\omega_4)
 d\omega_4d\omega_1d\omega_2\\
 &=\frac14\int_E\iint_{(\omega_4/2,\infty)^2}
 S_{\beta,N}(\omega_1,\omega_2,
             \omega_1+\omega_2-\omega_4,\omega_4)\\[-0.3ex]
 &\qquad\times f_N(t,\omega_1)f_N(t,\omega_2)
 f_N(t,\omega_1+\omega_2-\omega_4)
 d\omega_1d\omega_2d\omega_4
 =\int_E\sqrt{\omega_4}\,\mathcal G_N(t,\omega_4)d\omega_4.
 \end{aligned}
\end{equation}
Combining \eqref{eq:cond-restricted-output}--\eqref{eq:cond-cubic-output-change}, we obtain
\begin{equation}\label{eq:cond-output-measure-bound}
 \begin{aligned}
 ((\pi_4)_\#\Lambda_N[g_t^N])(E)
 &\ge\int_E\sqrt{\omega_4}\,\mathcal G_N(t,\omega_4)d\omega_4\\
 &\quad+\frac12\int_E\iint_{(0,\omega_4/2)^2}
 \frac{S_{\beta,N}(\omega_1,\omega_2,\omega_3,\omega_4)}
      {\sqrt{\omega_2\omega_3}}
 f_N(t,\omega_1)
 g_t^N(d\omega_2)g_t^N(d\omega_3)d\omega_4.
 \end{aligned}
\end{equation}
Taking absolutely continuous parts and dividing
the resulting densities by $\sqrt{\omega_4}$ yields
\begin{equation}\label{eq:cond-retained-output}
 \begin{aligned}
 \frac1{\sqrt{\omega_4}}
 \frac{d[(\pi_4)_\#\Lambda_N[g_t^N]]_{\mathrm{ac}}}{d\omega_4}
 &\ge\mathcal G_N(t,\omega_4)\\
 &\quad+\frac12\iint_{(0,\omega_4/2)^2}
 \frac{S_{\beta,N}(\omega_1,\omega_2,\omega_3,\omega_4)}
      {\sqrt{\omega_2\omega_3\omega_4}}
 f_N(t,\omega_1)g_t^N(d\omega_2)g_t^N(d\omega_3).
 \end{aligned}
\end{equation}
Substituting \eqref{eq:cond-retained-output} into
\eqref{eq:cond-density-identity} gives
\begin{equation}\label{eq:cond-density-gain-expanded}
 \begin{aligned}
 \partial_t f_N(t,\omega_4)
 \ge{}&V_N(t,\omega_4)f_N(t,\omega_4)+\mathcal G_N(t,\omega_4)\\
 &+\frac12\iint_{(0,\omega_4/2)^2}
 \frac{S_{\beta,N}(\omega_1,\omega_2,\omega_3,\omega_4)}
      {\sqrt{\omega_2\omega_3\omega_4}}
 [f_N(t,\omega_1)-f_N(t,\omega_4)]
 g_t^N(d\omega_2)g_t^N(d\omega_3),
 \end{aligned}
\end{equation}
which is \eqref{eq:cond-output-gain}.

It remains to estimate $V_N$. In the second integral of
\eqref{eq:cond-potential}, let
$(x,y,z)=\operatorname{sort}(\omega_2,\omega_3,\omega_4)$.
By \eqref{eq:cutoff} and \eqref{eq:condensation-kernel-bound},
\begin{equation}\label{eq:cond-loss-median}
 \frac{S_{\beta,N}(\omega_1,\omega_2,\omega_3,\omega_4)}
      {\sqrt{\omega_2\omega_3\omega_4}}
 \le\ell_N(\omega_4)
 \frac{S_\beta(\omega_1,\omega_2,\omega_3,\omega_4)}
      {\sqrt{\omega_2\omega_3\omega_4}}
 \le2^\beta y^{-d}\ell_N(\omega_4).
\end{equation}
If $(\omega_2,\omega_3)\notin(0,\omega_4/2)^2$, then
$\omega_2 \vee \omega_3\ge\omega_4/2$ and $y\ge\omega_4/2$.
Using \eqref{eq:cond-potential}, \eqref{eq:cond-loss-median}, we obtain
\begin{equation}\label{eq:cond-potential-crude}
 \begin{aligned}
 V_N(t,\omega_4)
 &\ge-\frac12\iint_{(0,\infty)^2\setminus(0,\omega_4/2)^2}
 \frac{S_{\beta,N}(\omega_1,\omega_2,\omega_3,\omega_4)}
      {\sqrt{\omega_2\omega_3\omega_4}}
 g_t^N(d\omega_2)g_t^N(d\omega_3)\\
 &\ge-2^{\beta-1}\ell_N(\omega_4)
 \iint_{(0,\infty)^2\setminus(0,\omega_4/2)^2}
 y^{-d}g_t^N(d\omega_2)g_t^N(d\omega_3)\\
 &\ge-2^{\beta+d-1}\omega_4^{-d}\ell_N(\omega_4)
       [g_t^N((0,\infty))]^2\\
 &\ge-2^{-3\beta}\omega_4^{-d}\ell_N(\omega_4)
 \ge-\omega_4^{-d}\ell_N(\omega_4).
 \end{aligned}
\end{equation}
Combining \eqref{eq:cond-output-gain} and \eqref{eq:cond-potential-crude}
with $f_N\ge0$ gives \eqref{eq:cond-density-forced}.
Discarding $\mathcal G_N\ge0$ gives \eqref{eq:cond-density-homogeneous}.
\end{proof}

\subsection{Comparison principle}
\label{subsec:cond-comparison}

\begin{lemma}[Boundedness of \(L_N\) and \(V_N\)]\label{lem:cond-cutoff-boundedness}
For every $N\ge1$ and $t\ge0$, \eqref{eq:cond-jump-operator} defines a
bounded linear operator on $L^1((0,\infty),\sqrt\omega d\omega)$.
For every $u\in L^1((0,\infty),\sqrt\omega d\omega)$,
\begin{equation}\label{eq:cond-jump-operator-boundedness}
 \int_0^\infty\sqrt\omega |L_N(t)[u](\omega)| d\omega
 \le N(2N)^{3/2}\int_0^\infty\sqrt\omega |u(\omega)| d\omega.
\end{equation}
For every $\omega>0$,
\begin{equation}\label{eq:cond-bounded-cutoff-coefficients}
 |V_N(t,\omega)|\le3N(2N)^{3/2}/4,
 \qquad 0\le\omega^{-d}\ell_N(\omega)\le(2N)^d.
\end{equation}
\end{lemma}
\begin{proof}
Fix $N\ge1$ and $t\ge0$, and set $C_N=N(2N)^{3/2}$. Action conservation and \eqref{eq:cutoff} give
\begin{equation}\label{eq:cond-cutoff-basic-bounds}
 \begin{aligned}
 0\le S_{\beta,N}(\omega_1,\omega_2,\omega_3,\omega_4)
 &\le N\prod_{j=1}^4\mathbf1_{\{\omega_j>(2N)^{-1}\}},\\
 0\le K_{\beta,N}(\omega_1,\omega_2,\omega_3)
 &\le N\prod_{j=1}^3
 \frac{\mathbf1_{\{\omega_j>(2N)^{-1}\}}}{\sqrt{\omega_j}}
 \le C_N,\\
 g_t^N((0,\infty))&\le g_t^N([0,\infty))=1.
 \end{aligned}
\end{equation}

Choose a finite Borel representative of $u$. 
For fixed $\omega_2,\omega_3>0$,
$\omega_2,\omega_3<\omega_4/2\iff \omega_4>2(\omega_2 \vee \omega_3)$.
Tonelli's theorem yields, for the gain term in \eqref{eq:cond-jump-operator},
\begin{equation}\label{eq:cond-cutoff-gain-bound}
 \begin{aligned}
 &\frac12\int_0^\infty\iint_{(0,\omega_4/2)^2}
 \frac{S_{\beta,N}(\omega_1,\omega_2,\omega_3,\omega_4)}
      {\sqrt{\omega_2\omega_3}}
 |u(\omega_1)| g_t^N(d\omega_2)g_t^N(d\omega_3)d\omega_4\\
 &=\frac12\iint_{(0,\infty)^2}
 \int_{2(\omega_2 \vee \omega_3)+\omega_3-\omega_2}^{\infty}
 K_{\beta,N}(\omega_1,\omega_2,\omega_3)
 \sqrt{\omega_1}|u(\omega_1)| d\omega_1
 g_t^N(d\omega_2)g_t^N(d\omega_3)\\
 &\le\frac{C_N}{2}[g_t^N((0,\infty))]^2
 \int_0^\infty\sqrt{\omega_1}|u(\omega_1)| d\omega_1\le\frac{C_N}{2}
 \int_0^\infty\sqrt\omega |u(\omega)| d\omega.
 \end{aligned}
\end{equation}
For the loss term, symmetry of $S_{\beta,N}$ and
\eqref{eq:cond-cutoff-basic-bounds} give
\begin{equation}\label{eq:cond-cutoff-loss-bound}
 \begin{aligned}
 &\frac12\int_0^\infty\iint_{(0,\omega_4/2)^2}
 \frac{S_{\beta,N}(\omega_1,\omega_2,\omega_3,\omega_4)}
      {\sqrt{\omega_2\omega_3}}
 |u(\omega_4)| g_t^N(d\omega_2)g_t^N(d\omega_3)d\omega_4\\
 &=\frac12\int_0^\infty\sqrt{\omega_4}|u(\omega_4)|
 \iint_{(0,\omega_4/2)^2}
 K_{\beta,N}(\omega_4,\omega_3,\omega_2)
 g_t^N(d\omega_2)g_t^N(d\omega_3)d\omega_4\\
 &\le\frac{C_N}{2}[g_t^N((0,\infty))]^2
 \int_0^\infty\sqrt{\omega_4}|u(\omega_4)| d\omega_4
 \le\frac{C_N}{2}\int_0^\infty\sqrt\omega |u(\omega)| d\omega.
 \end{aligned}
\end{equation}
Since $u\in L^1((0,\infty),\sqrt\omega d\omega)$, the triangle inequality,
\eqref{eq:cond-cutoff-gain-bound}, and \eqref{eq:cond-cutoff-loss-bound} give
\begin{equation}\label{eq:cond-cutoff-operator-estimate}
 \begin{aligned}
 &\int_0^\infty\sqrt{\omega_4}|L_N(t)[u](\omega_4)| d\omega_4\\
 &\le\frac12\int_0^\infty\iint_{(0,\omega_4/2)^2}
 \frac{S_{\beta,N}(\omega_1,\omega_2,\omega_3,\omega_4)}
      {\sqrt{\omega_2\omega_3}}
 \bigl(|u(\omega_1)|+|u(\omega_4)|\bigr)
 g_t^N(d\omega_2)g_t^N(d\omega_3)d\omega_4\\
 &\le\left(\frac{C_N}{2}+\frac{C_N}{2}\right)
 \int_0^\infty\sqrt\omega |u(\omega)| d\omega
 =C_N\int_0^\infty\sqrt\omega |u(\omega)| d\omega,
 \end{aligned}
\end{equation}
which proves \eqref{eq:cond-jump-operator-boundedness}.

If $u$ and $\widetilde u$ are two representatives of the same weighted
$L^1$ class, the same estimates give
\begin{equation}\label{eq:cond-cutoff-representative-independence}
 \begin{aligned}
 \int_0^\infty\sqrt\omega
 |L_N(t)[u](\omega)-L_N(t)[\widetilde u](\omega)| d\omega\le C_N\int_0^\infty\sqrt\omega
 |u(\omega)-\widetilde u(\omega)| d\omega=0.
 \end{aligned}
\end{equation}
Hence $L_N(t)$ is well defined. The linearity follows directly
from its integral formula.

For \eqref{eq:cond-bounded-cutoff-coefficients}, symmetry of $S_{\beta,N}$ and
\eqref{eq:cond-cutoff-basic-bounds} yield
\begin{equation}\label{eq:cond-cutoff-potential-bound}
 \begin{aligned}
 |V_N(t,\omega_4)|
 &\le\frac14\iint_{(0,\infty)^2}
 K_{\beta,N}(\omega_1,\omega_2,\omega_4)
 g_t^N(d\omega_1)g_t^N(d\omega_2)\\
 &\quad+\frac12\iint_{(0,\infty)^2\setminus(0,\omega_4/2)^2}
 K_{\beta,N}(\omega_4,\omega_3,\omega_2)
 g_t^N(d\omega_2)g_t^N(d\omega_3)\\
 &\le\left(\frac14+\frac12\right)C_N[g_t^N((0,\infty))]^2
 \le\frac34C_N.
 \end{aligned}
\end{equation}
Finally, since $0\le\ell_N\le1$ and
$\ell_N=0$ on $[0,(2N)^{-1}]$,
\begin{equation}\label{eq:cond-cutoff-multiplier-bound}
 0\le\omega^{-d}\ell_N(\omega)
 \le\omega^{-d}\mathbf1_{\{\omega>(2N)^{-1}\}}
 \le(2N)^d,\qquad \omega>0.
\end{equation}
We have thus proved \eqref{eq:cond-bounded-cutoff-coefficients}.
\end{proof}

\begin{lemma}\label{lem:cond-fixed-profile}
Let $0<a<b$ and set
\begin{align}
 u(\omega_4)
 &=\sin\!\left(\frac{\pi(\omega_4-a)}{b-a}\right)
   \mathbf1_{[a,b]}(\omega_4),
 \label{eq:cond-sine-profile}\\
 C&=\frac{\pi^2b^{2-d}}{4(b-a)^2}+\frac5{4a^d}.
 \label{eq:cond-profile-constant}
\end{align}
\begin{enumerate}[label=\textup{(\roman*)},leftmargin=*]
\item For $z\in[a,b]$ and $0<x\le y<z/2$,
\begin{equation}\label{eq:cond-paired-sine-bound}
 \begin{aligned}
 &\left(1+\frac{y-x}{z}\right)^\beta[u(z+y-x)-u(z)]
 +\left(1-\frac{y-x}{z}\right)^\beta[u(z+x-y)-u(z)]\\
 &\qquad\ge-\left(\frac{\pi^2(y-x)^2}{(b-a)^2}
                     +\frac{y-x}{z}\right)u(z).
 \end{aligned}
\end{equation}
\item For every $N\ge2/a \vee (3b/2)^{4\beta+1/2}$,
\begin{equation}\label{eq:cond-fixed-profile}
 L_N(t)[u](\omega_4)-\omega_4^{-d}u(\omega_4)
 \ge-Cu(\omega_4),\qquad t\ge0,\quad\omega_4>0.
\end{equation}
\end{enumerate}
\end{lemma}
\begin{proof}
Fix $N\ge2/a \vee (3b/2)^{4\beta+1/2}$ and $t\ge0$.
For $\omega_4\in[a,b]$, \eqref{eq:cond-jump-operator} gives
\begin{equation}\label{eq:cond-sine-jump-exchange}
 \begin{aligned}
 L_N(t)[u](\omega_4)&=\frac12\iint_{(0,\omega_4/2)^2}
 \frac{S_{\beta,N}(\omega_1,
                   \omega_2,\omega_3,\omega_4)}
      {\sqrt{\omega_2\omega_3\omega_4}}[u(\omega_1)-u(\omega_4)]
  g_t^N(d\omega_2)g_t^N(d\omega_3)\\
 &=\frac12\iint_{(0,\omega_4/2)^2}
 \frac{S_{\beta,N}(\omega_2+\omega_4-\omega_3,
                   \omega_3,\omega_2,\omega_4)}
      {\sqrt{\omega_2\omega_3\omega_4}}[u(\omega_2+\omega_4-\omega_3)-u(\omega_4)]
  g_t^N(d\omega_2)g_t^N(d\omega_3).
 \end{aligned}
\end{equation}
For each $(\omega_2,\omega_3)\in(0,\omega_4/2)^2$, set
$(x,y,z)=\operatorname{sort}(\omega_2,\omega_3,\omega_4)$.
Then $z=\omega_4$ and $0<x\le y<z/2$.
Hence \eqref{eq:cond-sine-jump-exchange}
yields
\begin{equation}\label{eq:cond-symmetrized-jump}
 \begin{aligned}
 L_N(t)[u](\omega_4) &=\frac14\iint_{(0,\omega_4/2)^2}
 \biggl\{
 \frac{S_{\beta,N}(z+y-x,x,y,z)}{\sqrt{xyz}}
 [u(z+y-x)-u(z)]\\
 &\hspace{24mm}+
 \frac{S_{\beta,N}(z+x-y,y,x,z)}{\sqrt{xyz}}
 [u(z+x-y)-u(z)]
 \biggr\}
  g_t^N(d\omega_2)g_t^N(d\omega_3).
 \end{aligned}
\end{equation}
We then estimate the sum inside braces in
\eqref{eq:cond-symmetrized-jump} for
$z\in[a,b]$ and $0<x\le y<z/2$.

\underline{Case 1. $y-x\le b-a$.}
Then $z\pm(y-x)\in[2a-b,2b-a]$.
Since $\sin(\pi(\omega_4-a)/(b-a))\le0$ for
$\omega_4\in[2a-b,a]\cup[b,2b-a]$,
\eqref{eq:cond-sine-profile} gives
\begin{equation}\label{eq:cond-sine-difference}
 \begin{aligned}
 &u(z+y-x)+u(z+x-y)-2u(z)\\
 &\quad\ge
 \sin\!\left(\frac{\pi(z+y-x-a)}{b-a}\right)
 +\sin\!\left(\frac{\pi(z+x-y-a)}{b-a}\right)
 -2\sin\!\left(\frac{\pi(z-a)}{b-a}\right)\\
 &\quad=2\left[\cos\!\left(\frac{\pi(y-x)}{b-a}\right)-1\right]u(z)\ge-\frac{\pi^2(y-x)^2}{(b-a)^2}u(z).
 \end{aligned}
\end{equation}
The last inequality uses
$1-\cos r=2\sin^2(r/2)\le r^2/2$ for $r\in\mathbb R$.

\underline{Case 2. $y-x>b-a$.} Then since \(u\geq 0\),
\begin{equation}\label{eq:cond-sine-large-difference}
 u(z+y-x)+u(z+x-y)-2u(z)
 \ge-2u(z)
 \ge-\frac{\pi^2(y-x)^2}{(b-a)^2}u(z).
\end{equation}
Thus the last bound in \eqref{eq:cond-sine-difference}
holds for every $0<x\le y<z/2$.

Since $z\in[a,b]$,
$N\ge2/a\ge2/z$ and
$N\ge(3b/2)^{4\beta+1/2}\ge(3z/2)^{4\beta+1/2}$. \eqref{eq:cutoff} and
\eqref{eq:cond-kernel-small-inputs} give
\begin{equation}\label{eq:cond-sine-kernel-coefficients}
 \begin{aligned}
 \frac{S_{\beta,N}(z+y-x,x,y,z)}{\sqrt{xyz}}
 &=z^{2\beta-1/2}x^\beta y^{\beta-1/2}
   \ell_N(x)\ell_N(y)
   \left(1+\frac{y-x}{z}\right)^\beta,\\
 \frac{S_{\beta,N}(z+x-y,y,x,z)}{\sqrt{xyz}}
 &=z^{2\beta-1/2}x^\beta y^{\beta-1/2}
   \ell_N(x)\ell_N(y)
   \left(1-\frac{y-x}{z}\right)^\beta.
 \end{aligned}
\end{equation}
Besides,
\begin{equation}\label{eq:cond-sine-coefficient-difference}
 \begin{aligned}
 0\le\left(1+\frac{y-x}{z}\right)^\beta
      -\left(1-\frac{y-x}{z}\right)^\beta
 =\beta\int_{1-(y-x)/z}^{1+(y-x)/z}r^{\beta-1}dr\le4\beta\frac{y-x}{z}
 \le\frac{y-x}{z}.
 \end{aligned}
\end{equation}
Here we used $(y-x)/z<1/2$, $0<\beta<1/4$, and
$r^{\beta-1}\le2^{1-\beta}\le2$ for $r\ge1/2$.

Since $0<(1-(y-x)/z)^\beta\le1$ and
$u(z+y-x)-u(z)\ge-u(z)$,
\eqref{eq:cond-sine-difference}--\eqref{eq:cond-sine-coefficient-difference} imply
\begin{equation}
 \begin{aligned}
 &\left(1+\frac{y-x}{z}\right)^\beta[u(z+y-x)-u(z)]
 +\left(1-\frac{y-x}{z}\right)^\beta[u(z+x-y)-u(z)]\\
 &=\left(1-\frac{y-x}{z}\right)^\beta
   [u(z+y-x)+u(z+x-y)-2u(z)]\\
 &\quad+\left[\left(1+\frac{y-x}{z}\right)^\beta
              -\left(1-\frac{y-x}{z}\right)^\beta\right]
   [u(z+y-x)-u(z)]\\
 &\ge-\left[
   \left(1-\frac{y-x}{z}\right)^\beta
   \frac{\pi^2(y-x)^2}{(b-a)^2}
   +\left(1+\frac{y-x}{z}\right)^\beta
   -\left(1-\frac{y-x}{z}\right)^\beta
   \right]u(z)\\
 &\ge-\left(\frac{\pi^2(y-x)^2}{(b-a)^2}
            +\frac{y-x}{z}\right)u(z).
 \end{aligned}
\end{equation}
This proves \eqref{eq:cond-paired-sine-bound}.
Substituting \eqref{eq:cond-sine-kernel-coefficients}
and using \eqref{eq:cond-paired-sine-bound}, we obtain
\begin{equation}\label{eq:cond-paired-jump-bound}
 \begin{aligned}
 &\frac{S_{\beta,N}(z+y-x,x,y,z)}{\sqrt{xyz}}
   [u(z+y-x)-u(z)]\\
 &\quad+\frac{S_{\beta,N}(z+x-y,y,x,z)}{\sqrt{xyz}}
   [u(z+x-y)-u(z)]\\
 &\ge-z^{2\beta-1/2}\ell_N(x)\ell_N(y)
   \left[
   \frac{\pi^2x^\beta y^{\beta-1/2}(y-x)^2}{(b-a)^2}
   +\frac{x^\beta y^{\beta-1/2}(y-x)}{z}
   \right]u(z)\\
 &\ge-\left[
   \frac{\pi^2z^{2\beta-1/2}z^{2\beta+3/2}}{(b-a)^2}
   +z^{2\beta-3/2}z^{2\beta+1/2}
   \right]u(z)\\
 &=-\left(\frac{\pi^2z^{2-d}}{(b-a)^2}+z^{-d}\right)u(z).
 \end{aligned}
\end{equation}
The second inequality follows from $0\le\ell_N\le1$,
\eqref{eq:cond-kernel-first-difference}, and
\eqref{eq:cond-kernel-second-difference}.

Substituting \eqref{eq:cond-paired-jump-bound} into
\eqref{eq:cond-symmetrized-jump}, with $z=\omega_4$, gives
\begin{equation}\label{eq:cond-sine-jump-bound}
 \begin{aligned}
 L_N(t)[u](\omega_4)
 &\ge-\frac14
 \left(\frac{\pi^2\omega_4^{2-d}}{(b-a)^2}+\omega_4^{-d}\right)
 u(\omega_4)
 \iint_{(0,\omega_4/2)^2}
 g_t^N(d\omega_2)g_t^N(d\omega_3)\\
 &=-\frac14
 \left(\frac{\pi^2\omega_4^{2-d}}{(b-a)^2}+\omega_4^{-d}\right)
 u(\omega_4)[g_t^N((0,\omega_4/2))]^2\\
 &\ge-\frac14
 \left(\frac{\pi^2\omega_4^{2-d}}{(b-a)^2}+\omega_4^{-d}\right)
 u(\omega_4).
 \end{aligned}
\end{equation}
The last inequality uses $g_t^N((0,\omega_4/2))\le1$.
Since $0<d<1$ and $a\le\omega_4\le b$,
\begin{equation}\label{eq:cond-sine-global-bound}
 \begin{aligned}
 L_N(t)[u](\omega_4)-\omega_4^{-d}u(\omega_4)
 &\ge-\left[
 \frac{\pi^2\omega_4^{2-d}}{4(b-a)^2}
 +\left(\frac14+1\right)\omega_4^{-d}
 \right]u(\omega_4)\\
 &\ge-\left[
 \frac{\pi^2b^{2-d}}{4(b-a)^2}+\frac5{4a^d}
 \right]u(\omega_4)
 =-Cu(\omega_4).
 \end{aligned}
\end{equation}

For $\omega_4\notin[a,b]$, $u(\omega_4)=0$.
With $\omega_1=\omega_3+\omega_4-\omega_2$,
\eqref{eq:cond-jump-operator} gives
\begin{equation}\label{eq:cond-sine-outside-support}
 \begin{aligned}
 L_N(t)[u](\omega_4)-\omega_4^{-d}u(\omega_4)
 &=\frac12\iint_{(0,\omega_4/2)^2}
 \frac{S_{\beta,N}(\omega_1,\omega_2,\omega_3,\omega_4)}
      {\sqrt{\omega_2\omega_3\omega_4}}
 u(\omega_1) g_t^N(d\omega_2)g_t^N(d\omega_3)\\
 &\ge0=-Cu(\omega_4).
 \end{aligned}
\end{equation}
Equations \eqref{eq:cond-sine-global-bound} and
\eqref{eq:cond-sine-outside-support} prove
\eqref{eq:cond-fixed-profile}.
\end{proof}

\needspace{5\baselineskip}
\begin{lemma}[Comparison principle]\label{lem:cond-sine-comparison}
Let $u,C,N$ be as in \cref{lem:cond-fixed-profile}\textup{(ii)}, and let $s,c\ge0$.
\begin{enumerate}[label=\textup{(\roman*)},leftmargin=*]
\item If $f_N(s,\omega_4)\ge cu(\omega_4)$, then
\begin{equation}\label{eq:cond-persistent-sine}
 f_N(t,\omega_4)\ge ce^{-C(t-s)}u(\omega_4),
 \qquad t\ge s.
\end{equation}
\item If $D\ge0$, $T>s$, and
\begin{equation}\label{eq:cond-forced-assumption}
 \begin{aligned}
 \partial_t f_N(t,\omega_4)
 \ge L_N(t)[f_N(t,\cdot)](\omega_4)
 -\omega_4^{-d}\ell_N(\omega_4)f_N(t,\omega_4)+ce^{-Dt}u(\omega_4),
 \qquad s\le t\le T,
 \end{aligned}
\end{equation}
then
\begin{equation}\label{eq:cond-forced-sine}
 f_N(t,\omega_4)
 \ge cu(\omega_4)\int_s^t e^{-C(t-\tau)-D\tau}d\tau,
 \qquad s\le t\le T.
\end{equation}
\end{enumerate}
\end{lemma}
\begin{proof}
\noindent\textit{Step 1. The difference $w=v-f_N$.}
Fix $T>s$, arbitrary in \textup{(i)}, and set
\begin{equation}\label{eq:cond-sine-comparison-profiles}
 v(t,\omega_4)=
 \begin{cases}
 ce^{-C(t-s)}u(\omega_4),&\text{in \textup{(i)}},\\[1ex]
 cu(\omega_4)\displaystyle\int_s^t e^{-C(t-\tau)-D\tau}d\tau,
 &\text{in \textup{(ii)}}.
 \end{cases}
\end{equation}
With either definition, $v$ is nonnegative and belongs to
$C^1([s,T];L^1((0,\infty),\sqrt\omega d\omega))$.
For each fixed $t$, $v(t,\cdot)$ is a nonnegative multiple of $u$.
By linearity of $L_N(t)$, \eqref{eq:cond-fixed-profile}, and
$0\le\ell_N\le1$,
\begin{equation}\label{eq:cond-sine-comparison-subsolution}
 \begin{aligned}
 L_N(t)[v](\omega_4)-\omega_4^{-d}\ell_N(\omega_4)v(t,\omega_4)
 &\ge L_N(t)[v](\omega_4)-\omega_4^{-d}v(t,\omega_4)\\
 &\ge-Cv(t,\omega_4).
 \end{aligned}
\end{equation}
Set $w=v-f_N$. We check its initial value and differential inequality
separately in the two cases.

In \textup{(i)}, $w(s,\omega_4)=cu(\omega_4)-f_N(s,\omega_4)\le0$.
Since $\partial_t v=-Cv$, \eqref{eq:cond-density-homogeneous}
and \eqref{eq:cond-sine-comparison-subsolution} give
\begin{equation}\label{eq:cond-comparison-unforced-difference}
 \begin{aligned}
 \partial_t w(t,\omega_4)
 &=-Cv(t,\omega_4)-\partial_t f_N(t,\omega_4)\\
 &\le-Cv(t,\omega_4)-L_N(t)[f_N](\omega_4)
       +\omega_4^{-d}\ell_N(\omega_4)f_N(t,\omega_4)\\
 &\le L_N(t)[v](\omega_4)-\omega_4^{-d}\ell_N(\omega_4)v(t,\omega_4)\\
 &\quad-L_N(t)[f_N](\omega_4)
       +\omega_4^{-d}\ell_N(\omega_4)f_N(t,\omega_4)\\
 &=L_N(t)[w](\omega_4)-\omega_4^{-d}\ell_N(\omega_4)w(t,\omega_4).
 \end{aligned}
\end{equation}

In \textup{(ii)}, $w(s,\omega_4)=-f_N(s,\omega_4)\le0$, and
\begin{equation}\label{eq:cond-sine-comparison-sources}
 \begin{aligned}
 \partial_t v(t,\omega_4)
 =cu(\omega_4)\left[e^{-Dt}
       -C\int_s^t e^{-C(t-\tau)-D\tau}d\tau\right]=ce^{-Dt}u(\omega_4)-Cv(t,\omega_4).
 \end{aligned}
\end{equation}
Using \eqref{eq:cond-forced-assumption}, we obtain
\begin{equation}\label{eq:cond-comparison-forced-difference}
 \begin{aligned}
 \partial_t w(t,\omega_4)
 &=ce^{-Dt}u(\omega_4)-Cv(t,\omega_4)-\partial_t f_N(t,\omega_4)\\
 &\le ce^{-Dt}u(\omega_4)-Cv(t,\omega_4)\\
 &\quad-\bigl[L_N(t)[f_N](\omega_4)
       -\omega_4^{-d}\ell_N(\omega_4)f_N(t,\omega_4)
       +ce^{-Dt}u(\omega_4)\bigr]\\
 &=-Cv(t,\omega_4)-L_N(t)[f_N](\omega_4)
       +\omega_4^{-d}\ell_N(\omega_4)f_N(t,\omega_4)\\
 &\le L_N(t)[v](\omega_4)-\omega_4^{-d}\ell_N(\omega_4)v(t,\omega_4)\\
 &\quad-L_N(t)[f_N](\omega_4)
       +\omega_4^{-d}\ell_N(\omega_4)f_N(t,\omega_4)\\
 &=L_N(t)[w](\omega_4)-\omega_4^{-d}\ell_N(\omega_4)w(t,\omega_4).
 \end{aligned}
\end{equation}
Thus, in both cases,
\begin{equation}\label{eq:cond-crude-comparison-assumption}
 w(s,\omega_4)\le0,
 \qquad
 \partial_t w(t,\omega_4)
 \le L_N(t)[w](\omega_4)
 -\omega_4^{-d}\ell_N(\omega_4)w(t,\omega_4).
\end{equation}

\smallskip
\noindent\textit{Step 2. The positive part of $w$.}
By \eqref{eq:cond-density-time-regularity},
$w\in C^1([s,T];L^1((0,\infty),\sqrt\omega d\omega))$, and hence
\begin{equation}\label{eq:cond-comparison-time-integrability}
 \begin{aligned}
 \int_s^T\int_0^\infty\sqrt{\omega_4}
 |\partial_\tau w(\tau,\omega_4)|d\omega_4d\tau\le(T-s)\sup_{\tau\in[s,T]}
 \int_0^\infty\sqrt{\omega_4}
 |\partial_\tau w(\tau,\omega_4)|d\omega_4<\infty.
 \end{aligned}
\end{equation}
The fundamental theorem of calculus in $L^1$ and Fubini's theorem
give an absolutely continuous representative of \(w\) in time such that,
for a.e.\ $\omega_4>0$ and every $t\in[s,T]$,
\begin{equation}\label{eq:cond-comparison-time-representative}
 w(t,\omega_4)=w(s,\omega_4)
       +\int_s^t\partial_\tau w(\tau,\omega_4)d\tau.
\end{equation}
By \eqref{eq:cond-cutoff-representative-independence}, this choice
does not change $L_N(t)[w]$ as a weighted $L^1$ class.
Hence 
\begin{equation}\label{eq:cond-comparison-positive-chain-rule}
 \begin{aligned}
 \partial_t w_+(t,\omega_4)
 &=\mathbf1_{\{w(t,\omega_4)>0\}}\partial_t w(t,\omega_4),\\
 |\partial_t w_+(t,\omega_4)|
 &\le|\partial_t w(t,\omega_4)|.
 \end{aligned}
\end{equation}
Using \eqref{eq:cond-crude-comparison-assumption} and
\eqref{eq:cond-jump-operator}, we obtain
\begin{equation}\label{eq:cond-comparison-positive-part-expanded}
 \begin{aligned}
 \partial_t w_+(t,\omega_4)
 &\le\mathbf1_{\{w(t,\omega_4)>0\}}L_N(t)[w](\omega_4)
       -\omega_4^{-d}\ell_N(\omega_4)w_+(t,\omega_4)\\
 &=\frac12\iint_{(0,\omega_4/2)^2}
 \frac{S_{\beta,N}(\omega_1,\omega_2,\omega_3,\omega_4)}
      {\sqrt{\omega_2\omega_3\omega_4}}
 \bigl[\mathbf1_{\{w(t,\omega_4)>0\}}w(t,\omega_1)-w_+(t,\omega_4)\bigr]
 g_t^N(d\omega_2)g_t^N(d\omega_3)\\
 &\quad-\omega_4^{-d}\ell_N(\omega_4)w_+(t,\omega_4).
 \end{aligned}
\end{equation}
Since $\mathbf1_{\{w(t,\omega_4)>0\}}w(t,\omega_1)\le w_+(t,\omega_1)$ and \(S_{\beta,N}\geq 0\),
\begin{equation}\label{eq:cond-comparison-positive-part}
 \begin{aligned}
 \partial_t w_+(t,\omega_4)
 &\le\frac12\iint_{(0,\omega_4/2)^2}
 \frac{S_{\beta,N}(\omega_1,\omega_2,\omega_3,\omega_4)}
      {\sqrt{\omega_2\omega_3\omega_4}}
 [w_+(t,\omega_1)-w_+(t,\omega_4)]\\
 &\quad\times g_t^N(d\omega_2)g_t^N(d\omega_3)
 -\omega_4^{-d}\ell_N(\omega_4)w_+(t,\omega_4)\\
 &\le\frac12\iint_{(0,\omega_4/2)^2}
 \frac{S_{\beta,N}(\omega_1,\omega_2,\omega_3,\omega_4)}
      {\sqrt{\omega_2\omega_3\omega_4}}
 w_+(t,\omega_1)g_t^N(d\omega_2)g_t^N(d\omega_3).
 \end{aligned}
\end{equation}

\smallskip
\noindent\textit{Step 3. The weighted integral of $w_+$.}
By \eqref{eq:cond-comparison-time-integrability}--\eqref{eq:cond-comparison-positive-chain-rule}
and Fubini's theorem,
\begin{equation}\label{eq:cond-comparison-integrated-chain-rule}
 \begin{aligned}
 \int_0^\infty\sqrt{\omega_4}w_+(t,\omega_4)d\omega_4
 -\int_0^\infty\sqrt{\omega_4}w_+(s,\omega_4)d\omega_4=\int_s^t\int_0^\infty\sqrt{\omega_4}
 \partial_\tau w_+(\tau,\omega_4)d\omega_4d\tau.
 \end{aligned}
\end{equation}
Thus the weighted integral is absolutely continuous in $t$.
For a.e.\ $t\in[s,T]$, \eqref{eq:cond-comparison-positive-part} gives
\begin{equation}\label{eq:cond-comparison-integral-bound}
 \begin{aligned}
 &\frac{d}{dt}\int_0^\infty\sqrt{\omega_4}w_+(t,\omega_4)d\omega_4
 =\int_0^\infty\sqrt{\omega_4}\,\partial_t w_+(t,\omega_4)d\omega_4\\
 &\quad\le\frac12\int_0^\infty\iint_{(0,\omega_4/2)^2}
 \frac{S_{\beta,N}(\omega_1,\omega_2,\omega_3,\omega_4)}
      {\sqrt{\omega_2\omega_3}}
 w_+(t,\omega_1)g_t^N(d\omega_2)g_t^N(d\omega_3)d\omega_4\\
 &\quad=\frac12\iint_{(0,\infty)^2}
 \int_{2(\omega_2\vee\omega_3)+\omega_3-\omega_2}^\infty
 K_{\beta,N}(\omega_1,\omega_2,\omega_3)
 \sqrt{\omega_1}w_+(t,\omega_1)d\omega_1
 g_t^N(d\omega_2)g_t^N(d\omega_3)\\
 &\quad\le\frac{N(2N)^{3/2}}2[g_t^N((0,\infty))]^2
 \int_0^\infty\sqrt{\omega_1}w_+(t,\omega_1)d\omega_1\\
 &\quad\le\frac{N(2N)^{3/2}}2
 \int_0^\infty\sqrt{\omega_4}w_+(t,\omega_4)d\omega_4.
 \end{aligned}
\end{equation}
Here we used Tonelli's theorem as in
\eqref{eq:cond-cutoff-gain-bound}.
Since $w_+(s,\cdot)=0$, Gronwall's inequality yields, for every $t\in[s,T]$,
\begin{equation}\label{eq:cond-crude-comparison-conclusion}
 \begin{aligned}
 0\le\int_0^\infty\sqrt{\omega_4}w_+(t,\omega_4)d\omega_4\le e^{N(2N)^{3/2}(t-s)/2}
 \int_0^\infty\sqrt{\omega_4}w_+(s,\omega_4)d\omega_4=0.
 \end{aligned}
\end{equation}
Hence $v(t,\cdot)\le f_N(t,\cdot)$ a.e.\ for every $t\in[s,T]$.
Substituting \eqref{eq:cond-sine-comparison-profiles} gives the two conclusions,
and $T>s$ is arbitrary in \textup{(i)}.
\end{proof}

\section{Finite-time condensation}\label{sec:condensation}
Throughout this section, $0<\beta<1/4$, $d=1-4\beta>0$ and \(0\leq g_0\in\mathcal M_1\).
We use the solutions constructed in \cref{thm:measure-main}
and \cref{sec:measure}, with the corresponding cutoff
solutions from \cref{prop:integrated-cutoff-existence}.

\begin{theorem}\label{thm:smooth-condensation}
Let $g_0\in\mathcal M_1$ be nonnegative with a nonzero absolutely
continuous part. The following estimates hold uniformly for all solutions
constructed in \cref{sec:measure} from $g_0$.
\begin{enumerate}[label=\textup{(\roman*)},leftmargin=*]
\item For every $0<\zeta<1$, there is a finite time
$T=T(\beta,g_0,\zeta)$ such that
\begin{equation}\label{eq:cond-general-conclusion}
 g_t(\{0\})\ge(1-\zeta)\mathcal N_0,
 \qquad t\ge T.
\end{equation}
\item There is a constant $C>0$, depending only on $\beta$ and
$g_0/\mathcal N_0$, such that
\begin{equation}\label{eq:cond-tv-rate}
 \|g_t-\mathcal N_0\delta_0\|_{\mathrm{TV}}
 \le C\mathcal N_0(1+\mathcal N_0^2t)^{-d/40},
 \qquad t\ge0.
\end{equation}
\end{enumerate}
\end{theorem}

For the rest of this section, we may assume without loss of generality that $\mathcal N_0=1$.
Indeed, let $g_0$ have wave
action $\mathcal N_0>0$, and let $g^N$ be the corresponding
cutoff solution from \cref{sec:measure}. Define
\begin{equation}\label{eq:cond-mass-normalization}
 \widetilde g_\tau^N
 =\mathcal N_0^{-1}g_{\tau/\mathcal N_0^2}^N,
 \qquad
 \widetilde g_0=\mathcal N_0^{-1}g_0.
\end{equation}
The mass conservation in \cref{prop:integrated-cutoff-existence} then gives
\begin{equation}\label{eq:cond-normalized-unit-mass}\int_{[0,\infty)}\widetilde{g}^N_\tau(d\omega)=\mathcal N_0^{-1}\int_{[0,\infty)}g^N_{\tau/\mathcal{N}_0^2}(d\omega)=\mathcal N_0^{-1}\int_{[0,\infty)}g^N_0(d\omega)=1,\qquad \tau\geq 0.\end{equation}
Besides, \(0\leq \widetilde g_0\in\mathcal M_1\) and has a nonzero absolutely continuous part.

We check that $\widetilde g^N$ satisfies the same cutoff
equation \eqref{eq:cutoff-weak} with initial value \(\widetilde g_0\). For every $\phi\in\operatorname{Lip}_c([0,\infty))$,
the chain rule and \eqref{eq:cutoff-weak} give
\begin{align}
 \dfrac{d}{d\tau}
 \int_{[0,\infty)}\phi(\omega) \widetilde g_\tau^N(d\omega)
 &=
 \frac1{\mathcal N_0^3}
 \left.
 \dfrac{d}{dt}\right|_{t=\tau/\mathcal N_0^2}\int_{[0,\infty)}\phi(\omega) g_t^N(d\omega)
 \notag\\
 &=
 \frac1{4\mathcal N_0^3}
 \iiint_{[0,\infty)^3}
 K_{\beta,N}\Delta\phi 
 d\bigl(g_{\tau/\mathcal N_0^2}^N\bigr)^{\otimes3}
 \notag\\
 &=
 \frac14\iiint_{[0,\infty)^3}
 K_{\beta,N}\Delta\phi 
 d(\widetilde g_\tau^N)^{\otimes3}.
 \label{eq:cond-normalized-equation}
\end{align}

Set $\widetilde g_\tau
=\mathcal N_0^{-1}g_{\tau/\mathcal N_0^2}$.
By \eqref{eq:uniform-narrow-convergence}, there is a subsequence
$(N_j)$ such that, for every $\phi\in C_b([0,\infty))$ and $T>0$, as \(j\to\infty,\)
\begin{align}
 \sup_{0\le\tau\le T}
 \left|\int_{[0,\infty)}\phi(\omega)
 \bigl[\widetilde g_\tau^{N_j}(d\omega)
       -\widetilde g_\tau(d\omega)\bigr]\right|=
 \mathcal N_0^{-1}
 \sup_{0\le t\le T/\mathcal N_0^2}
 \left|\int_{[0,\infty)}\phi(\omega)
 \bigl[g_t^{N_j}(d\omega)-g_t(d\omega)\bigr]\right|
 \to 0.
 \label{eq:cond-normalized-convergence}
\end{align}
Moreover, mass and energy conservation in \cref{prop:integrated-cutoff-existence} 
give
\begin{equation}\label{eq:cond-normalized-moment-bound}
 \int_{[0,\infty)}(1+\omega) 
 \widetilde g_\tau^{N_j}(d\omega)
 =1+\frac{\mathcal E_0}{\mathcal N_0},
 \qquad \tau\ge0,\quad j\ge1.
\end{equation}
Since each $\widetilde g^{N_j}$ satisfies \eqref{eq:cutoff-weak}
with initial measure $\widetilde g_0$, the passage-to-the-limit argument in Step~3 of the proof
of \cref{thm:measure-main} applies to this sequence and shows
that $\widetilde g$ satisfies \eqref{eq:weak}.
Consequently, $\widetilde g$ is a solution constructed in
\cref{sec:measure} from $\widetilde g_0$.

Suppose that \cref{thm:smooth-condensation} for \(\mathcal N_0=1\) gives
$\widetilde g_\tau(\{0\})\ge1-\zeta$ for all
$\tau\ge\widetilde T$. Returning to the original variables yields
\begin{equation}\label{eq:cond-normalization-condensate}
 g_t(\{0\})
 =\mathcal N_0\widetilde g_{\mathcal N_0^2t}(\{0\})
 \ge(1-\zeta)\mathcal N_0,
 \qquad t\ge\widetilde T/\mathcal N_0^2.
\end{equation}
Besides,
\begin{equation}\label{eq:cond-normalization-tv}
 \|g_t-\mathcal N_0\delta_0\|_{\mathrm{TV}}
 =\mathcal N_0
 \|\widetilde g_{\mathcal N_0^2t}-\delta_0\|_{\mathrm{TV}},
\end{equation}
and
\begin{equation}\label{eq:cond-normalization-tv-equivalence}
    \|g_t-\mathcal N_0\delta_0\|_{\rm TV}\to 0\iff\|\widetilde{g_t}-\delta_0\|_{\rm TV}\to 0.
\end{equation}
It therefore suffices to prove the assertions for \(\mathcal N_0=1\).
\subsection{Positive density on an interval}
\label{subsec:cond-positive-interval}

\begin{lemma}
\label{lem:cond-ac-initial-continuity}
Suppose that $\mathcal N_0=1$ and $g_0$ has a nonzero absolutely
continuous part. There exist $r,m>0$ and a measurable set
$E\subset[r,5r/4]$ with $|E|>0$ such that
$f_N(0,\cdot)\ge m\mathbf1_E$ for every $N\ge1$.
With $w_N(t,\omega)=f_N(t,\omega)\wedge m\mathbf1_E(\omega)$,
\[
 \lim_{\tau\downarrow0}\sup_{N\ge1}\sup_{0\le t\le\tau}
 \|w_N(t,\cdot)-m\mathbf1_E\|_{L^1(\mathbb R)}=0.
\]
\end{lemma}
\begin{proof}
\noindent\textit{Step 1. A fixed frequency interval.}
By \eqref{eq:cond-density-definition}, $f_N(0,\cdot)$ is independent of
$N$ and
\begin{equation}\label{eq:cond-ac-positive-initial-mass}
 \int_0^\infty \sqrt\omega\,f_N(0,\omega) d\omega>0.
\end{equation}
Since $(0,\infty)=\bigcup_{j\in\mathbb Z}[(5/4)^j,(5/4)^{j+1}]$,
there are $r,m>0$ such that
\begin{equation}\label{eq:cond-ac-initial-set}
 E=\{\omega\in[r,5r/4]:f_N(0,\omega)\ge m\},
 \qquad |E|>0,
 \qquad f_N(0,\omega)\ge m\mathbf1_E(\omega).
\end{equation}
Define
\begin{equation}\label{eq:cond-ac-fixed-operator}
 L_N^r(t)[\phi](\omega_4)
 =\frac12\iint_{(0,r/8)^2}
 \frac{S_{\beta,N}(\omega_1,\omega_2,\omega_3,\omega_4)}
      {\sqrt{\omega_2\omega_3\omega_4}}
 [\phi(\omega_1)-\phi(\omega_4)]
  g_t^N(d\omega_2)g_t^N(d\omega_3).
\end{equation}
For $\omega_4\in(r/2,2r)$, discarding $\mathcal G_N\ge0$ in
\eqref{eq:cond-output-gain} and using \eqref{eq:cond-loss-median} give
\begin{align}
 &\partial_t f_N(t,\omega_4)
 \ge L_N(t)[f_N(t,\cdot)](\omega_4)
       +V_N(t,\omega_4)f_N(t,\omega_4)\notag\\
 &\quad= L_N^r(t)[f_N(t,\cdot)](\omega_4)\notag\\*
 &\qquad+\frac12
 \iint_{(0,\omega_4/2)^2\setminus(0,r/8)^2}
 \frac{S_{\beta,N}}{\sqrt{\omega_2\omega_3\omega_4}}
 f_N(t,\omega_1)g_t^N(d\omega_2)g_t^N(d\omega_3)\notag\\*
 &\qquad+\frac{f_N(t,\omega_4)}4
 \iint_{(0,\infty)^2}
 \frac{S_{\beta,N}(\omega_1,\omega_2,\omega_1+\omega_2-\omega_4,\omega_4)}
      {\sqrt{\omega_1\omega_2\omega_4}}
 g_t^N(d\omega_1)g_t^N(d\omega_2)\notag\\*
 &\qquad-\frac{f_N(t,\omega_4)}2
 \iint_{(0,\omega_4/2)^2\setminus(0,r/8)^2}
 \frac{S_{\beta,N}}{\sqrt{\omega_2\omega_3\omega_4}}
 g_t^N(d\omega_2)g_t^N(d\omega_3)\notag\\*
 &\qquad-\frac{f_N(t,\omega_4)}2
 \iint_{(0,\infty)^2\setminus(0,\omega_4/2)^2}
 \frac{S_{\beta,N}}{\sqrt{\omega_2\omega_3\omega_4}}
 g_t^N(d\omega_2)g_t^N(d\omega_3)\notag\\
 &\quad\ge L_N^r(t)[f_N(t,\cdot)](\omega_4)\notag\\*
 &\qquad-\frac{f_N(t,\omega_4)}2
 \iint_{(0,\infty)^2\setminus(0,r/8)^2}
 \frac{S_{\beta,N}}{\sqrt{\omega_2\omega_3\omega_4}}
 g_t^N(d\omega_2)g_t^N(d\omega_3)\notag\\
 &\quad\ge L_N^r(t)[f_N(t,\cdot)](\omega_4)-2^{\beta-1}(8/r)^d[g_t^N((0,\infty))]^2
 f_N(t,\omega_4)\notag\\
 &\quad\ge L_N^r(t)[f_N(t,\cdot)](\omega_4)
 -2^{\beta-1}(8/r)^d f_N(t,\omega_4).
 \label{eq:cond-ac-local-evolution}
\end{align}

\smallskip
\noindent\textit{Step 2. Two estimates for $L_N^r(t)$.}
Let $0\le\phi\in L^1((0,\infty),\sqrt\omega d\omega)$.
By the symmetry of $S_{\beta,N}$,
\begin{equation}\label{eq:cond-ac-weighted-gain-loss}
 \begin{aligned}
 &\int_0^\infty\iint_{(0,r/8)^2}
 \frac{S_{\beta,N}(\omega_3+\omega_4-\omega_2,
                   \omega_2,\omega_3,\omega_4)}
      {\sqrt{\omega_2\omega_3}}
 \phi(\omega_3+\omega_4-\omega_2)
  g_t^N(d\omega_2)g_t^N(d\omega_3)d\omega_4\\
 &=\iint_{(0,r/8)^2}\int_{\omega_3-\omega_2}^{\infty}
 \frac{S_{\beta,N}(\omega_1,\omega_2,\omega_3,
                   \omega_1+\omega_2-\omega_3)}
      {\sqrt{\omega_2\omega_3}}
 \phi(\omega_1) d\omega_1 
 g_t^N(d\omega_2)g_t^N(d\omega_3)\\
 &=\iint_{(0,r/8)^2}\int_0^{\infty}
 \frac{S_{\beta,N}(\omega_1,\omega_3,\omega_2,
                   \omega_1+\omega_3-\omega_2)}
      {\sqrt{\omega_2\omega_3}}
 \phi(\omega_1) d\omega_1 
 g_t^N(d\omega_2)g_t^N(d\omega_3)\\
 &=\int_0^\infty\iint_{(0,r/8)^2}
 \frac{S_{\beta,N}(\omega_3+\omega_4-\omega_2,
                   \omega_2,\omega_3,\omega_4)}
      {\sqrt{\omega_2\omega_3}}
 \phi(\omega_4)
  g_t^N(d\omega_2)g_t^N(d\omega_3)d\omega_4.
 \end{aligned}
\end{equation}
By \eqref{eq:cond-cutoff-basic-bounds}, their common value satisfies
\begin{equation}\label{eq:cond-ac-fixed-loss-bound}
 \begin{aligned}
 &\int_0^\infty\iint_{(0,r/8)^2}
 \frac{S_{\beta,N}(\omega_3+\omega_4-\omega_2,
                   \omega_2,\omega_3,\omega_4)}
      {\sqrt{\omega_2\omega_3}}
 \phi(\omega_4)
 g_t^N(d\omega_2)g_t^N(d\omega_3)d\omega_4\\
 &\quad\le N(2N)^{3/2}[g_t^N((0,r/8))]^2
 \int_0^\infty\sqrt{\omega_4}\phi(\omega_4)d\omega_4\\
 &\quad\le N(2N)^{3/2}
 \int_0^\infty\sqrt{\omega_4}\phi(\omega_4)d\omega_4<\infty.
 \end{aligned}
\end{equation}
Thus both integrals are finite, and
\begin{equation}\label{eq:cond-ac-weighted-cancellation}
 \begin{aligned}
 &\int_0^\infty\sqrt{\omega_4}
 L_N^r(t)[\phi](\omega_4)d\omega_4\\
 &=\frac12\int_0^\infty\iint_{(0,r/8)^2}
 \frac{S_{\beta,N}(\omega_1,\omega_2,\omega_3,\omega_4)}
      {\sqrt{\omega_2\omega_3}}
 [\phi(\omega_1)-\phi(\omega_4)]
 g_t^N(d\omega_2)g_t^N(d\omega_3)d\omega_4\\
 &=\frac12\int_0^\infty\iint_{(0,r/8)^2}
 \frac{S_{\beta,N}(\omega_1,\omega_2,\omega_3,\omega_4)}
      {\sqrt{\omega_2\omega_3}}
 \phi(\omega_1)g_t^N(d\omega_2)g_t^N(d\omega_3)d\omega_4\\
 &\quad-\frac12\int_0^\infty\iint_{(0,r/8)^2}
 \frac{S_{\beta,N}(\omega_1,\omega_2,\omega_3,\omega_4)}
      {\sqrt{\omega_2\omega_3}}
 \phi(\omega_4)g_t^N(d\omega_2)g_t^N(d\omega_3)d\omega_4\\
 &=0.
 \end{aligned}
\end{equation}

Now take $\phi\in C_c^\infty((r/2,2r))$. Since
$\omega_1-\omega_4=\omega_3-\omega_2$,
\begin{equation}\label{eq:cond-ac-translation-estimate}
 \int_{\mathbb R}|\phi(\omega_1)-\phi(\omega_4)| d\omega_4
 \le |\omega_3-\omega_2|\,\|\phi'\|_{L^1(\mathbb R)}.
\end{equation}
For $(\omega_2,\omega_3)\in(0,r/8)^2$, set
$(x,y,z)=\operatorname{sort}(\omega_2,\omega_3,\omega_4)$.
If $\phi(\omega_1)-\phi(\omega_4)\ne0$, at least one of
$\omega_1,\omega_4$ lies in $(r/2,2r)$.
Since $|\omega_1-\omega_4|=|\omega_3-\omega_2|<r/8$, they\ 
both lie in
$(r/2-r/8,2r+r/8)=(3r/8,17r/8)$.
Hence $z=\omega_4$, $x=\omega_2\wedge\omega_3$ and
$y=\omega_2\vee\omega_3$ on the support. In particular,
$x,y$ are independent of $\omega_4$ if \((\omega_2,\omega_3)\in (0,r/8)^2\).
By \eqref{eq:cutoff} and the definition of $S_\beta$,
\begin{equation}\label{eq:cond-ac-translation-kernel}
 \begin{gathered}
 0<x\le y<r/8<3r/8<\omega_1\wedge z,
 \qquad \min\{\omega_1,\omega_2,\omega_3,\omega_4\}=x,\\[0.7ex]
 \begin{aligned}
 \frac{S_{\beta,N}(\omega_1,\omega_2,\omega_3,\omega_4)}
      {\sqrt{\omega_2\omega_3}}
 &\le\frac{(\omega_1\omega_4xy)^\beta\sqrt{x}}{\sqrt{xy}}
 \le(17r/8)^{2\beta}x^\beta y^{\beta-1/2}
 \le(3r)^{2\beta}x^\beta y^{\beta-1/2},\\[0.7ex]
 x^\beta y^{\beta-1/2}|\omega_3-\omega_2|
 &=x^\beta y^{\beta-1/2}(y-x)\le x^\beta y^{\beta+1/2}
 \le y^{2\beta+1/2}
 \le(r/8)^{2\beta+1/2}.
 \end{aligned}
 \end{gathered}
\end{equation}
Consequently,
\begin{align}
 &\int_0^\infty\sqrt{\omega_4}
 |L_N^r(t)[\phi](\omega_4)|d\omega_4\notag\\*
 &\quad\le\frac12\int_0^\infty\iint_{(0,r/8)^2}
 \frac{S_{\beta,N}(\omega_1,\omega_2,\omega_3,\omega_4)}
      {\sqrt{\omega_2\omega_3}}
 |\phi(\omega_1)-\phi(\omega_4)|
 g_t^N(d\omega_2)g_t^N(d\omega_3)d\omega_4\notag\\
 &\quad=\frac12\iint_{(0,r/8)^2}\int_0^\infty
 \frac{S_{\beta,N}(\omega_1,\omega_2,\omega_3,\omega_4)}
      {\sqrt{\omega_2\omega_3}}
 |\phi(\omega_1)-\phi(\omega_4)|d\omega_4
 g_t^N(d\omega_2)g_t^N(d\omega_3)\notag\\
 &\quad\le\frac{(3r)^{2\beta}}2
 \iint_{(0,r/8)^2}x^\beta y^{\beta-1/2}
 \left(\int_0^\infty
 |\phi(\omega_1)-\phi(\omega_4)|d\omega_4\right)
 g_t^N(d\omega_2)g_t^N(d\omega_3)\notag\\
 &\quad\le\frac{(3r)^{2\beta}}2
 \iint_{(0,r/8)^2}x^\beta y^{\beta-1/2}
 \left(\int_{\mathbb R}
 |\phi(\omega_1)-\phi(\omega_4)|d\omega_4\right)
 g_t^N(d\omega_2)g_t^N(d\omega_3)\notag\\
 &\quad\le\frac{(3r)^{2\beta}}2\|\phi'\|_{L^1(\mathbb R)}
 \iint_{(0,r/8)^2}x^\beta y^{\beta-1/2}
 |\omega_3-\omega_2|g_t^N(d\omega_2)g_t^N(d\omega_3)\notag\\
 &\quad\le\frac{(3r)^{2\beta}}2(r/8)^{2\beta+1/2}
 \|\phi'\|_{L^1(\mathbb R)}
 \iint_{(0,r/8)^2}g_t^N(d\omega_2)g_t^N(d\omega_3)\notag\\
 &\quad=\frac{(3r)^{2\beta}}2(r/8)^{2\beta+1/2}
 [g_t^N((0,r/8))]^2\|\phi'\|_{L^1(\mathbb R)}
 \le\frac{(3r)^{2\beta}}2(r/8)^{2\beta+1/2}
 \|\phi'\|_{L^1(\mathbb R)}.
 \label{eq:cond-ac-smooth-bound}
\end{align}

\smallskip
\noindent\textit{Step 3. Uniform continuity at the initial time.}
Set
\begin{equation}\label{eq:cond-ac-bounded-density}
 w_N(t,\omega)=f_N(t,\omega) \wedge m\mathbf1_E(\omega).
\end{equation}
Then $0\le w_N\le m\mathbf1_E$ and
$w_N(0,\omega)=m\mathbf1_E(\omega)$. Fix
$0\le\phi\in C_c^\infty((r/2,2r))$. By \eqref{eq:cond-ac-fixed-operator},
\begin{equation}\label{eq:cond-ac-positive-part}
 \mathbf1_{\{\phi>f_N\}}L_N^r(t)[\phi-f_N]
 \le L_N^r(t)[(\phi-f_N)_+]
 \qquad\text{a.e. on }(0,\infty).
\end{equation}
On $\{\phi>f_N\}$, we have $0\le f_N<\phi$, so
$\omega_4\in\operatorname{supp}\phi\Subset(r/2,2r)$ and
\eqref{eq:cond-ac-local-evolution} applies.
Using the time representatives and positive-part chain rule in
\eqref{eq:cond-comparison-time-representative}--\eqref{eq:cond-comparison-positive-chain-rule} and \eqref{eq:cond-ac-positive-part}, 
for a.e.\ $(t,\omega_4)$,
\begin{equation}\label{eq:cond-ac-positive-part-differential}
 \begin{aligned}
 \partial_t(\phi-f_N)_+
 &=\mathbf1_{\{\phi>f_N\}}\partial_t(\phi-f_N)
 =-\mathbf1_{\{\phi>f_N\}}\partial_t f_N\\
 &\le-\mathbf1_{\{\phi>f_N\}}L_N^r(t)[f_N]
       +2^{\beta-1}(8/r)^d\mathbf1_{\{\phi>f_N\}}f_N\\
 &=\mathbf1_{\{\phi>f_N\}}L_N^r(t)[\phi-f_N]
   +\mathbf1_{\{\phi>f_N\}}
     \bigl(-L_N^r(t)[\phi]+2^{\beta-1}(8/r)^d f_N\bigr)\\
 &\le L_N^r(t)[(\phi-f_N)_+]
   +\mathbf1_{\{\phi>f_N\}}
     \bigl(|L_N^r(t)[\phi]|+2^{\beta-1}(8/r)^d\phi\bigr)\\
 &\le L_N^r(t)[(\phi-f_N)_+]
       +|L_N^r(t)[\phi]|+2^{\beta-1}(8/r)^d\phi.
 \end{aligned}
\end{equation}
The same chain rule and Fubini's theorem give, for a.e.\ $t\ge0$,
\begin{equation}\label{eq:cond-ac-positive-part-time}
 \begin{aligned}
 &\frac{d}{dt}\int_0^\infty\sqrt{\omega_4}
 (\phi-f_N)_+(t,\omega_4)d\omega_4\\
 &=\int_0^\infty\sqrt{\omega_4}
 \partial_t(\phi-f_N)_+(t,\omega_4)d\omega_4\\
 &\le\int_0^\infty\sqrt{\omega_4}
 L_N^r(t)[(\phi-f_N)_+](\omega_4)d\omega_4\\
 &\quad+\int_0^\infty\sqrt{\omega_4}
 |L_N^r(t)[\phi](\omega_4)|d\omega_4
 +2^{\beta-1}(8/r)^d
 \int_0^\infty\sqrt{\omega_4}\phi(\omega_4)d\omega_4\\
 &=\int_0^\infty\sqrt{\omega_4}
 |L_N^r(t)[\phi](\omega_4)|d\omega_4
 +2^{\beta-1}(8/r)^d
 \int_0^\infty\sqrt{\omega_4}\phi(\omega_4)d\omega_4\\
 &\le\frac{(3r)^{2\beta}}2(r/8)^{2\beta+1/2}
 \|\phi'\|_{L^1(\mathbb R)}+2^{\beta-1}(8/r)^d
 \int_0^\infty\sqrt{\omega_4}\phi(\omega_4)d\omega_4.
 \end{aligned}
\end{equation}
The second equality uses \eqref{eq:cond-ac-weighted-cancellation}, and the
last inequality uses \eqref{eq:cond-ac-smooth-bound}.
Since $w_N=f_N\wedge m\mathbf1_E$ and
$f_N(0,\cdot)\ge m\mathbf1_E$, we have
\begin{equation}\label{eq:cond-ac-truncation-bridge}
 \begin{aligned}
 0\le m\mathbf1_E-w_N(t,\cdot)
 &=(m\mathbf1_E-f_N(t,\cdot))_+\le|m\mathbf1_E-\phi|+(\phi-f_N(t,\cdot))_+,\\
 (\phi-f_N(0,\cdot))_+
 &\le(\phi-m\mathbf1_E)_+\le|\phi-m\mathbf1_E|.
 \end{aligned}
\end{equation}
Since $m\mathbf1_E-w_N$ vanishes outside $E\subset[r,5r/4]$,
using \eqref{eq:cond-ac-truncation-bridge} and integrating
\eqref{eq:cond-ac-positive-part-time} from $0$ to $t$, we obtain
\begin{equation}\label{eq:cond-ac-initial-continuity-estimate}
 \begin{aligned}
 &\sqrt r\,\|w_N(t,\cdot)-m\mathbf1_E\|_{L^1(\mathbb R)}
 =\sqrt r\int_E[m-w_N(t,\omega)]d\omega\\
 &\quad\le\int_0^\infty\sqrt\omega
 [m\mathbf1_E(\omega)-w_N(t,\omega)]d\omega\\
 &\quad\le\int_0^\infty\sqrt\omega
 |m\mathbf1_E(\omega)-\phi(\omega)|d\omega
 +\int_0^\infty\sqrt\omega
 (\phi(\omega)-f_N(t,\omega))_+d\omega\\
 &\quad\le\int_0^\infty\sqrt\omega
 |m\mathbf1_E(\omega)-\phi(\omega)|d\omega
 +\int_0^\infty\sqrt\omega
 (\phi(\omega)-f_N(0,\omega))_+d\omega\\
 &\qquad+\int_0^t\left[
 \frac{(3r)^{2\beta}}2(r/8)^{2\beta+1/2}
 \|\phi'\|_{L^1(\mathbb R)}
 +2^{\beta-1}(8/r)^d\int_0^\infty\sqrt\omega\,\phi(\omega)d\omega
 \right]ds\\
 &\quad\le2\int_0^\infty\sqrt\omega
 |m\mathbf1_E(\omega)-\phi(\omega)|d\omega\\
 &\qquad+t\left[
 \frac{(3r)^{2\beta}}2(r/8)^{2\beta+1/2}
 \|\phi'\|_{L^1(\mathbb R)}
 +2^{\beta-1}(8/r)^d\int_0^\infty\sqrt\omega\,\phi(\omega)d\omega
 \right].
 \end{aligned}
\end{equation}

Let $\varepsilon>0$. Since $E\subset[r,5r/4]\Subset(r/2,2r)$,
choose $0\le\phi\in C_c^\infty((r/2,2r))$ with
$\|m\mathbf1_E-\phi\|_{L^1(\mathbb R)}<\varepsilon$.
Keep this $\phi$ fixed. Since the bound
\eqref{eq:cond-ac-initial-continuity-estimate} is uniform in $N$,
\begin{equation}\label{eq:cond-ac-approximation-limit}
 \begin{aligned}
 &\limsup_{\tau\downarrow0}\sup_{N\ge1}\sup_{0\le t\le\tau}
 \|w_N(t,\cdot)-m\mathbf1_E\|_{L^1(\mathbb R)}\\
 &\quad\le\frac2{\sqrt r}\int_{r/2}^{2r}\sqrt\omega
 |m\mathbf1_E(\omega)-\phi(\omega)|d\omega\\
 &\quad\le2\sqrt2\int_{r/2}^{2r}
 |m\mathbf1_E(\omega)-\phi(\omega)|d\omega
 =2\sqrt2\|m\mathbf1_E-\phi\|_{L^1(\mathbb R)}
 <2\sqrt2\varepsilon.
 \end{aligned}
\end{equation}
Since $\varepsilon>0$ is arbitrary,
\begin{equation}\label{eq:cond-ac-uniform-continuity}
 \lim_{\tau\downarrow0}\sup_{N\ge1}\sup_{0\le t\le\tau}
 \|w_N(t,\cdot)-m\mathbf1_E\|_{L^1(\mathbb R)}=0.
\end{equation}
\end{proof}

\begin{lemma}[Positive density on an interval]\label{lem:cond-ac-interval}
Suppose that $\mathcal N_0=1$ and $g_0$ has a nonzero absolutely
continuous part. There exist $s,c>0$ and
$0<a<\alpha_0<\gamma_0<b$, depending only on $\beta,g_0$, such that,
for all sufficiently large $N$,
\[
 \mathcal G_N(t,\omega_4)\ge c,
 \qquad 0\le t\le s,\quad a\le\omega_4\le b,
\]
and
\[
 f_N(s,\omega_4)\ge c,
 \qquad \alpha_0\le\omega_4\le\gamma_0.
\]
\end{lemma}
\begin{proof}
Take $r,m,E,w_N$ from \cref{lem:cond-ac-initial-continuity}.

\noindent\textit{Step 1. Positivity of the triple convolution.}
For $\omega_3=\omega_1+\omega_2-\omega_4$, define
\begin{equation}\label{eq:cond-ac-initial-convolution}
 H(\omega_4)=m^3\iint_{(0,\infty)^2}
 \mathbf1_E(\omega_1)\mathbf1_E(\omega_2)\mathbf1_E(\omega_3)
  d\omega_1d\omega_2.
\end{equation}
Since $E\subset[r,5r/4]$,
\begin{equation}\label{eq:cond-ac-convolution-support}
 \operatorname{supp}H\subset[r+r-5r/4,5r/4+5r/4-r]=[3r/4,3r/2].
\end{equation}
For every $\omega_4,h\in\mathbb R$, we have
\begin{equation}\label{eq:cond-ac-convolution-continuity}
 \begin{aligned}
 &|H(\omega_4+h)-H(\omega_4)|\\
 &\quad=m^3\left|\int_E\int_E
 [\mathbf1_E(\omega_1+\omega_2-\omega_4-h)
 -\mathbf1_E(\omega_1+\omega_2-\omega_4)]d\omega_2d\omega_1\right|\\
 &\quad\le m^3\int_E\int_E
 |\mathbf1_E(\omega_1+\omega_2-\omega_4-h)
 -\mathbf1_E(\omega_1+\omega_2-\omega_4)|d\omega_2d\omega_1\\
 &\quad\le m^3\int_E\int_{\mathbb R}
 |\mathbf1_E(\omega_1+\omega_2-\omega_4-h)
 -\mathbf1_E(\omega_1+\omega_2-\omega_4)|d\omega_2d\omega_1\\
 &\quad=m^3\int_E\int_{\mathbb R}
 |\mathbf1_E(\omega_3-h)-\mathbf1_E(\omega_3)|d\omega_3d\omega_1\\
 &\quad=m^3|E|\,
 \|\mathbf1_E(\cdot-h)-\mathbf1_E\|_{L^1(\mathbb R)}
 \longrightarrow0,\qquad\text{ as }h\to0,
 \end{aligned}
\end{equation}
by the translation continuity in $L^1(\mathbb R)$.
Hence $H$ is continuous.
Tonelli's theorem also gives
\begin{equation}\label{eq:cond-ac-convolution-mass}
 \int_{\mathbb R}H(\omega_4) d\omega_4=m^3|E|^3>0.
\end{equation}
Hence there are
$3r/4<a<b<3r/2$ and $c_0>0$ such that
\begin{equation}\label{eq:cond-ac-positive-convolution-interval}
 H(\omega_4)\ge2c_0,
 \qquad a\le\omega_4\le b.
\end{equation}
By \eqref{eq:cond-ac-uniform-continuity}, choose $s>0$ such that
\begin{equation}\label{eq:cond-ac-time-choice}
 3m^2|E|\sup_{N\ge1}\sup_{0\le t\le s}
 \|w_N(t,\cdot)-m\mathbf1_E\|_{L^1(\mathbb R)}\le c_0.
\end{equation}
Since
\begin{equation}\label{eq:cond-ac-convolution-difference}
 \begin{aligned}
 &m^3\mathbf1_E(\omega_1)\mathbf1_E(\omega_2)\mathbf1_E(\omega_3)
 -w_N(t,\omega_1)w_N(t,\omega_2)w_N(t,\omega_3)\\
 &\quad=m^2[m\mathbf1_E(\omega_1)-w_N(t,\omega_1)]
 \mathbf1_E(\omega_2)\mathbf1_E(\omega_3)\\
 &\qquad+m w_N(t,\omega_1)
 [m\mathbf1_E(\omega_2)-w_N(t,\omega_2)]\mathbf1_E(\omega_3)\\
 &\qquad+w_N(t,\omega_1)w_N(t,\omega_2)
 [m\mathbf1_E(\omega_3)-w_N(t,\omega_3)],
 \end{aligned}
\end{equation}
$0\le w_N\le m\mathbf1_E\le m$ gives
\begin{equation}\label{eq:cond-ac-convolution-error-bounds}
 m^2\iint_{E^2}[m-w_N(t,\omega_1)]\mathbf1_E(\omega_3)d\omega_1d\omega_2
 \le m^2|E|\|m\mathbf1_E-w_N(t)\|_{L^1},
\end{equation}
\begin{equation}\label{eq:cond-ac-convolution-error-second}
 m\iint_{E^2}w_N(t,\omega_1)[m-w_N(t,\omega_2)]
 \mathbf1_E(\omega_3)d\omega_1d\omega_2
 \le m^2|E|\|m\mathbf1_E-w_N(t)\|_{L^1},
\end{equation}
\begin{equation}\label{eq:cond-ac-convolution-error-third}
 \begin{aligned}
 &\iint_{E^2}w_N(t,\omega_1)w_N(t,\omega_2)
 [m\mathbf1_E(\omega_3)-w_N(t,\omega_3)]d\omega_1d\omega_2\\
 &\quad\le m^2\int_E\int_{\mathbb R}
 [m\mathbf1_E(\omega_3)-w_N(t,\omega_3)]d\omega_3d\omega_1
 =m^2|E|\|m\mathbf1_E-w_N(t)\|_{L^1}.
 \end{aligned}
\end{equation}
Thus \eqref{eq:cond-ac-positive-convolution-interval} and
\eqref{eq:cond-ac-time-choice} imply
\begin{equation}\label{eq:cond-ac-convolution-lower}
 \begin{aligned}
 &\iint_{(0,\infty)^2}
 w_N(t,\omega_1)w_N(t,\omega_2)w_N(t,\omega_3)d\omega_1d\omega_2\\
 &\quad\ge H(\omega_4)
 -3m^2|E|\|w_N(t)-m\mathbf1_E\|_{L^1}\ge2c_0-c_0=c_0,
 \end{aligned}
\end{equation}
for $0\le t\le s$ and $a\le\omega_4\le b$.

\smallskip
\noindent\textit{Step 2. Comparison on an interval.}
Take
\begin{equation}\label{eq:cond-ac-output-cutoff-threshold}
 N\ge4/r \vee (3r)^{4\beta+1/2}.
\end{equation}
On the support of \(w_N(t,\omega_1)w_N(t,\omega_2)w_N(t,\omega_3)\),
$\omega_1,\omega_2,\omega_3\in[r,5r/4]$ and
$\omega_4\in[a,b]\subset[3r/4,3r/2]$. Hence \(S_\beta=S_{\beta,N}\) and
\begin{equation}\label{eq:cond-ac-output-kernel-lower}
 \begin{aligned}
 \frac{S_{\beta,N}(\omega_1,\omega_2,\omega_3,\omega_4)}
      {4\sqrt{\omega_4}}
 =\frac{(\omega_1\omega_2\omega_3\omega_4)^\beta}{4}
 \sqrt{\frac{\min\{\omega_1,\omega_2,\omega_3,\omega_4\}}{\omega_4}}\ge\frac{(3r/4)^{4\beta}}4\sqrt{\frac{3r/4}{3r/2}}
 =\frac{(3r/4)^{4\beta}}{4\sqrt2}.
 \end{aligned}
\end{equation}
Since $f_N\ge w_N$ and
$\omega_1,\omega_2\ge r>\omega_4/2$ on this support,
\eqref{eq:cond-ac-convolution-lower} yields
\begin{equation}\label{eq:cond-ac-output-source}
 \mathcal G_N(t,\omega_4)
 \ge\frac{(3r/4)^{4\beta}}{4\sqrt2}
 \iint_{(0,\infty)^2}
 w_N(t,\omega_1)w_N(t,\omega_2)w_N(t,\omega_3)d\omega_1d\omega_2
 \ge\frac{c_0(3r/4)^{4\beta}}{4\sqrt2},
\end{equation}
for $0\le t\le s$ and $a\le\omega_4\le b$.

Let $u,C$ be given by \eqref{eq:cond-sine-profile} and
\eqref{eq:cond-profile-constant}.
Since $0\le u\le\mathbf1_{[a,b]}$, substituting
\eqref{eq:cond-ac-output-source} into \eqref{eq:cond-density-forced} gives
\begin{equation}\label{eq:cond-ac-forced-evolution}
 \begin{aligned}
 \partial_t f_N(t,\omega_4)
 \ge L_N(t)[f_N(t,\cdot)](\omega_4)
 -\omega_4^{-d}\ell_N(\omega_4)f_N(t,\omega_4)+\frac{c_0(3r/4)^{4\beta}}{4\sqrt2}u(\omega_4),
 \qquad 0\le t\le s.
 \end{aligned}
\end{equation}
Because $2/a<4/r$ and $3b/2<3r$,
\cref{lem:cond-sine-comparison}\textup{(ii)} applies on $[0,s]$ with $D=0$.
Evaluating \eqref{eq:cond-forced-sine} at $t=s$ gives
\begin{equation}\label{eq:cond-ac-sine-lower}
 \begin{aligned}
 f_N(s,\omega_4)
 \ge\frac{c_0(3r/4)^{4\beta}}{4\sqrt2}
 u(\omega_4)\int_0^s e^{-C(s-\tau)}d\tau=\frac{c_0(3r/4)^{4\beta}}{4\sqrt2\,C}
 (1-e^{-Cs})u(\omega_4).
 \end{aligned}
\end{equation}
Set
\begin{equation}\label{eq:cond-ac-interior-output-interval}
 a<\alpha_0=\frac{3a+b}{4}
 <\gamma_0=\frac{a+3b}{4}<b.
\end{equation}
Then $u(\omega_4)\ge1/\sqrt2$ for $\omega_4\in[\alpha_0,\gamma_0]$, and therefore
\begin{equation}\label{eq:cond-ac-final-lower}
 f_N(s,\omega_4)
 \ge\frac{c_0(3r/4)^{4\beta}}{8C}(1-e^{-Cs})=:c>0,
 \qquad\alpha_0\le\omega_4\le\gamma_0.
\end{equation}
Replace $c$ by $\min\{c,c_0(3r/4)^{4\beta}/(4\sqrt2)\}$.
Both conclusions follow from \eqref{eq:cond-ac-output-source} and
\eqref{eq:cond-ac-final-lower}. The constants
$s,a,b,\alpha_0,\gamma_0,c$ are independent of $N$.
\end{proof}

\subsection{Positive density at small frequencies}
\label{subsec:cond-density}

\begin{lemma}[Expansion of an interval of positive density]
\label{lem:cond-interval-expansion}
Let $\mathcal N_0=1$, $s\ge0$, $\omega_*>0$,
$0<w<3\omega_*/4$, and $A,D>0$. Suppose that, for all sufficiently
large $N$,
\[
 f_N(t,\omega_4)\ge Ae^{-Dt},
 \qquad t\ge s,\quad |\omega_4-\omega_*|\le w.
\]
Set
\[
 a=(\omega_*-5w/2)\vee(\omega_*/8),\qquad
 b=2\omega_*-a,\qquad
 c_w=\frac{w^2a^{4\beta+1/2}}{144\sqrt b}.
\]
There exists $A'>0$, independent of $N$, such that, for all
sufficiently large $N$,
\[
 \mathcal G_N(t,\omega_4)\ge c_wA^3e^{-3Dt},
 \qquad t\ge s,\quad a\le\omega_4\le b,
\]
and
\[
 f_N(t,\omega_4)\ge A'e^{-3Dt},
 \qquad t\ge s+1,\quad
 |\omega_4-\omega_*|\le2w\wedge(3\omega_*/4).
\]
\end{lemma}
\begin{proof}
Suppose that, for some $s\ge0$, $0<w<3\omega_*/4$, and $A,D>0$
independent of $N$, all sufficiently large $N$ satisfy
\begin{equation}\label{eq:cond-density-lower-expansion-assumption}
 f_N(t,\omega_4)\ge Ae^{-Dt},
 \qquad t\ge s,\quad
 \omega_4\in I=[\omega_*-w,\omega_*+w].
\end{equation}
For all sufficiently large $N$, we will show that
$f_N(t,\omega_4)\ge A'e^{-3Dt}$ for $t\ge s+1$ and
$|\omega_4-\omega_*|\le2w\wedge(3\omega_*/4)$, with $A'>0$
independent of $N$.
Set $a=(\omega_*-5w/2)\vee(\omega_*/8)$ and $b=2\omega_*-a$.
Thus $I\subset(a,b)$ and $|\omega_4-\omega_*|\le5w/2$ for
$\omega_4\in[a,b]$.

Fix $\omega_4\in[a,b]$. We seek a square of pairs
$(\omega_1,\omega_2)$ for which all three frequencies
$\omega_1,\omega_2,\omega_3=\omega_1+\omega_2-\omega_4$ remain in $I$.
Choose $q=(\omega_4+2\omega_*)/3$. At the center $(q,q)$, the
three frequencies are $q,q,2q-\omega_4$, and
\[
 |q-\omega_*|=|2q-\omega_4-\omega_*|
 =\frac{|\omega_4-\omega_*|}{3}\le\frac{5w}{6}.
\]
Thus $q,2q-\omega_4\in[\omega_*-5w/6,\omega_*+5w/6]$.
If $|\omega_i-q|<w/12$ for $i=1,2$, then all three frequencies
$\omega_1,\omega_2,\omega_3$ lie in $I$. More precisely, set
\begin{equation}\label{eq:cond-density-lower-q-neighborhood}
 R_{\omega_4}=\left(q-\frac w{12},q+\frac w{12}\right)^2,
 \qquad |R_{\omega_4}|=\left(\frac w6\right)^2.
\end{equation}
On this square,
\[
 \begin{aligned}
 |\omega_i-\omega_*|&<\frac{5w}{6}+\frac w{12}
 =\frac{11w}{12}<w,\qquad i=1,2,\\
 |\omega_3-\omega_*|&<\frac{5w}{6}+\frac w{12}+\frac w{12}=w.
 \end{aligned}
\]
Moreover, $R_{\omega_4}\subset(\omega_4/2,\infty)^2$.
Indeed, $\omega_4\le\omega_*+5w/2$ gives
\[
 q-\frac w{12}-\frac{\omega_4}{2}
 =\frac{4\omega_*-\omega_4}{6}-\frac w{12}
 \ge\frac{\omega_*-w}{2}>0.
\]
One can refer to \cref{fig:cond-frequency-geometry}(a) to see that
$\omega_1,\omega_2,\omega_3\in I$ whenever
$(\omega_1,\omega_2)\in R_{\omega_4}$.

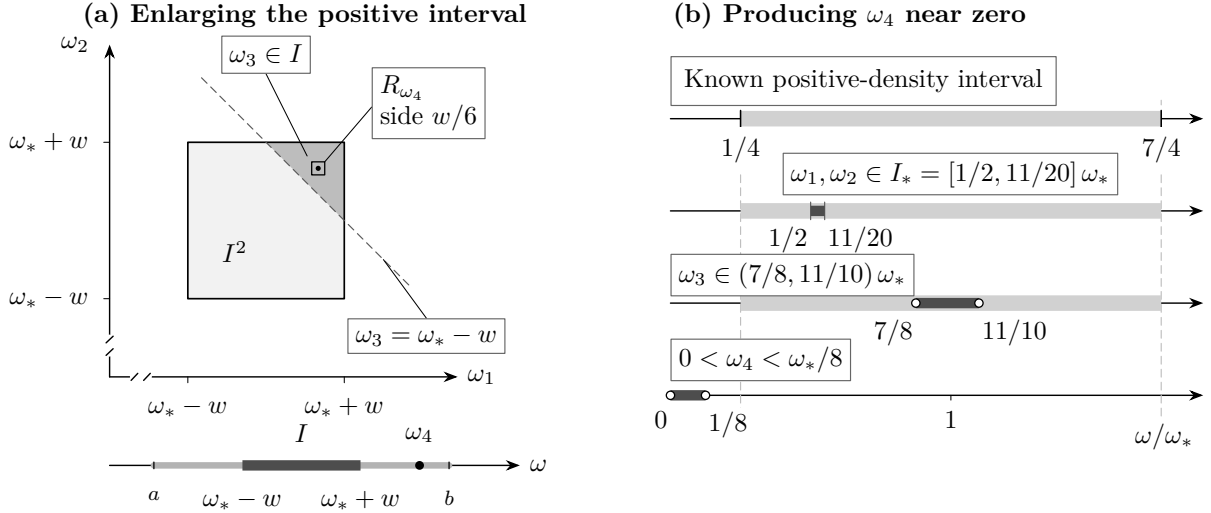
\begin{figure}[tbp]
\centering
\resizebox{\textwidth}{!}{%
\begin{tikzpicture}[
>=Stealth,line cap=round,line join=round,
  every node/.style={font=\small},
  axis/.style={->,line width=.55pt},
  guide/.style={densely dashed,line width=.45pt,black!65},
  note/.style={draw=black!60,line width=.45pt,fill=white,inner sep=5pt,align=left},
  band/.style={line width=4pt,line cap=butt,black!70}
]
  \node[anchor=west,font=\small\bfseries] at (0,6.25)
    {(a) Enlarging the positive interval};
  \node[anchor=west,font=\small\bfseries] at (8.0,6.25)
    {(b) Producing \(\omega_4\) near zero};

  % Representative geometry: omega_*=1, w=.2, omega_4=1.4.
  % Local coordinates subtract .6 from both frequency variables.
  \begin{scope}[shift={(.5,1.35)},x=5.3cm,y=5.3cm]
    \fill[black!5] (.2,.2) rectangle (.6,.6);
    \fill[black!26] (.4,.6)--(.6,.4)--(.6,.6)--cycle;
    \draw[line width=.65pt] (.2,.2) rectangle (.6,.6);
    \draw[axis] (0,0)--(.89,0) node[right] {$\omega_1$};
    \draw[axis] (0,0)--(0,.85) node[left=4pt] {$\omega_2$};
    % Breaks make clear that the displayed frequency axes are truncated.
    \draw[white,line width=2.5pt] (.06,0)--(.10,0) (0,.06)--(0,.10);
    \draw[line width=.55pt] (.055,-.009)--(.072,.009)
      (.087,-.009)--(.104,.009) (-.009,.055)--(.009,.072)
      (-.009,.087)--(.009,.104);
    \draw[guide] (.235,.765)--(.765,.235);
    \node at (.32,.32) {$I^2$};
    \draw[line width=.5pt] (.5166667,.5166667) rectangle (.55,.55);
    \fill (.5333333,.5333333) circle[radius=.0065];
    \draw[line width=.4pt] (.55,.55)--(.67,.69);
    \node[note,anchor=west,inner sep=3pt] at (.67,.70)
      {$R_{\omega_4}$\\side $w/6$};
    \draw[line width=.4pt] (.70,.30)--(.81,.15);
    \node[note,anchor=north,inner sep=3pt] at (.81,.15)
      {$\omega_3=\omega_*-w$};
    \foreach \t/\lab in {.2/{\omega_*-w},.6/{\omega_*+w}}{
      \draw (\t,0)--(\t,-.015) node[below=2pt] {$\lab$};
      \draw (0,\t)--(-.015,\t) node[left=2pt] {$\lab$};
    }
    \draw[line width=.4pt] (.4,.775)--(.505,.565);
    \node[note,anchor=south,inner sep=3pt] at (.4,.775)
      {$\omega_3\in I$};
  \end{scope}

  % The light band is the comparison interval [a,b]; here a=.5,b=1.5.
  % Local coordinates subtract .35 from the frequencies.
  \begin{scope}[shift={(.5,.15)},x=4cm,y=1cm]
    \draw[axis] (0,0)--(1.39,0) node[right] {$\omega$};
    \draw[line width=2.5pt,black!30] (.15,0)--(1.15,0);
    \draw[band] (.45,0)--(.85,0);
    \foreach \t in {.15,1.15}{\draw[line width=.5pt] (\t,-.045)--(\t,.045);}
    \fill (1.05,0) circle[radius=1.7pt];
    \node[below=6pt,font=\scriptsize] at (.15,0) {$a$};
    \node[below=6pt] at (.45,0) {$\omega_*-w$};
    \node[below=6pt] at (.85,0) {$\omega_*+w$};
    \node[above=5pt] at (.65,0) {$I$};
    \node[above=5pt] at (1.05,0) {$\omega_4$};
    \node[below=6pt,font=\scriptsize] at (1.15,0) {$b$};
  \end{scope}

  \begin{scope}[shift={(8.1,0)},x=3.8cm,y=1cm]
    % All rows use the same linear frequency scale omega/omega_*.
    \foreach \y in {4.85,3.6,2.35,1.1}{\draw[axis] (0,\y)--(1.9,\y);}
    \draw[guide,black!25] (.25,.82)--(.25,5.1);
    \draw[guide,black!25] (1.75,.82)--(1.75,5.1);
    \draw[line width=6pt,black!18,line cap=butt] (.25,4.85)--(1.75,4.85);
    \draw[line width=.6pt] (.25,4.74)--(.25,4.96);
    \draw[line width=.6pt] (1.75,4.74)--(1.75,4.96);
    \node[note,anchor=west] at (0,5.36) {Known positive-density interval};
    \node[below=3pt,fill=white] at (.25,4.85) {$1/4$};
    \node[below=3pt,fill=white] at (1.75,4.85) {$7/4$};
    \foreach \y in {3.6,2.35}{
      \draw[line width=6pt,black!18,line cap=butt] (.25,\y)--(1.75,\y);
    }
    \draw[band] (.5,3.6)--(.55,3.6);
    \draw[black!70] (.5,3.48)--(.5,3.72) (.55,3.48)--(.55,3.72);
    \node[anchor=north east,inner sep=1pt,yshift=-5pt] at (.5,3.6) {$1/2$};
    \node[anchor=north west,inner sep=1pt,yshift=-5pt] at (.55,3.6) {$11/20$};
    \node[note,anchor=west,inner sep=3pt] at (.40,4.10)
      {$\omega_1,\omega_2\in I_*=[1/2,11/20]\,\omega_*$};
    \draw[band] (.875,2.35)--(1.1,2.35);
    \foreach \t in {.875,1.1}{\draw[fill=white,line width=.5pt] (\t,2.35) circle[radius=1.6pt];}
    \node[anchor=north east,inner sep=1pt,yshift=-5pt] at (.875,2.35) {$7/8$};
    \node[anchor=north west,inner sep=1pt,yshift=-5pt] at (1.1,2.35) {$11/10$};
    \node[note,anchor=west,inner sep=3pt] at (0,2.73)
      {$\omega_3\in(7/8,11/10)\,\omega_*$};
    \draw[line width=4pt,black!70,line cap=butt] (0,1.1)--(.125,1.1);
    \foreach \t in {0,.125}{\draw[fill=white,line width=.5pt] (\t,1.1) circle[radius=1.6pt];}
    \node[note,anchor=west,inner sep=3pt] at (0,1.58) {$0<\omega_4<\omega_*/8$};
    \node[anchor=north east,inner sep=1pt,yshift=-5pt] at (0,1.1) {$0$};
    \node[anchor=north west,inner sep=1pt,yshift=-5pt,fill=white] at (.125,1.1) {$1/8$};
    \node[below=5pt,inner sep=1pt] at (1,1.1) {$1$};
    \draw (1,1.1)--(1,1.035);
    \node[anchor=east] at (1.9,.54) {$\omega/\omega_*$};
  \end{scope}

\end{tikzpicture}%
}
\caption{The frequency regions used in \cref{lem:cond-interval-expansion,lem:cond-low-frequency-source}.
(a) The large square is $I^2$, with
its darker triangle satisfies $\omega_3\in I$.
The dashed line is $\omega_3=\omega_*-w$.
The small square $R_{\omega_4}$ has center $(q,q)$ (the dot),
side length $w/6$, and lies in this shaded region.
Below, the dark band is $I$ and the light band is $[a,b]$.
(b) On a common scale $\omega/\omega_*$, the gray bands show the interval
where \eqref{eq:cond-density-lower-wide-interval} gives a density lower bound. The dark bands in the next
two rows show $I_*$ for $\omega_1,\omega_2$ and the uniform bounds for
$\omega_3$. The three frequencies stay in the gray interval while
$\omega_4$ in the last row approaches zero.}
\label{fig:cond-frequency-geometry}
\end{figure}

Assume in addition that $N\ge(2/a)\vee(3b/2)^{4\beta+1/2}$.
Since all four frequencies lie in $[a,b]$ on $R_{\omega_4}$, \eqref{eq:cutoff} gives
\begin{equation}\label{eq:cond-density-lower-expansion-kernel}
 \frac{S_{\beta,N}(\omega_1,\omega_2,\omega_3,\omega_4)}
      {4\sqrt{\omega_4}}
 =\frac{S_\beta(\omega_1,\omega_2,\omega_3,\omega_4)}
      {4\sqrt{\omega_4}}
 \ge\frac{a^{4\beta+1/2}}{4\sqrt b}.
\end{equation}
For $t\ge s$ and $\omega_4\in[a,b]$, since $R_{\omega_4}\subset(\omega_4/2,\infty)^2$, \eqref{eq:cond-density-lower-expansion-assumption}
yields
\begin{equation}\label{eq:cond-density-lower-expansion-gain}
 \begin{aligned}
 \mathcal G_N(t,\omega_4)
 &\ge\frac{A^3e^{-3Dt}}{4\sqrt{\omega_4}}
 \iint_{R_{\omega_4}}S_{\beta,N}d\omega_1d\omega_2\\
 &\ge\frac{a^{4\beta+1/2}}{4\sqrt b}
 A^3e^{-3Dt}\left(\frac w6\right)^2
 =\frac{w^2a^{4\beta+1/2}}{144\sqrt b}A^3e^{-3Dt}
 =:c_w A^3e^{-3Dt}.
 \end{aligned}
\end{equation}
Here $c_w>0$ is independent of $N$. Hence
\eqref{eq:cond-density-forced} gives
\begin{equation}\label{eq:cond-density-lower-expansion-evolution}
 \partial_t f_N\ge L_N(t)[f_N]
 -\omega_4^{-d}\ell_N(\omega_4)f_N
 +c_w A^3e^{-3Dt}\mathbf1_{[a,b]}(\omega_4),
 \qquad t\ge s.
\end{equation}
Let $u,C$ be in \eqref{eq:cond-sine-profile}, \eqref{eq:cond-profile-constant} on $[a,b]$.
Since $0\le u\le\mathbf1_{[a,b]}$, \eqref{eq:cond-forced-sine} gives
\begin{equation}\label{eq:cond-density-lower-expansion-comparison}
 f_N(t,\omega_4)\ge c_w A^3u(\omega_4)
 \int_s^t e^{-C(t-\tau)-3D\tau}d\tau.
\end{equation}
For $t\ge s+1$,
\begin{equation}\label{eq:cond-density-lower-expansion-time-integral}
 \begin{aligned}
 \int_s^t e^{-C(t-\tau)-3D\tau}d\tau
 \ge\int_{t-1}^t e^{-C(t-\tau)-3D\tau}d\tau\ge\int_{t-1}^t e^{-C-3Dt}d\tau
 =e^{-C}e^{-3Dt}.
 \end{aligned}
\end{equation}

It remains to show that $u$ is uniformly positive on
$[\omega_*-(2w\wedge(3\omega_*/4)),\,
\omega_*+(2w\wedge(3\omega_*/4))]$.
Since $\omega_*-a=(5w/2)\wedge(7\omega_*/8)$,
\begin{equation}\label{eq:cond-density-lower-target-inset}
 2w\wedge\frac{3\omega_*}{4}
 \le\frac67(\omega_*-a)=\frac{3(b-a)}7.
\end{equation}
Thus, for $|\omega_4-\omega_*|\le2w\wedge(3\omega_*/4)$,
\begin{equation}\label{eq:cond-density-lower-expansion-profile-min}
 u(\omega_4)
 =\cos\!\left(\frac{\pi(\omega_4-\omega_*)}{b-a}\right)
 \ge\cos\frac{3\pi}{7}>0.
\end{equation}
Combining these bounds with
\eqref{eq:cond-density-lower-expansion-comparison}, we obtain
\begin{equation}\label{eq:cond-density-lower-expanded}
 \begin{aligned}
 f_N(t,\omega_4)
 \ge c_w e^{-C}\cos\frac{3\pi}{7}\,A^3e^{-3Dt}
 =:A'e^{-3Dt},\quad t\ge s+1,\quad
 |\omega_4-\omega_*|\le2w\wedge\frac{3\omega_*}{4},
 \end{aligned}
\end{equation}
where $A'>0$ is independent of $N$.
\end{proof}

\begin{lemma}[A lower bound for the gain near zero]
\label{lem:cond-low-frequency-source}
Suppose that $\mathcal N_0=1$ and $g_0$ has a nonzero absolutely
continuous part. There exist $B,D,s,\omega_*>0$ and $N_0\ge1$,
depending only on $\beta,g_0$, such that, for every $N\ge N_0$,
\[
 \mathcal G_N(t,\omega_4)
 \ge Be^{-Dt}\omega_4^\beta\ell_N(\omega_4),
 \qquad t\ge s,\quad 0<\omega_4<\omega_*/8.
\]
\end{lemma}
\begin{proof}
By \cref{lem:cond-ac-interval}, there are
$s_0>0$, $0<\alpha_0<\gamma_0$, and $c_0>0$ such that, for all
sufficiently large $N$,
\begin{equation}\label{eq:cond-density-lower-initial-interval}
 f_N(s_0,\omega_4)\ge c_0,
 \qquad \alpha_0\le\omega_4\le\gamma_0.
\end{equation}

\noindent\textit{Step 1. Persistence on the initial interval.}
Let $u_0$ be the function in \eqref{eq:cond-sine-profile} with
$a=\alpha_0$ and $b=\gamma_0$, and let $C_0$ be the corresponding
constant in \eqref{eq:cond-profile-constant}. Choose $N_0$ so that,
for every $N\ge N_0$, \eqref{eq:cond-density-lower-initial-interval}
holds and \eqref{eq:cond-fixed-profile} is valid with $u=u_0$ and $C=C_0$.
Since $0\le u_0\le1$, \eqref{eq:cond-density-lower-initial-interval} and
\eqref{eq:cond-persistent-sine} give
\begin{equation}\label{eq:cond-density-lower-persistent-initial}
 f_N(t,\omega_4)
 \ge c_0e^{-C_0(t-s_0)}u_0(\omega_4),
 \qquad t\ge s_0.
\end{equation}
Set
\begin{equation}\label{eq:cond-density-lower-center}
 \omega_*=\frac{\alpha_0+\gamma_0}{2},
 \qquad
 w_0=\frac{\gamma_0-\alpha_0}{4},
 \qquad
 I_0=[\omega_*-w_0,\omega_*+w_0].
\end{equation}
Since $0<\alpha_0<\gamma_0$, we have
$0<w_0<(\alpha_0+\gamma_0)/4=\omega_*/2$.
For $\omega_4\in I_0$,
$\pi/4\le\pi(\omega_4-\alpha_0)/(\gamma_0-\alpha_0)\le3\pi/4$,
so $u_0(\omega_4)\ge1/\sqrt2$. Thus
\begin{equation}\label{eq:cond-density-lower-I0}
 f_N(t,\omega_4)
 \ge\frac{c_0e^{C_0s_0}}{\sqrt2}e^{-C_0t}
 =A_0e^{-D_0t},
 \qquad t\ge s_0,\quad \omega_4\in I_0,
\end{equation}
where $A_0=c_0e^{C_0s_0}/\sqrt2$ and $D_0=C_0$ are independent of $N$.

\smallskip
\noindent\textit{Step 2. Repeated expansion of the interval.}
Set $w_k=2^kw_0\wedge(3\omega_*/4)$ for $k\ge0$, and let $K$ be
the first index for which $w_K=3\omega_*/4$.
Since $0<w_0<\omega_*/2$, we have $1\le K<\infty$.
For $0\le k<K$, the interval widths satisfy
$w_k<3\omega_*/4$ and $w_{k+1}=2w_k\wedge(3\omega_*/4)$.

We prove by induction that, for each $0\le k\le K$, there is a
constant $A_k>0$, independent of $N$, such that all sufficiently
large $N$ satisfy
\begin{equation}\label{eq:cond-density-lower-iteration-bound}
 \begin{aligned}
 f_N(t,\omega_4)&\ge A_ke^{-3^kD_0t},\qquad t\ge s_0+k,\qquad |\omega_4-\omega_*|\le w_k.
 \end{aligned}
\end{equation}
For $k=0$, this is \eqref{eq:cond-density-lower-I0}.
Suppose that the bound holds for some $k<K$.
Apply \cref{lem:cond-interval-expansion} with $w=w_k$, $s=s_0+k$, $A=A_k$, and $D=3^kD_0$.
It gives a constant $A_{k+1}>0$, independent of $N$, such that
\[
 \begin{aligned}
 f_N(t,\omega_4)&\ge A_{k+1}e^{-3^{k+1}D_0t},\qquad t\ge s_0+k+1,\quad |\omega_4-\omega_*|\le
 2w_k\wedge\frac{3\omega_*}{4}=w_{k+1}.
 \end{aligned}
\]
This proves \eqref{eq:cond-density-lower-iteration-bound}.
Since we only need to apply \cref{lem:cond-interval-expansion} a finite number of times, we may increase
$N_0$ so that this bound holds for every $N\ge N_0$ and every
$0\le k\le K$.

Taking $k=K$, so that $w_K=3\omega_*/4$, gives
\begin{equation}\label{eq:cond-density-lower-wide-interval}
 f_N(t,\omega_4)\ge A_Ke^{-3^KD_0t},
 \qquad t\ge s_0+K,\quad
 \frac{\omega_*}{4}\le\omega_4\le\frac{7\omega_*}{4}.
\end{equation}

\smallskip
\noindent\textit{Step 3. The gain at small frequencies.}
Fix $0<\omega_4<\omega_*/8$ and let
$\omega_1,\omega_2\in I_*=[\omega_*/2,11\omega_*/20]$.
Hence $\omega_1+\omega_2\in[\omega_*,11\omega_*/10]$,
\begin{equation}\label{eq:cond-density-lower-third-frequency}
 \frac{7\omega_*}{8}<\omega_3
 =\omega_1+\omega_2-\omega_4<\frac{11\omega_*}{10}.
\end{equation}
Thus $\omega_1,\omega_2,\omega_3\in[\omega_*/4,7\omega_*/4]$, where
\eqref{eq:cond-density-lower-wide-interval} applies.
Since $0<\omega_4<\omega_*/8$, we have
$\omega_1,\omega_2\ge\omega_*/2>\omega_4$
and $\omega_3>7\omega_*/8>\omega_4$.
Hence $I_*^2\subset(\omega_4/2,\infty)^2$, and $\omega_4$ is the
smallest of the four frequencies.
One can refer to \cref{fig:cond-frequency-geometry}(b) to see the
ranges of $\omega_1,\omega_2,\omega_3$ and $\omega_4$.

Increase $N_0$, if necessary, so that
$N_0\ge(2/\omega_*)\vee(7\omega_*/4)^{4\beta+1/2}$,
and let $N\ge N_0$.
For $(\omega_1,\omega_2)\in I_*^2$, we have
$\ell_N(\omega_j)=1$ for $j=1,2,3$, and
$S_{\beta,N}=S_\beta\ell_N(\omega_4)$.
Since the smallest frequency is $\omega_4$, 
\begin{equation}\label{eq:cond-density-lower-source-kernel}
 \begin{aligned}
 \frac{S_{\beta,N}(\omega_1,\omega_2,\omega_3,\omega_4)}
      {4\sqrt{\omega_4}}
 =\frac14(\omega_1\omega_2\omega_3)^\beta
   \omega_4^\beta\ell_N(\omega_4)\ge
 \frac14\left(\frac{\omega_*^2}{4}
              \frac{7\omega_*}{8}\right)^\beta
 \omega_4^\beta\ell_N(\omega_4).
 \end{aligned}
\end{equation}
For $t\ge s_0+K$ and $0<\omega_4<\omega_*/8$,  \eqref{eq:cond-density-lower-wide-interval} gives
\begin{equation}\label{eq:cond-density-lower-small-frequency-gain}
 \begin{aligned}
 \mathcal G_N(t,\omega_4)
 &\ge\frac1{4\sqrt{\omega_4}}
 \iint_{I_*^2}
 S_{\beta,N}A_K^3e^{-3^{K+1}D_0t}d\omega_1d\omega_2\\
 &\ge\frac14\left(\frac{7\omega_*^3}{32}\right)^\beta
 A_K^3e^{-3^{K+1}D_0t}\omega_4^\beta\ell_N(\omega_4)|I_*|^2
 =Be^{-Dt}\omega_4^\beta\ell_N(\omega_4),
 \end{aligned}
\end{equation}
where
\begin{equation}\label{eq:cond-density-lower-source-constants}
 B=\frac{A_K^3}{4}\left(\frac{7\omega_*^3}{32}\right)^\beta
       \left(\frac{\omega_*}{20}\right)^2>0,
 \qquad D=3^{K+1}D_0.
\end{equation}
Thus \eqref{eq:cond-density-forced} yields
\begin{equation}\label{eq:cond-density-source}
 \begin{aligned}
 \partial_t f_N(t,\omega_4)
 &\ge L_N(t)[f_N(t,\cdot)](\omega_4)
 -\omega_4^{-d}\ell_N(\omega_4)f_N(t,\omega_4)\\
 &\quad+Be^{-Dt}\omega_4^\beta\ell_N(\omega_4)
 \mathbf1_{(0,\omega_*/8)}(\omega_4),
 \qquad t\ge s_0+K.
 \end{aligned}
\end{equation}

Set $s=s_0+K$. This proves the assertion.
\end{proof}

\begin{lemma}[Positive density at small frequencies]\label{lem:cond-density-lower}
Suppose that $\mathcal N_0=1$ and $g_0$ has a nonzero absolutely
continuous part. There exist $c,D,s_*,r_*>0$, depending only on
$\beta,g_0$, with the following property. For every $a\in(0,r_*)$,
there is $N_a\ge1$, depending only on $a,\beta,g_0$, such that,
for every $N\ge N_a$,
\begin{equation}\label{eq:cond-density-lower}
 f_N(t,\omega_4)\ge ce^{-Dt}\omega_4^{1-3\beta},
 \qquad t\ge s_*,\quad a/2\le\omega_4\le3a/2.
\end{equation}
\end{lemma}
\begin{proof}
Let $B,D,s,\omega_*,N_0$ be as in
\cref{lem:cond-low-frequency-source}. Combining its gain bound with
\eqref{eq:cond-density-forced} gives \eqref{eq:cond-density-source}
for $t\ge s$.
Set $r_*=1\wedge(\omega_*/14)$ and fix $0<a<r_*$ and $N\ge N_a$, where
\begin{equation}\label{eq:cond-density-lower-final-cutoff}
 N_a=\max\left\{N_0,\frac8a,
                 \left(\frac{21a}{8}\right)^{4\beta+1/2}\right\}.
\end{equation}
Since $a<r_*$, we have $a<1$ and
$7a/4<7\omega_*/56=\omega_*/8$.
Define
\begin{equation}\label{eq:cond-density-lower-small-profile}
 u_a(\omega_4)
 =\sin\!\left(
 \frac{2\pi}{3a}\left(\omega_4-\frac a4\right)\right)
 \mathbf1_{[a/4,\,7a/4]}(\omega_4).
\end{equation}
This is the function in \eqref{eq:cond-sine-profile} with endpoints
$a/4$ and $7a/4$, whose difference is $3a/2$.
In particular, $0\le u_a\le\mathbf1_{[a/4,7a/4]}$.
The constant in \eqref{eq:cond-profile-constant} is therefore
\begin{equation}\label{eq:cond-density-lower-small-profile-constant}
 \frac{\pi^2(7a/4)^{2-d}}{4(3a/2)^2}
  +\frac5{4(a/4)^d}
 =\left[\frac{\pi^2(7/4)^{2-d}}9+5\cdot4^{d-1}\right]a^{-d}
 =C_\beta a^{-d}.
\end{equation}Hence \cref{lem:cond-fixed-profile} and
$0\le\ell_N\le1$ give, for every $t\ge0$ and $\omega_4>0$,
\begin{equation}\label{eq:cond-density-lower-small-profile-bound}
 \begin{aligned}
 L_N(t)[u_a](\omega_4)
 -\omega_4^{-d}\ell_N(\omega_4)u_a(\omega_4)
 &\ge L_N(t)[u_a](\omega_4)-\omega_4^{-d}u_a(\omega_4)\\
 &\ge-C_\beta a^{-d}u_a(\omega_4).
 \end{aligned}
\end{equation}

Moreover, $N^{-1}\le a/8<a/4$, so $\ell_N=1$ on
$[a/4,7a/4]$ by \eqref{eq:cutoff}. On this interval,
$\omega_4^\beta\ge(a/4)^\beta=4^{-\beta}a^\beta$, since $\beta>0$.
Using $[a/4,7a/4]\subset(0,\omega_*/8)$ and
$0\le u_a\le\mathbf1_{[a/4,7a/4]}$, we obtain
\begin{equation}\label{eq:cond-density-lower-small-source}
 \begin{aligned}
 Be^{-Dt}\omega_4^\beta\ell_N(\omega_4)
 \mathbf1_{(0,\omega_*/8)}(\omega_4)
 &\ge B4^{-\beta}a^\beta e^{-Dt}
 \mathbf1_{[a/4,7a/4]}(\omega_4)\\
 &\ge B4^{-\beta}a^\beta e^{-Dt}u_a(\omega_4).
 \end{aligned}
\end{equation}
Combining \eqref{eq:cond-density-source} with
\eqref{eq:cond-density-lower-small-source} verifies
\eqref{eq:cond-forced-assumption} with
$u=u_a$, $C=C_\beta a^{-d}$, and $c=B4^{-\beta}a^\beta$.
Thus \cref{lem:cond-sine-comparison}\textup{(ii)} gives, for $t\ge s$,
\begin{equation}\label{eq:cond-density-lower-small-comparison}
 f_N(t,\omega_4)
 \ge B4^{-\beta}a^\beta u_a(\omega_4)
 \int_{s}^t
 e^{-C_{\beta}a^{-d}(t-\tau)-D\tau} d\tau.
\end{equation}
Set $s_*=s+1$. For $t\ge s_*$, since $0<a<1$ and $d=1-4\beta>0$,
we have $a^d<1$ and $[t-a^d,t]\subset[s,t]$.
Making
the change of variables $r=(t-\tau)/a^d$ gives
\begin{equation}\label{eq:cond-density-lower-small-time-integral}
 \begin{aligned}
 \int_{s}^t e^{-C_\beta a^{-d}(t-\tau)-D\tau}d\tau
 &\ge e^{-Dt}\int_{t-a^d}^t
 e^{-C_\beta a^{-d}(t-\tau)}d\tau
 =\frac{1-e^{-C_\beta}}{C_\beta}a^de^{-Dt}.
 \end{aligned}
\end{equation}
Inserting this bound into \eqref{eq:cond-density-lower-small-comparison}
gives
\begin{equation}\label{eq:cond-density-lower-small-final}
 f_N(t,\omega_4)
 \ge\frac{B4^{-\beta}(1-e^{-C_\beta})}{C_\beta}
 a^{\beta+d}e^{-Dt}u_a(\omega_4).
\end{equation}
If $a/2\le\omega_4\le3a/2$, then
\[
 \frac\pi6\le\frac{2\pi}{3a}\left(\omega_4-\frac a4\right)
 \le\frac{5\pi}6.
\]
Hence
$u_a(\omega_4)\ge1/2$.
Also, $a\ge2\omega_4/3$ and $\beta+d=1-3\beta>0$, whence
\[
 a^{\beta+d}\ge\left(\frac{2\omega_4}{3}\right)^{\beta+d}
 =\left(\frac23\right)^{\beta+d}\omega_4^{1-3\beta}.
\]
Substituting these two bounds into
\eqref{eq:cond-density-lower-small-final} gives
\begin{equation}\label{eq:cond-density-lower-final-constant}
 f_N(t,\omega_4)
 \ge\frac{B4^{-\beta}(1-e^{-C_\beta})}{2C_\beta}
 \left(\frac23\right)^{\beta+d}
 e^{-Dt}\omega_4^{1-3\beta}
 =ce^{-Dt}\omega_4^{1-3\beta},
\end{equation}
for $t\ge s_*$ and $a/2\le\omega_4\le3a/2$.
Here $c$ is the positive coefficient in
\eqref{eq:cond-density-lower-final-constant}.
The constants $B,D,s,\omega_*$ are supplied by
\cref{lem:cond-low-frequency-source} and depend only on $\beta,g_0$,
while $C_\beta$ depends only on $\beta$.
Thus $c,D,s_*,r_*$ are independent of $a,N$ and depend only on
$\beta,g_0$. This proves \eqref{eq:cond-density-lower}.
\end{proof}

\subsection{Transfer to smaller frequency intervals}
\label{subsec:cond-transfer}

Let $L=L(\beta)$ be as in \cref{lem:kernel}. For
$0<r\le R/2<\infty$ and $0<\varepsilon<1$, set $a=r/(8L)$ and define
\begin{equation}\label{eq:cond-transfer-sine-profile}
 \psi_a(\omega)=\cos\!\left(\frac{\pi(\omega-a)}a\right)
 \mathbf1_{[a/2,3a/2]}(\omega).
\end{equation}
\begin{equation}\label{eq:cond-local-weighted-mass}
 0\le J_N(t)=\int_{(0,2La)}(\eta/a)^\beta\ell_N(\eta)g_t^N(d\eta)
 \le(2L)^\beta g_t^N((0,2La))\le(2L)^\beta.
\end{equation}
\begin{equation}\label{eq:cond-growth-rate}
 \lambda=\frac\beta{48}a^\beta r^{\beta+1/2}R^{2\beta-3/2}\varepsilon^2.
\end{equation}
Choose $H>2R\vee\lambda^{-1/d}$ and an integer
$N_0\ge(4/a)\vee(2H)^{4\beta+1/2}$, and let $N\ge N_0$.
Here $N_0$ depends only on $\beta,r,R,\varepsilon$.
Throughout this subsection, $C>0$ denotes constants depending only on $\beta$.

\begin{lemma}[Lower bounds for $V_N$ and $L_N\psi_a$]
\label{lem:cond-potential-jump-lower}
There exists $C>0$, depending only on $\beta$, such that, for
$0<r\le R/2$, $0<\varepsilon<1$,
$N\ge N_0(\beta,r,R,\varepsilon)$, $s\ge0$, and $\tau>0$, the condition
\begin{equation}\label{eq:cond-transfer-annular-mass}
 g_t^N([r,R])\ge\varepsilon,\qquad s\le t\le s+\tau,
\end{equation}
implies, for $s\le t\le s+\tau$,
\[
 V_N(t,\omega)\ge\lambda-Ca^{-d}J_N(t),
 \qquad a/2\le\omega\le3a/2,
\]
\[
 L_N(t)[\psi_a](\omega)\ge-Ca^{-d}J_N(t)\psi_a(\omega),
 \qquad\omega>0,
\]
and
\[
 L_N(t)[\psi_a](\omega)+V_N(t,\omega)\psi_a(\omega)
 \ge[\lambda-Ca^{-d}J_N(t)]\psi_a(\omega),
 \qquad\omega>0.
\]
\end{lemma}
\begin{proof}
Fix $t\in[s,s+\tau]$.

\noindent\textit{Step 1. The lower bound for $V_N$.}
Fix $a/2\le\omega_4\le3a/2$. In the first integral below,
$\omega_3=\omega_1+\omega_2-\omega_4$; in the second,
$\omega_1=\omega_3+\omega_4-\omega_2$. In the second integral,
rename the integration variables $(\omega_2,\omega_3)$ as
$(\omega_2,\omega_1)$, then average under
$(\omega_1,\omega_2)\mapsto(\omega_2,\omega_1)$.
By \eqref{eq:cond-potential} and the symmetry of $S_{\beta,N}$,
\begin{equation}\label{eq:cond-symmetrized-potential}
 \begin{aligned}
 V_N(t,\omega_4)
 &=\frac14\iint_{(0,\infty)^2}
 \frac{S_{\beta,N}(\omega_1,\omega_2,\omega_3,\omega_4)}
      {\sqrt{\omega_1\omega_2\omega_4}}
 \,g_t^N(d\omega_1)g_t^N(d\omega_2)\\
 &\quad-\frac12\iint_{(0,\infty)^2\setminus(0,\omega_4/2)^2}
 \frac{S_{\beta,N}(\omega_1,\omega_2,\omega_3,\omega_4)}
      {\sqrt{\omega_2\omega_3\omega_4}}
 \,g_t^N(d\omega_2)g_t^N(d\omega_3)\\
 &=\frac14\iint_{(0,\infty)^2}K_{\beta,N}(\omega_1,\omega_2,\omega_4)
 \,g_t^N(d\omega_1)g_t^N(d\omega_2)\\
 &\quad-\frac14\iint_{\substack{0<\omega_1\wedge\omega_2\\\omega_1\vee\omega_2\ge\omega_4/2}}
 \bigl[K_{\beta,N}(\omega_4,\omega_1,\omega_2)
       +K_{\beta,N}(\omega_4,\omega_2,\omega_1)\bigr]
 \,g_t^N(d\omega_1)g_t^N(d\omega_2).
 \end{aligned}
\end{equation}
By \eqref{eq:cutoff},
$K_{\beta,N}(\omega_4,\omega_1,\omega_2)
\le\ell_N(\omega_1)\ell_N(\omega_2)K_\beta(\omega_4,\omega_1,\omega_2)$.
Since $L\omega_4\le3La/2<2La$, \eqref{eq:cond-potential-negative-kernel}
gives
\begin{equation}
 \begin{aligned}
 &\frac14\iint_{\substack{0<\omega_1\wedge\omega_2<L\omega_4\\
                          \omega_1\vee\omega_2\ge\omega_4/2}}
 \bigl[K_{\beta,N}(\omega_4,\omega_1,\omega_2)
       +K_{\beta,N}(\omega_4,\omega_2,\omega_1)\bigr]
       \,g_t^N(d\omega_1)g_t^N(d\omega_2)\\
 &\quad\le\sqrt2\,\omega_4^{-d-\beta}g_t^N((0,\infty))
          \int_{(0,L\omega_4)}\eta^\beta\ell_N(\eta)g_t^N(d\eta)
 \le\sqrt2\,\omega_4^{-d-\beta}a^\beta J_N(t)
 \le C a^{-d}J_N(t).
 \end{aligned}
\end{equation}

Recall the definition in \cref{sec:measure}:
\begin{equation}
 (x,y,z)=\operatorname{sort}(\omega_1,\omega_2,\omega_3)
 =\left(\min_{1\le j\le3}\omega_j,\,
        \operatorname{med}_{1\le j\le3}\omega_j,\,
        \max_{1\le j\le3}\omega_j\right).
\end{equation}
For $\omega_1\wedge\omega_2\ge L\omega_4$, we have
$\omega_4<\omega_1\wedge\omega_2\le\omega_1\vee\omega_2$, so
\begin{equation}
 \begin{aligned}
 &K_{\beta,N}(\omega_1,\omega_2,\omega_4)
 -K_{\beta,N}(\omega_4,\omega_1,\omega_2)
 -K_{\beta,N}(\omega_4,\omega_2,\omega_1)\\
 &\quad=(G_{3,N}-G_{2,N}-G_{1,N})
 (\omega_4,\omega_1\wedge\omega_2,\omega_1\vee\omega_2).
 \end{aligned}
\end{equation}
For $a/2\le x\le3a/2$ and $Lx\le y\le z$, we have
$N^{-1}<x\le x+z-y\le y+z-x$. Thus \eqref{eq:cutoff} gives
\begin{equation}
 \begin{aligned}
 G_{3,N}(x,y,z)
 &=\frac{\bigl[(xyz)^\beta\sqrt x\,(y+z-x)^\beta\bigr]\wedge N}{\sqrt{xyz}}\\
 &\ge\frac{\bigl[(xyz)^\beta\sqrt x\,(x+z-y)^\beta\bigr]\wedge N}{\sqrt{xyz}}
 =G_{2,N}(x,y,z).
 \end{aligned}
\end{equation}
If $z\le H$, then $y+z-x\le2H$, and hence
\begin{equation}
 \begin{gathered}
 S_\beta(x,y,z,y+z-x)\le(2H)^{4\beta+1/2}\le N,\\
 G_{3,N}(x,y,z)=G_3(x,y,z),\qquad
 G_{2,N}(x,y,z)=G_2(x,y,z).
 \end{gathered}
\end{equation}
Using $G_{1,N}\le G_1$ and \eqref{eq:cond-high-potential}, we obtain
\begin{equation}
 \begin{aligned}
 (G_{3,N}-G_{2,N}-G_{1,N})(x,y,z)
 &\ge(G_3-G_2-G_1)(x,y,z)\\
 &\ge\beta2^{\beta-1}x^\beta y^{\beta+1/2}z^{2\beta-3/2}\ge0.
 \end{aligned}
\end{equation}
Since $L\omega_4\le3La/2=3r/16<r\le R<H$ and $2\beta-3/2<0$,
\eqref{eq:cond-transfer-annular-mass} gives
\begin{equation}\label{eq:cond-uncut-positive-contribution}
 \begin{aligned}
 &\frac14\iint_{\substack{\omega_1\wedge\omega_2\ge L\omega_4\\\omega_1\vee\omega_2\le H}}
 (G_{3,N}-G_{2,N}-G_{1,N})
 (\omega_4,\omega_1\wedge\omega_2,\omega_1\vee\omega_2)
 \,g_t^N(d\omega_1)g_t^N(d\omega_2)\\
 &\quad\ge\frac14\iint_{[r,R]^2}(G_3-G_2-G_1)
 (\omega_4,\omega_1\wedge\omega_2,\omega_1\vee\omega_2)
 \,g_t^N(d\omega_1)g_t^N(d\omega_2)\\
 &\quad\ge\frac\beta8a^\beta r^{\beta+1/2}R^{2\beta-3/2}
       \bigl[g_t^N([r,R])\bigr]^2\ge2\lambda.
 \end{aligned}
\end{equation}

For $z>H$, $G_{3,N}\ge G_{2,N}$ and
\eqref{eq:condensation-kernel-bound} give
\begin{equation}
 (G_{2,N}+G_{1,N}-G_{3,N})_+(x,y,z)
 \le G_{1,N}(x,y,z)
 \le2^\beta y^{-d}\mathbf1_{\{z<x+y\}}.
\end{equation}
On the set $z<x+y$, we have
\begin{equation}
 y>z-x>H-3a/2>H/2.
\end{equation}
Since $\beta+d-2=-1-3\beta<0$ and $g_t^N([0,\infty))=1$,
\begin{equation}
 \begin{aligned}
 &\frac14\iint_{\substack{\omega_1\wedge\omega_2\ge L\omega_4\\\omega_1\vee\omega_2>H}}
 (G_{2,N}+G_{1,N}-G_{3,N})_+
 (\omega_4,\omega_1\wedge\omega_2,\omega_1\vee\omega_2)
 \,g_t^N(d\omega_1)g_t^N(d\omega_2)\\
 &\quad\le2^{\beta-2}\iint_{\substack{\omega_1\wedge\omega_2\ge L\omega_4\\\omega_1\vee\omega_2>H,\ |\omega_1-\omega_2|<\omega_4}}
 (\omega_1\wedge\omega_2)^{-d}\,g_t^N(d\omega_1)g_t^N(d\omega_2)\\
 &\quad\le2^{\beta+d-2}H^{-d}\bigl[g_t^N((H/2,\infty))\bigr]^2
 \le H^{-d}\le\lambda.
 \end{aligned}
\end{equation}
Combining these estimates with \eqref{eq:cond-symmetrized-potential},
we obtain
\begin{equation}\label{eq:cond-potential-lower}
 \begin{aligned}
 V_N(t,\omega_4)
 &\ge\frac14\iint_{\{\omega_1\wedge\omega_2\ge L\omega_4\}}
 (G_{3,N}-G_{2,N}-G_{1,N})
 (\omega_4,\omega_1\wedge\omega_2,\omega_1\vee\omega_2)\\
 &\hspace{4em}\times g_t^N(d\omega_1)g_t^N(d\omega_2)
 -C a^{-d}J_N(t)\\
 &=\frac14\iint_{\substack{\omega_1\wedge\omega_2\ge L\omega_4\\
                         \omega_1\vee\omega_2\le H}}
 (G_{3,N}-G_{2,N}-G_{1,N})
 (\omega_4,\omega_1\wedge\omega_2,\omega_1\vee\omega_2)\\
 &\hspace{4em}\times g_t^N(d\omega_1)g_t^N(d\omega_2)\\
 &\quad+\frac14\iint_{\substack{\omega_1\wedge\omega_2\ge L\omega_4\\
                         \omega_1\vee\omega_2>H}}
 (G_{3,N}-G_{2,N}-G_{1,N})
 (\omega_4,\omega_1\wedge\omega_2,\omega_1\vee\omega_2)\\
 &\hspace{4em}\times g_t^N(d\omega_1)g_t^N(d\omega_2)
 -C a^{-d}J_N(t)\\
 &\ge2\lambda-\lambda-C a^{-d}J_N(t)
 =\lambda-C a^{-d}J_N(t),\qquad a/2\le\omega_4\le3a/2.
 \end{aligned}
\end{equation}

\smallskip
\noindent\textit{Step 2. The lower bound for $L_N\psi_a$.}
Fix $a/2\le\omega_4\le3a/2$ and $0<\omega_2\le\omega_3<\omega_4/2$.
The choice of $N_0$ gives $N\ge4/a$ and
$N\ge(2H)^{4\beta+1/2}\ge(9a/4)^{4\beta+1/2}$.
Applying \eqref{eq:cond-paired-sine-bound} to $\psi_a$, we obtain
\begin{equation}
 \begin{aligned}
 &\sum_{\sigma=\pm1}\left(1+\sigma\frac{\omega_3-\omega_2}{\omega_4}\right)^\beta
 [\psi_a(\omega_4+\sigma(\omega_3-\omega_2))-\psi_a(\omega_4)]\\
 &\quad\ge-\left[\frac{\pi^2(\omega_3-\omega_2)^2}{a^2}
              +\frac{\omega_3-\omega_2}{\omega_4}\right]\psi_a(\omega_4)
 \ge-C\frac{\omega_3-\omega_2}{a}\psi_a(\omega_4).
 \end{aligned}
\end{equation}
Here $0\le\omega_3-\omega_2\le\omega_3<3a/4$ and $\omega_4\ge a/2$.
Since $2\beta-1/2<0$ and $d=1-4\beta$,
\begin{equation}\label{eq:cond-transfer-jump-weight}
 \begin{aligned}
 \omega_4^{2\beta-1/2}\omega_2^\beta\omega_3^{\beta-1/2}
 \frac{\omega_3-\omega_2}{a}
 &\le C a^{2\beta-3/2}\omega_2^\beta\omega_3^{\beta+1/2}\\
 &\le C a^{2\beta-1}\omega_2^\beta\omega_3^\beta
 =C a^{-d}(\omega_2/a)^\beta(\omega_3/a)^\beta.
 \end{aligned}
\end{equation}
For $\omega_3<\omega_2$, exchange the integration variables
$\omega_2,\omega_3$. Substituting \eqref{eq:cond-sine-kernel-coefficients}
into \eqref{eq:cond-symmetrized-jump} and using
$\omega_4/2<2La$ and $J_N(t)\le(2L)^\beta$, we obtain
\begin{equation}\label{eq:cond-transfer-jump-loss}
 \begin{aligned}
 &L_N(t)[\psi_a](\omega_4)\\
 &\quad\ge-Ca^{-d}\psi_a(\omega_4)
 \iint_{(0,\omega_4/2)^2}(\omega_2/a)^\beta(\omega_3/a)^\beta
 \ell_N(\omega_2)\ell_N(\omega_3)\,g_t^N(d\omega_2)g_t^N(d\omega_3)\\
 &\quad=-Ca^{-d}\psi_a(\omega_4)
 \left[\int_{(0,\omega_4/2)}(\eta/a)^\beta\ell_N(\eta)g_t^N(d\eta)\right]^2\\
 &\quad\ge-Ca^{-d}J_N(t)^2\psi_a(\omega_4)
 \ge-Ca^{-d}J_N(t)\psi_a(\omega_4).
 \end{aligned}
\end{equation}
For $\omega_4>0$ outside $[a/2,3a/2]$, $\psi_a(\omega_4)=0$.
With $\omega_1=\omega_3+\omega_4-\omega_2$,
\eqref{eq:cond-jump-operator} gives
\begin{equation}
 L_N(t)[\psi_a](\omega_4)
 =\frac12\iint_{(0,\omega_4/2)^2}
 \frac{S_{\beta,N}(\omega_1,\omega_2,\omega_3,\omega_4)}
      {\sqrt{\omega_2\omega_3\omega_4}}
 \psi_a(\omega_1)\,g_t^N(d\omega_2)g_t^N(d\omega_3)\ge0.
\end{equation}
Combining this with \eqref{eq:cond-potential-lower}, we obtain, for every $\omega_4>0$,
\begin{equation}\label{eq:cond-fixed-profile-jump}
 L_N(t)[\psi_a](\omega_4)+V_N(t,\omega_4)\psi_a(\omega_4)
 \ge[\lambda-Ca^{-d}J_N(t)]\psi_a(\omega_4).
\end{equation}
\end{proof}

\begin{lemma}[Time integral of $J_N$]
\label{lem:cond-weighted-mass-integral}
Let $0<r\le R/2$, $0<\varepsilon<1$,
$N\ge N_0(\beta,r,R,\varepsilon)$, $s\ge0$, and $\tau>0$.
Assume \eqref{eq:cond-transfer-annular-mass}.
For $\phi(\omega)=(1-2\omega/r)_+$ and $s\le t\le s+\tau$,
\[
 \frac{d}{dt}\int_{[0,\infty)}\phi(\omega)g_t^N(d\omega)
 \ge\lambda J_N(t),
\]
and
\[
 \lambda\int_s^tJ_N(u)\,du
 \le\int_{[0,\infty)}\phi(\omega)g_t^N(d\omega)
 -\int_{[0,\infty)}\phi(\omega)g_s^N(d\omega)
 \le1.
\]
\end{lemma}
\begin{proof}
Set $\phi(\omega)=(1-2\omega/r)_+$. For $0<x<r/4$ and
$r\le y\le z\le R$, we have $\phi(x)\ge1/2$ and
$\phi(y)=\phi(z)=\phi(y+z-x)=0$.
Since $N^{-1}\le a/4<r$ and $(2R)^{4\beta+1/2}\le N$,
\eqref{eq:cutoff} gives
\begin{equation}
 \begin{aligned}
 G_{3,N}(x,y,z)&=\ell_N(x)G_3(x,y,z),\\
 G_{2,N}(x,y,z)&=\ell_N(x)\ell_N(x+z-y)G_2(x,y,z)
 \le\ell_N(x)G_2(x,y,z).
 \end{aligned}
\end{equation}
Since $\phi$ is nonincreasing and $(x+y-z)_+\le x$, by \eqref{eq:cond-kernel-gain-difference}, $2x<y$, and
\eqref{eq:cond-growth-rate},
\begin{equation}\label{eq:cond-hinge-kernel-lower-bound}
 \begin{aligned}
 F_{\beta,N}[\phi](x,y,z)
 &=(G_{3,N}-G_{2,N})(x,y,z)\phi(x)+G_{2,N}(x,y,z)\phi(x+z-y)\\
 &\quad+G_{1,N}(x,y,z)\bigl[\phi((x+y-z)_+)-\phi(x)\bigr]\\
 &\ge\ell_N(x)(G_3-G_2)(x,y,z)\phi(x)\\
 &\ge\frac\beta4x^\beta\ell_N(x)r^{\beta+1/2}R^{2\beta-3/2}
 =\frac{12\lambda}{\varepsilon^2}(x/a)^\beta\ell_N(x).
 \end{aligned}
\end{equation}
Since $\phi$ is convex, \cref{lem:ordered} gives
$F_{\beta,N}[\phi]\ge0$. Using \eqref{eq:ordered},
$2La=r/4$, and \eqref{eq:cond-transfer-annular-mass}, we obtain
\begin{equation}\label{eq:cond-hinge-ordered-evolution}
 \begin{aligned}
 &\frac{d}{dt}\int_{[0,\infty)}\phi(\omega)g_t^N(d\omega)\\
 &\quad=\frac1{12}\iiint_{[0,\infty)^3}
 F_{\beta,N}[\phi](\operatorname{sort}(\omega_1,\omega_2,\omega_3))
 \,g_t^N(d\omega_1)g_t^N(d\omega_2)g_t^N(d\omega_3)\\
 &\quad\ge\frac1{12}\iiint_{(0,r/4)\times[r,R]^2}
 F_{\beta,N}[\phi](\omega_1,\omega_2\wedge\omega_3,\omega_2\vee\omega_3)
 \,g_t^N(d\omega_1)g_t^N(d\omega_2)g_t^N(d\omega_3)\\
 &\quad\ge\frac{\lambda}{\varepsilon^2}
 \iiint_{(0,r/4)\times[r,R]^2}(\omega_1/a)^\beta\ell_N(\omega_1)
 \,g_t^N(d\omega_1)g_t^N(d\omega_2)g_t^N(d\omega_3)\\
 &\quad=\frac{\lambda}{\varepsilon^2}J_N(t)\bigl[g_t^N([r,R])\bigr]^2
 \ge\lambda J_N(t).
 \end{aligned}
\end{equation}
Integrating over $[s,t]$ and using $0\le\phi\le1$ gives
\begin{equation}\label{eq:cond-hinge-budget}
 \lambda\int_s^tJ_N(u)\,du
 \le\int_{[0,\infty)}\phi(\omega)g_t^N(d\omega)
 -\int_{[0,\infty)}\phi(\omega)g_s^N(d\omega)
 \le1,\qquad s\le t\le s+\tau.
\end{equation}
\end{proof}

\begin{lemma}[Exponential lower bound for the density]
\label{lem:cond-density-amplification}
Let $0<r\le R/2$, $0<\varepsilon<1$,
$N\ge N_0(\beta,r,R,\varepsilon)$, $s\ge0$, $\tau>0$, and $c_0>0$.
Assume \eqref{eq:cond-transfer-annular-mass} and
\begin{equation}\label{eq:cond-amplification-initial}
 f_N(s,\omega)\ge c_0\omega^{1-3\beta},
 \qquad a/2\le\omega\le3a/2.
\end{equation}
There exists $C>0$, depending only on $\beta$, such that, for
$s\le t\le s+\tau$ and almost every $\omega>0$,
\[
 \begin{aligned}
 f_N(t,\omega)
 &\ge c_0(a/2)^{1-3\beta}
 \exp\!\left(\lambda(t-s)-Ca^{-d}\int_s^tJ_N(u)\,du\right)\psi_a(\omega)\\
 &\ge c_0(a/2)^{1-3\beta}
 e^{\lambda(t-s)-Ca^{-d}/\lambda}\psi_a(\omega).
 \end{aligned}
\]
\end{lemma}
\begin{proof}
By \cref{lem:cond-potential-jump-lower,lem:cond-weighted-mass-integral},
\eqref{eq:cond-fixed-profile-jump} and \eqref{eq:cond-hinge-budget} hold
on $[s,s+\tau]$. We use \eqref{eq:cond-output-gain}; the nonnegative source
$\mathcal G_N$ can be discarded for this lower bound.
Fix $C$ as in \eqref{eq:cond-fixed-profile-jump} and set
\begin{equation}
 v(t,\omega_4)=c_0(a/2)^{1-3\beta}
 \exp\!\left(\lambda(t-s)-Ca^{-d}\int_s^tJ_N(u)\,du\right)\psi_a(\omega_4).
\end{equation}
By \eqref{eq:cond-amplification-initial}, $1-3\beta>0$, and $0\le\psi_a\le1$,
\begin{equation}
 v(s,\omega_4)\le c_0\omega_4^{1-3\beta}
 \mathbf1_{[a/2,3a/2]}(\omega_4)\le f_N(s,\omega_4),\qquad\omega_4>0.
\end{equation}
Set $w=v-f_N$. Equations \eqref{eq:cond-fixed-profile-jump}
and \eqref{eq:cond-output-gain} give
\begin{equation}
 \begin{aligned}
 \partial_tw(t,\omega_4)
 &=[\lambda-Ca^{-d}J_N(t)]v(t,\omega_4)-\partial_tf_N(t,\omega_4)\\
 &\le L_N(t)[w(t,\cdot)](\omega_4)+V_N(t,\omega_4)w(t,\omega_4)
       -\mathcal G_N(t,\omega_4)\\
 &\le L_N(t)[w(t,\cdot)](\omega_4)+V_N(t,\omega_4)w(t,\omega_4).
 \end{aligned}
\end{equation}
Since $J_N$ is continuous, the definition of $v$ and
\eqref{eq:cond-density-time-regularity} give
$w\in C^1([s,s+\tau];L^1((0,\infty),\sqrt\omega\,d\omega))$.
We choose a representative such that $w(\cdot,\omega_4)$ is absolutely
continuous on $[s,s+\tau]$ for almost every $\omega_4>0$.
For $\omega_1,\omega_4>0$,
\begin{equation}
 \mathbf1_{\{w(t,\omega_4)>0\}}
 [w(t,\omega_1)-w(t,\omega_4)]
 \le w_+(t,\omega_1)-w_+(t,\omega_4).
\end{equation}
Using \eqref{eq:cond-jump-operator} and $S_{\beta,N}\ge0$, we obtain,
for almost every $(t,\omega_4)$,
\begin{equation}
 \begin{aligned}
 \partial_tw_+(t,\omega_4)
 &=\mathbf1_{\{w(t,\omega_4)>0\}}\partial_tw(t,\omega_4)\\
 &\le\mathbf1_{\{w(t,\omega_4)>0\}}L_N(t)[w(t,\cdot)](\omega_4)
       +V_N(t,\omega_4)w_+(t,\omega_4)\\
 &\le L_N(t)[w_+(t,\cdot)](\omega_4)
       +V_N^+(t,\omega_4)w_+(t,\omega_4).
 \end{aligned}
\end{equation}
Using \eqref{eq:cond-comparison-integral-bound} and
\eqref{eq:cond-bounded-cutoff-coefficients}, we obtain, for almost every $t$,
\begin{equation}\label{eq:cond-transfer-comparison}
 \begin{aligned}
 \frac{d}{dt}\int_0^\infty\sqrt\omega\,w_+(t,\omega)\,d\omega
 &\le\left(\frac{N(2N)^{3/2}}2+\|V_N^+(t,\cdot)\|_\infty\right)
       \int_0^\infty\sqrt\omega\,w_+(t,\omega)\,d\omega\\
 &\le\frac{5N(2N)^{3/2}}4
       \int_0^\infty\sqrt\omega\,w_+(t,\omega)\,d\omega.
 \end{aligned}
\end{equation}
Since $w_+(s,\cdot)=0$, Gronwall's inequality yields
\begin{equation}
 0\le\int_0^\infty\sqrt\omega\,w_+(t,\omega)\,d\omega
 \le e^{5N(2N)^{3/2}(t-s)/4}
 \int_0^\infty\sqrt\omega\,w_+(s,\omega)\,d\omega=0.
\end{equation}
Consequently, \eqref{eq:cond-hinge-budget} gives
\begin{equation}\label{eq:cond-transfer-density-growth}
 f_N(t,\omega_4)\ge v(t,\omega_4)
 \ge c_0(a/2)^{1-3\beta}e^{\lambda(t-s)-Ca^{-d}/\lambda}\psi_a(\omega_4).
\end{equation}
\end{proof}

\begin{lemma}[Transfer to a smaller interval]\label{lem:cond-smooth-transfer}
There exists $C_\beta>0$ such that, for every $0<r\le R/2<\infty$
and $0<\varepsilon<1$, there is an integer
$N_0=N_0(\beta,r,R,\varepsilon)\ge1$ with the following property.
For $N\ge N_0$, $s\ge0$, $\varepsilon<m\le1$, and $c_0>0$, set
$a=r/(8L)$ and
\begin{equation}\label{eq:cond-stage-duration}
 \tau=C_\beta R^d(R/r)^2\varepsilon^{-4}
 (1+|\log r|+|\log c_0|).
\end{equation}
If
\begin{equation}\label{eq:cond-transfer-initial-assumptions}
 \int_{[0,\infty)}(1-\omega/R)_+ g_s^N(d\omega)\ge m,
 \qquad
 f_N(s,\omega)\ge c_0\omega^{1-3\beta}
 \quad\text{for }a/2\le\omega\le3a/2,
\end{equation}
then $g_t^N([0,r])\ge m-\varepsilon$ for some $t\in[s,s+\tau]$.
\end{lemma}
\begin{proof}
Choose $N_0$ as at the beginning of this subsection.

Suppose to the contrary, that
$g_t^N([0,r])<m-\varepsilon$ for every $t\in[s,s+\tau]$.
Since $(1-\omega/R)_+$ is convex,
\cref{lem:ordered} and \eqref{eq:cond-transfer-initial-assumptions} give
\begin{equation}
 g_t^N([0,R])\ge\int_{[0,\infty)}(1-\omega/R)_+g_t^N(d\omega)
 \ge\int_{[0,\infty)}(1-\omega/R)_+g_s^N(d\omega)\ge m.
\end{equation}
Subtracting $g_t^N([0,r])$ yields
\begin{equation}\label{eq:cond-outer-mass}
 g_t^N([r,R])\ge g_t^N((r,R])>\varepsilon,
 \qquad s\le t\le s+\tau.
\end{equation}
In the rest of the proof, all estimates are made on \([s,s+\tau]\).
By \cref{lem:cond-potential-jump-lower,lem:cond-weighted-mass-integral,lem:cond-density-amplification},
\eqref{eq:cond-transfer-density-growth} holds for $s\le t\le s+\tau$.

On $[3a/4,5a/4]$, we have $\psi_a\ge1/2$ and $\sqrt\omega\ge\sqrt a/2$.
Since $(a/2)^{1-3\beta}\ge a^{1-3\beta}/2$, integrating
\eqref{eq:cond-transfer-density-growth} gives
\begin{equation}\label{eq:cond-transfer-mass-contradiction}
 \begin{aligned}
 1\ge\int_0^\infty\sqrt\omega\,f_N(t,\omega)\,d\omega
 &\ge c_0(a/2)^{1-3\beta}e^{\lambda(t-s)-Ca^{-d}/\lambda}
       \int_{3a/4}^{5a/4}\sqrt\omega\,\psi_a(\omega)\,d\omega\\
 &\ge\frac{c_0a^{5/2-3\beta}}{16}e^{\lambda(t-s)-Ca^{-d}/\lambda}.
 \end{aligned}
\end{equation}
Taking logarithms gives
\begin{equation}\label{eq:cond-occupation-bound}
 t-s\le\frac1\lambda\log\frac{16}{c_0a^{5/2-3\beta}}+\frac{Ca^{-d}}{\lambda^2}
 \le\frac C\lambda(1+|\log a|+|\log c_0|)+\frac{Ca^{-d}}{\lambda^2}.
\end{equation}
Substituting $a=r/(8L)$ and $d=1-4\beta$ into
\eqref{eq:cond-growth-rate}, and using $R/r\ge2$ and $0<\varepsilon<1$, we obtain
\begin{equation}\label{eq:cond-transfer-rate-scaling}
 \lambda^{-1}+\frac{a^{-d}}{\lambda^2}
 \le CR^d\bigl[(R/r)^{2\beta+1/2}\varepsilon^{-2}+(R/r)^2\varepsilon^{-4}\bigr]
 \le CR^d(R/r)^2\varepsilon^{-4}.
\end{equation}
Since $|\log a|\le|\log r|+\log(8L)$, choosing $C_\beta=2C$ in
\eqref{eq:cond-stage-duration} and setting $t=s+\tau$ in
\eqref{eq:cond-occupation-bound} gives
\begin{equation}
 \tau\le CR^d(R/r)^2\varepsilon^{-4}
 (1+|\log r|+|\log c_0|)=\frac\tau2<\tau.
\end{equation}
Since $\tau>0$, this contradicts the initial assumption.
Hence $g_t^N([0,r])\ge m-\varepsilon$ for some $t\in[s,s+\tau]$.
\end{proof}

\subsection{Iteration and proof of the condensation theorem}\label{subsec:cond-proof}

\begin{lemma}[Transfer to arbitrarily small intervals in finite time]
\label{lem:cond-summable-transfer}
Suppose that $g_0$ has a nonzero absolutely continuous part and
$\mathcal N_0=1$. There exists $C>0$, depending only on $\beta,g_0$,
such that, for every $0<\zeta\le1/2$, there exist
$R_n\downarrow0$, increasing times $s_n\uparrow T_\zeta<\infty$,
and integers $\widehat N_n\ge1$ such that, with
\[
 m_n=1-\frac\zeta2+\frac\zeta4\,2^{-nd/12},
 \qquad \chi_n(\omega)=(1-\omega/R_n)_+,
\]
\[
 \int_{[0,\infty)}\chi_n(\omega)g_t^N(d\omega)\ge m_n,
 \qquad n\ge0,\quad t\ge s_n,\quad N\ge\widehat N_n,
\]
and
\[
 1+T_\zeta\le C\zeta^{-40/d}.
\]
\end{lemma}
\begin{proof}
Fix $0<\zeta\le1/2$.
Let $c,D,s_*,r_*$ be as in \cref{lem:cond-density-lower}, with $s_*\ge1$.
Throughout this proof, $C>0$ denotes constants depending only on $\beta,g_0$.

\smallskip
\noindent\textit{Step 1. Iteration of the transfer estimate.}
Set
\begin{equation}\label{eq:cond-iteration-radii}
 R_0=1\vee\frac{4\mathcal E_0}{\zeta},\qquad
 R_1=(1/2\wedge r_*)\zeta^{7/d},\qquad
 R_n=R_1 2^{-(n-1)} \qquad n\geq 1,
\end{equation}
and, for $n\ge0$,
\begin{equation}\label{eq:cond-iteration-errors-and-masses}
 \varepsilon_n=\frac{1-2^{-d/12}}8\zeta 2^{-nd/12},\qquad
 r_n=\varepsilon_nR_{n+1},\qquad
 m_n=1-\frac\zeta2+\frac\zeta4\,2^{-nd/12}.
\end{equation}
Then
\begin{equation}
 m_0=1-\frac\zeta4,\qquad m_n-m_{n+1}=2\varepsilon_n,\qquad
 m_n\downarrow1-\frac\zeta2.
\end{equation}
In particular, $0<\varepsilon_n\le\zeta/8<m_n\le1$,
$0<r_n<R_{n+1}\le R_n/2$, and $r_n/(8L)<R_1<r_*$.

Let $\chi_n(\omega)=(1-\omega/R_n)_+$.
By \eqref{eq:ordered} and \eqref{eq:convex},
\begin{equation}\label{eq:cond-iteration-monotonicity}
 \begin{aligned}
 &\frac{d}{dt}\int_{[0,\infty)}\chi_n(\omega)g_t^N(d\omega)\\
 &\quad=\frac1{12}\iiint_{[0,\infty)^3}
 F_{\beta,N}[\chi_n](\operatorname{sort}(\omega_1,\omega_2,\omega_3))
 \,g_t^N(d\omega_1)g_t^N(d\omega_2)g_t^N(d\omega_3)\ge0.
 \end{aligned}
\end{equation}
Starting with $s_0=s_*$, define
\begin{equation}\label{eq:cond-summable-durations}
 s_{n+1}-s_n
 =C_\beta R_n^d(R_n/R_{n+1})^2\varepsilon_n^{-6}
 (1+|\log r_n|+|\log c|+Ds_n),\qquad n\ge0,
\end{equation}
where $C_\beta$ is the constant in \eqref{eq:cond-stage-duration}.
We prove by induction that, for each $n\ge0$, there is an integer
$\widehat N_n\ge1$ such that
\begin{equation}\label{eq:cond-stage-convex-mass}
 \int_{[0,\infty)}\chi_n(\omega)g_{s_n}^N(d\omega)\ge m_n,
 \qquad N\ge\widehat N_n.
\end{equation}
For $n=0$, we may take $\widehat N_0=1$, since
\begin{equation}\label{eq:cond-initial-low-mass}
 \begin{aligned}
 \int_{[0,\infty)}\chi_0(\omega)g_{s_*}^N(d\omega)
 &\ge\int_{[0,\infty)}\chi_0(\omega)g_0(d\omega)\\
 &\ge\int_{[0,\infty)}(1-\omega/R_0)g_0(d\omega)
 =1-\frac{\mathcal E_0}{R_0}\ge m_0.
 \end{aligned}
\end{equation}
Suppose that \eqref{eq:cond-stage-convex-mass} holds for $n$.
Choose an integer
$\widehat N_{n+1}\ge\widehat N_n\vee N_{r_n/(8L)}
\vee N_0(\beta,r_n,R_n,\varepsilon_n)$,
where the last two bounds are from \cref{lem:cond-density-lower}
and \cref{lem:cond-smooth-transfer}.
For $N\ge\widehat N_{n+1}$, \cref{lem:cond-density-lower} gives
\begin{equation}
 f_N(s_n,\omega)\ge ce^{-Ds_n}\omega^{1-3\beta},
 \qquad \frac{r_n}{16L}\le\omega\le\frac{3r_n}{16L}.
\end{equation}
Since $r_n=\varepsilon_nR_{n+1}$ and
$|\log c-Ds_n|\le|\log c|+Ds_n$,
\begin{equation}
 C_\beta R_n^d(R_n/r_n)^2\varepsilon_n^{-4}
 (1+|\log r_n|+|\log c-Ds_n|)\le s_{n+1}-s_n.
\end{equation}
Apply \cref{lem:cond-smooth-transfer} at $s_n$ with
$R=R_n$, $r=r_n$, $\varepsilon=\varepsilon_n$, $m=m_n$,
and $c_0=ce^{-Ds_n}$.
There is $t\in[s_n,s_{n+1}]$ such that
$g_t^N([0,r_n])\ge m_n-\varepsilon_n$.
On $[0,r_n]$, $\chi_{n+1}\ge1-r_n/R_{n+1}=1-\varepsilon_n$.
Thus \eqref{eq:cond-iteration-monotonicity} gives
\begin{equation}\label{eq:cond-iteration-hinge-step}
 \begin{aligned}
 \int_{[0,\infty)}\chi_{n+1}(\omega)g_{s_{n+1}}^N(d\omega)
 &\ge\int_{[0,\infty)}\chi_{n+1}(\omega)g_t^N(d\omega)\\
 &\ge(1-\varepsilon_n)g_t^N([0,r_n])
 \ge(1-\varepsilon_n)(m_n-\varepsilon_n)\\
 &=m_n-(1+m_n)\varepsilon_n+\varepsilon_n^2
 \ge m_n-2\varepsilon_n=m_{n+1}.
 \end{aligned}
\end{equation}
This proves \eqref{eq:cond-stage-convex-mass} for $n+1$.

\smallskip
\noindent\textit{Step 2. Summability of the transfer times.}
The definitions give $R_0\le C\zeta^{-1}$ and
$1+|\log r_0|\le C(1+|\log\zeta|)$.
Since $s_0=s_*$, \eqref{eq:cond-summable-durations} yields
\begin{equation}\label{eq:cond-iteration-first-wait}
 \begin{aligned}
 s_1-s_*&\le C R_0^{d+2}R_1^{-2}\varepsilon_0^{-6}(1+|\log\zeta|)\\
 &\le C\zeta^{-(d+2)}\zeta^{-14/d}\zeta^{-6}(1+|\log\zeta|)
 =C\zeta^{-(8+d+14/d)}(1+|\log\zeta|).
 \end{aligned}
\end{equation}
For $n\ge1$, $R_n/R_{n+1}=2$ and $r_n=r_0\,2^{-n(1+d/12)}$, so
\begin{equation}\label{eq:cond-iteration-tail-coefficients}
 \begin{aligned}
 R_n^d(R_n/R_{n+1})^2\varepsilon_n^{-6}
 &=4R_1^d2^{-d(n-1)}\varepsilon_0^{-6}2^{nd/2}
 \le C\zeta\,2^{-nd/2},\\
 1+|\log r_n|
 &\le1+|\log r_0|+n(1+d/12)\log2
 \le C(n+1)(1+|\log\zeta|).
 \end{aligned}
\end{equation}
Using $\zeta(1+|\log\zeta|)\le1$ and \eqref{eq:cond-summable-durations},
\begin{equation}\label{eq:cond-iteration-time-product}
 s_{n+1}-s_n
 \le C\zeta\,2^{-nd/2}\bigl[(n+1)(1+|\log\zeta|)+s_n\bigr]
 \le C(n+1)2^{-nd/2}(1+s_n).
\end{equation}
Iterating this inequality gives, for $k\ge2$,
\begin{equation}
 \begin{aligned}
 1+s_k
 &\le(1+s_1)\prod_{n=1}^{k-1}\bigl(1+C(n+1)2^{-nd/2}\bigr)\\
 &\le(1+s_1)\exp\!\left(C\sum_{n=1}^{\infty}(n+1)2^{-nd/2}\right)
 \le C(1+s_1).
 \end{aligned}
\end{equation}
Since $d>0$ and $s_n$ is increasing, $T_\zeta=\lim_{n\to\infty}s_n<\infty$.
By \eqref{eq:cond-iteration-first-wait},
\begin{equation}\label{eq:cond-common-time}
 1+T_\zeta\le C(1+s_1)
 \le C\zeta^{-(8+d+14/d)}(1+|\log\zeta|)
 \le C\zeta^{-40/d}.
\end{equation}
Here $1+|\log\zeta|\le C\zeta^{-1}$ and
$9+d+14/d<40/d$ for $0<d<1$.

By \eqref{eq:cond-iteration-monotonicity}, the lower bound
\eqref{eq:cond-stage-convex-mass} remains valid for every $t\ge s_n$.
This proves the lemma.
\end{proof}

\begin{proof}[Proof of \cref{thm:smooth-condensation}]
Assume first that $\mathcal N_0=1$ and fix $0<\zeta\le1/2$.
The case $1/2<\zeta<1$ follows by taking $\zeta=1/2$.
Let $R_n,s_n,T_\zeta,m_n,\widehat N_n$ and $\chi_n$ be as in
\cref{lem:cond-summable-transfer}.
Let $g^{N_j}\to g$ as in \eqref{eq:uniform-narrow-convergence},
and fix $t\ge T_\zeta$.
For each fixed $n$, $N_j\ge\widehat N_n$ for all sufficiently large $j$.
Since $\chi_n$ is bounded and continuous and $t\ge s_n$,
\eqref{eq:cond-iteration-monotonicity} and \eqref{eq:cond-stage-convex-mass} give
\begin{equation}\label{eq:cond-iteration-limit-hinge}
 \begin{aligned}
 \int_{[0,\infty)}\chi_n(\omega)g_t(d\omega)
 &=\lim_{j\to\infty}\int_{[0,\infty)}\chi_n(\omega)g_t^{N_j}(d\omega)\\
 &\ge\liminf_{j\to\infty}\int_{[0,\infty)}\chi_n(\omega)g_{s_n}^{N_j}(d\omega)
 \ge m_n.
 \end{aligned}
\end{equation}
Now $R_n\downarrow0$ and $0\le\chi_n\le1$, with
$\chi_n(\omega)\to\mathbf1_{\{0\}}(\omega)$.
Dominated convergence yields
\begin{equation}\label{eq:cond-limit-atom-bound}
 g_t(\{0\})
 =\lim_{n\to\infty}\int_{[0,\infty)}\chi_n(\omega)g_t(d\omega)
 \ge\lim_{n\to\infty}m_n
 =1-\frac\zeta2\ge1-\zeta,\qquad t\ge T_\zeta.
\end{equation}
Since the bounds above are independent of the chosen subsequence, they hold
for every constructed solution from $g_0$.

Since $g_t([0,\infty))=1$ and the measures
$g_t|_{(0,\infty)}$ and $(1-g_t(\{0\}))\delta_0$ are mutually singular,
\begin{equation}\label{eq:cond-tv-atom-identity}
 \|g_t-\delta_0\|_{\mathrm{TV}}
 =g_t((0,\infty))+1-g_t(\{0\})
 =2(1-g_t(\{0\})).
\end{equation}
Fix $C\ge1$ such that $T_\zeta\le C\zeta^{-40/d}$ in
\eqref{eq:cond-common-time}.
For $t\ge C2^{40/d}$, take $\zeta=(C/t)^{d/40}\le1/2$.
Then $T_\zeta\le t$, and
\begin{equation}
 \|g_t-\delta_0\|_{\mathrm{TV}}
 =2(1-g_t(\{0\}))\le\zeta=C^{d/40}t^{-d/40}.
\end{equation}
For $0\le t<C2^{40/d}$, use $\|g_t-\delta_0\|_{\mathrm{TV}}\le2$.
Increasing $C$, we obtain
\begin{equation}
 \|g_t-\delta_0\|_{\mathrm{TV}}\le C(1+t)^{-d/40},\qquad t\ge0.
\end{equation}

For general mass $\mathcal N_0>0$, use
$\widetilde g_\tau=\mathcal N_0^{-1}g_{\tau/\mathcal N_0^2}$
from \eqref{eq:cond-mass-normalization}.
If $\widetilde T_\zeta$ is the time obtained for initial measure
$g_0/\mathcal N_0$, the unit-mass bounds give
\begin{equation}
 g_t(\{0\})=\mathcal N_0\widetilde g_{\mathcal N_0^2t}(\{0\})
 \ge(1-\zeta)\mathcal N_0,
 \qquad t\ge\widetilde T_\zeta/\mathcal N_0^2,
\end{equation}
and
\begin{equation}
 \|g_t-\mathcal N_0\delta_0\|_{\mathrm{TV}}
 =\mathcal N_0\|\widetilde g_{\mathcal N_0^2t}-\delta_0\|_{\mathrm{TV}}
 \le C\mathcal N_0(1+\mathcal N_0^2t)^{-d/40}.
\end{equation}
Here $C$ depends only on $\beta$ and $g_0/\mathcal N_0$.
This proves \eqref{eq:cond-general-conclusion} and \eqref{eq:cond-tv-rate}.
\end{proof}

\begin{proof}[Proof of \cref{thm:action}]
For $1/4\le\beta\le3/4$, \cref{thm:measure-main} and
\cref{thm:action-preservation} give, for every constructed solution,
\begin{equation}
 g_t(\{0\})=g_0(\{0\}),\qquad t\ge0.
\end{equation}
This proves the first assertion.

For $0<\beta<1/4$, suppose that $g_0$ has a nonzero absolutely continuous part.
Then $\mathcal N_0>0$, and \cref{thm:smooth-condensation}, with $d=1-4\beta$, gives
\begin{equation}
 \begin{aligned}
 g_t(\{0\})&\ge(1-\zeta)\mathcal N_0,\qquad t\ge T(\beta,g_0,\zeta),\\
 \|g_t-\mathcal N_0\delta_0\|_{\mathrm{TV}}
 &\le C\mathcal N_0(1+\mathcal N_0^2t)^{-(1-4\beta)/40},\qquad t\ge0.
 \end{aligned}
\end{equation}
For each $0<\zeta<1$, the time $T$ is finite and the same for all constructed
solutions; $C$ depends only on $\beta$ and $g_0/\mathcal N_0$.
These are \eqref{eq:main-condensation} and \eqref{eq:main-condensation-rate}.
If $g_0(\{0\})=0$, taking $\zeta=1/2$ gives
$g_T(\{0\})\ge\mathcal N_0/2>0$, proving finite-time condensation.
\end{proof}

\appendix

\section{Proofs of kernel estimates}
\label[appendix]{app:measure-estimates}

\begin{proof}[Proof of \cref{lem:ordered}]
\noindent\textit{(i) Symmetrization.}
For $0<x\le y\le z$, 
\begin{align}
(K_{\beta,N}\Delta\phi)(x,y,z)
&=(K_{\beta,N}\Delta\phi)(y,x,z)=G_{1,N}\Delta_1\phi,\\
(K_{\beta,N}\Delta\phi)(x,z,y)
&=(K_{\beta,N}\Delta\phi)(z,x,y)=G_{2,N}\Delta_2\phi,\\
(K_{\beta,N}\Delta\phi)(y,z,x)
&=(K_{\beta,N}\Delta\phi)(z,y,x)=G_{3,N}\Delta_3\phi.
\end{align}
The same equalities hold with $K_\beta$ and $G_i$ in place
of $K_{\beta,N}$ and $G_{i,N}$.
Writing $\mathcal S_3$ for the permutations of $\{1,2,3\}$,
we obtain
\begin{align}
\frac16\sum_{\sigma\in\mathcal S_3}
(K_{\beta,N}\Delta\phi)
(\omega_{\sigma(1)},\omega_{\sigma(2)},\omega_{\sigma(3)})
&=\frac13\bigl(F_{\beta,N}[\phi]\circ\operatorname{sort}\bigr)
(\omega_1,\omega_2,\omega_3),\\
\frac16\sum_{\sigma\in\mathcal S_3}
(K_\beta\Delta\phi)
(\omega_{\sigma(1)},\omega_{\sigma(2)},\omega_{\sigma(3)})
&=\frac13\bigl(F_\beta[\phi]\circ\operatorname{sort}\bigr)
(\omega_1,\omega_2,\omega_3).
\end{align}
Since both sides vanish when
$\min\{\omega_1,\omega_2,\omega_3\}=0$,
the identities hold on $[0,\infty)^3$.

Let $|\phi(\omega)|\le C_\phi(1+\omega)$. With
$C_N=N(2N)^{3/2}$, the estimate in
\cref{prop:integrated-cutoff-existence} gives
\[
\frac14\iiint_{[0,\infty)^3}
K_{\beta,N}|\Delta\phi|\,d(g_t^N)^{\otimes3}
\le C_\phi C_N\mathcal N_0^2
(\mathcal N_0+\mathcal E_0)<\infty.
\]
Since $(g_t^N)^{\otimes3}$ is invariant under
permutations, 
\begin{align*}
\dfrac{d}{dt}\int_{[0,\infty)}\phi\,dg_t^N
&=\frac14\iiint_{[0,\infty)^3}
K_{\beta,N}\Delta\phi\,d(g_t^N)^{\otimes3}\\
&=\frac1{24}\sum_{\sigma\in\mathcal S_3}
\iiint_{[0,\infty)^3}
(K_{\beta,N}\Delta\phi)
(\omega_{\sigma(1)},\omega_{\sigma(2)},\omega_{\sigma(3)})
\,d(g_t^N)^{\otimes3}\\
&=\frac1{12}\iiint_{[0,\infty)^3}
\bigl(F_{\beta,N}[\phi]\circ\operatorname{sort}\bigr)
\,d(g_t^N)^{\otimes3}\\
&=\frac1{12}\int_{\mathcal D_{\mathrm{ord}}}
F_{\beta,N}[\phi](x,y,z)\,
d\bigl(\operatorname{sort}_\#(g_t^N)^{\otimes3}\bigr).
\end{align*}
This proves \eqref{eq:ordered}. For the uncut kernel, the
assumed absolute integrability permits the same permutation
average. The invariance of $g^{\otimes3}$ under permutations
and the identity proved above give the stated formula for
$F_\beta[\phi]$.

\smallskip
\noindent\textit{(ii) Convex test functions.}
Let $\phi$ be convex. Then for $0\le u\le v$ and \(a,b\in [u,v]\) such that \(a+b=u+v\),
\begin{equation}
\phi(a)+\phi(b)\leq \phi(u)+\phi(v).
\end{equation}
For $0<x\le y\le z$, we have
\begin{align*}
&x,y\in[x+y-z,z],
&&x+y=(x+y-z)+z,
&&\qquad x+y>z,\\
&y,x+z-y\in[x,z],
&&y+(x+z-y)=x+z,\\
&y,z\in[x,y+z-x],
&&y+z=x+(y+z-x).
\end{align*}
Applying the preceding convexity inequality to these three
pairs gives
\begin{align}
\Delta_1\phi
&=\phi(z)+\phi(x+y-z)-\phi(x)-\phi(y)\ge0,
\qquad x+y>z,\\
\Delta_2\phi
&=\phi(y)+\phi(x+z-y)-\phi(x)-\phi(z)\le0,\\
\Delta_3\phi
&=\phi(x)+\phi(y+z-x)-\phi(y)-\phi(z)\ge0.
\end{align}
For $x+y\le z$, we have $\Delta_1\phi=0$ by definition.
Since $(x+z-y)+(y+z-x)=2z$, convexity also gives
\begin{equation}
\Delta_2\phi+\Delta_3\phi
=\phi(x+z-y)+\phi(y+z-x)-2\phi(z)
\ge2\phi(z)-2\phi(z)=0.
\end{equation}
For these fixed $x,y,z$, \eqref{eq:cutoff} gives
\begin{equation}
S_{\beta,N}(x,y,z,\omega_4)
=\bigl((xyz)^\beta\omega_4^\beta
(\sqrt{x} \wedge \sqrt{\omega_4}) \wedge N\bigr)
\ell_N(\omega_4)\ell_N(x)\ell_N(y)\ell_N(z),
\qquad \omega_4\ge0,
\end{equation}
and every factor depending on $\omega_4$ is nonnegative and
nondecreasing. Since $y+z-x\ge x+z-y$, we obtain
\begin{align}
G_{3,N}-G_{2,N}
&=\frac{S_{\beta,N}(x,y,z,y+z-x)
-S_{\beta,N}(x,y,z,x+z-y)}{\sqrt{xyz}}\ge0,\\
G_3-G_2
&=x^\beta(yz)^{\beta-1/2}
\bigl((y+z-x)^\beta-(x+z-y)^\beta\bigr)\ge0.
\end{align}
Hence
\begin{align}
F_{\beta,N}[\phi]
&=G_{1,N}\Delta_1\phi
+G_{2,N}(\Delta_2\phi+\Delta_3\phi)
+(G_{3,N}-G_{2,N})\Delta_3\phi\ge0,\\
F_\beta[\phi]
&=G_1\Delta_1\phi
+G_2(\Delta_2\phi+\Delta_3\phi)
+(G_3-G_2)\Delta_3\phi\ge0.
\end{align}
Since $\Delta_2\phi\le0$ and $G_{2,N},G_2\ge0$, we also have
\begin{align*}
F_{\beta,N}[\phi]
&=G_{1,N}\Delta_1\phi+G_{3,N}\Delta_3\phi
+G_{2,N}\Delta_2\phi
\le G_{1,N}\Delta_1\phi+G_{3,N}\Delta_3\phi,\\
F_\beta[\phi]
&=G_1\Delta_1\phi+G_3\Delta_3\phi+G_2\Delta_2\phi
\le G_1\Delta_1\phi+G_3\Delta_3\phi.
\end{align*}
On $\{x=0\}$, all coefficients vanish, so both inequalities
hold there as well. This proves \eqref{eq:convex} throughout
$\mathcal D_{\mathrm{ord}}$.
\end{proof}

\begin{proof}[Proof of \cref{lem:compact}]
Choose $L\ge1$ such that $\phi=0$ on $[L,\infty)$.
Throughout the proof, $C$ denotes a positive constant depending
only on $\beta$, $L$, $\|\phi\|_\infty$, and
$\operatorname{Lip}(\phi)$, whose value may change from line to line.

Since $0<\beta\le3/4$,
\begin{equation}
\beta-\frac12\le\frac{3\beta-1}{2}
\le2\beta-\frac12\le1,
\qquad
\frac{3\beta-1}{2}\le5/8 \wedge \beta.
\end{equation}
Hence
$0\le q\le p\le1$, $q\le5/8$, and $q\le\beta$.
We also use $q\ge\beta-1/2$ and $3\beta-1\le2q$.

Since all the coefficients vanish at $x=0$, without loss of generality we assume
$0<x\le y\le z$.
Since \(G_{i,N}\leq G_i\), it suffices to prove
\eqref{eq:raw-test-bound} for \(G_i\). We then prove
\eqref{eq:compact-envelope} and \eqref{eq:compact-small-frequency}
for both the cutoff and uncut coefficients.

\smallskip
\noindent\textit{Step 1. Estimate for \(G_1|\Delta_1\phi|\).}
If $G_1|\Delta_1\phi|\ne0$, then
$0<x+y-z\le x \wedge L$.
Indeed, $x+y-z$ is the smallest frequency,
and all four test values vanish when $x+y-z\ge L$.
Since $z-y=x-(x+y-z)$, Lipschitz continuity gives
\begin{align}
|\Delta_1\phi|\le|\phi(z)-\phi(y)|+|\phi(x+y-z)-\phi(x)|\le2\operatorname{Lip}(\phi)(z-y)\le2\operatorname{Lip}(\phi)x.
\end{align}
Together with $|\Delta_1\phi|\le4\|\phi\|_\infty$, this yields
$|\Delta_1\phi|\le C(x \wedge 1)$.
Since
$x \wedge L\le(1+L)x/(1+x)$ and
$x \wedge 1\le2x/(1+x)$,
\begin{equation}
G_1|\Delta_1\phi|\leq C(xyz)^{\beta-1/2}(x \wedge L)^{\beta+1/2}(x \wedge 1)
\le C(xyz)^{\beta-1/2}
\left(\frac{x}{1+x}\right)^{\beta+3/2}.
\end{equation}
For every $u\ge x$, we have
\[
u^{\beta-1/2}
=(1+u)^{\beta-1/2}
\left(\frac{u}{1+u}\right)^\beta
\sqrt{\frac{1+u}{u}}
\le(1+u)^{\beta-1/2}\sqrt{\frac{1+x}{x}}.
\]
Applying this inequality with $u=x,y,z$ and using
$\beta-1/2\le q$, we obtain
\begin{align}
G_1|\Delta_1\phi|
&\le C\left(\frac{x}{1+x}\right)^\beta
\bigl((1+x)(1+y)(1+z)\bigr)^{\beta-1/2}\notag\\
&\le C\left(\frac{x}{1+x}\right)^\beta
(1+x)^q(1+y)^q(1+z)^q.
\end{align}

\smallskip
\noindent\textit{Step 2. Estimate of \(G_2|\Delta_2\phi|+G_3|\Delta_3\phi|\).}
Since every argument in $\Delta_2\phi$ and $\Delta_3\phi$ is at least
$x$, both increments vanish when $x\ge L$.
We may therefore assume $0<x\le L$.
Hence
\begin{align}
|\Delta_2\phi|
&\le|\phi(y)-\phi(x)|+|\phi(x+z-y)-\phi(z)|
\le2\operatorname{Lip}(\phi)(y-x),\\
|\Delta_3\phi|
&\le|\phi(x)-\phi(y)|+|\phi(y+z-x)-\phi(z)|
\le2\operatorname{Lip}(\phi)(y-x).
\end{align}
Combining these bounds with $|\Delta_i\phi|\le4\|\phi\|_\infty$
gives $|\Delta_i\phi|\le C(y \wedge 1)$ for $i=2,3$.
Moreover,
\begin{equation}
0\le G_2\le G_3
=x^\beta(yz)^{\beta-1/2}(y+z-x)^\beta
\le2^\beta x^\beta y^{\beta-1/2}z^{2\beta-1/2}.
\end{equation}
Consequently,
\begin{equation}
\sum_{i=2}^3G_i|\Delta_i\phi|
\le Cx^\beta y^{\beta-1/2}z^{2\beta-1/2}(y \wedge 1).
\end{equation}

Since $y \wedge 1\le2y/(1+y)$ and
$0<y/(1+y)\le z/(1+z)<1$, we have
\begin{align*}
\frac{y^{\beta-1/2}z^{2\beta-1/2}(y \wedge 1)}
{(1+y)^{\beta-1/2}(1+z)^{2\beta-1/2}}
\le2\left(\frac{y}{1+y}\right)^{\beta+1/2}
       \left(\frac{z}{1+z}\right)^{2\beta-1/2}\le2\left(\frac{z}{1+z}\right)^{3\beta}\le2.
\end{align*}
It follows that
\begin{align}
\sum_{i=2}^3G_i|\Delta_i\phi|\le Cx^\beta(1+y)^{\beta-1/2}(1+z)^{2\beta-1/2}\le CL^\beta(1+y)^p(1+z)^p,
\end{align}
where the last inequality uses
$p\ge2\beta-1/2\ge\beta-1/2$.
Together with Step~1, $q\le p$,
\begin{equation}
\sum_{i=1}^3G_{i,N}|\Delta_i\phi|
\le\sum_{i=1}^3G_i|\Delta_i\phi|
\le C(1+x)^p(1+y)^p(1+z)^p,
\end{equation}
which is \eqref{eq:raw-test-bound}.

\smallskip
\noindent\textit{Step 3. Estimates with exponent $q$.}
It remains to prove \eqref{eq:compact-envelope} and
\eqref{eq:compact-small-frequency}.
Since Step~1 already gives the required bound for \(G_1\Delta_1\phi\),
we then estimate
$G_{2,N}\Delta_2\phi+G_{3,N}\Delta_3\phi$
and $G_2\Delta_2\phi+G_3\Delta_3\phi$.
Again, we assume
$0<x\le L$, since they both vanish when \(x>L.\)

Suppose first that $z\le2(y+L)$.
Then $1\le(1+z)/(1+y)\le2L+1$, and hence
\begin{align*}
(1+y)^{\beta-1/2}(1+z)^{2\beta-1/2}
&=(1+y)^{3\beta-1}
\left(\frac{1+z}{1+y}\right)^{2\beta-1/2}\\
&\le(2L+1)^p(1+y)^{3\beta-1}\\
&\le C(1+y)^{2q}
\le C(1+y)^q(1+z)^q.
\end{align*}
The estimate from Step~2 therefore gives
\begin{equation}
\sum_{i=2}^3G_i|\Delta_i\phi|
\le Cx^\beta(1+y)^q(1+z)^q.
\end{equation}

Suppose now that $z>2(y+L)$.
Since $x\le y$, we have $x+y<z$, and
\(
z/2<x+z-y\le y+z-x<2z.
\)
Thus
\begin{equation}
G_1=G_{1,N}=0,\qquad
\Delta_2\phi=\phi(y)-\phi(x),\qquad
\Delta_3\phi=\phi(x)-\phi(y),
\end{equation}
and
\begin{equation}
F_{\beta,N}[\phi]
=(G_{3,N}-G_{2,N})(\phi(x)-\phi(y)),
\qquad
F_\beta[\phi]=(G_3-G_2)(\phi(x)-\phi(y)).
\end{equation}

For every $s\in[x+z-y,y+z-x]$, we have $s\ge x$ and
$s>z/2>L\ge1$.
Therefore $\ell_N(s)=1$ for every $N\ge1$, and
\begin{equation}
S_{\beta,N}(x,y,z,s)
=\ell_N(x)\ell_N(y)\ell_N(z)
((xyz)^\beta\sqrt{x}\,s^\beta \wedge N).
\end{equation}
Using $0\le\ell_N\le1$ and
$0\le(v \wedge N)-(u \wedge N)\le v-u$ for $0\le u\le v$,
we obtain
\begin{align}
0\le G_{3,N}-G_{2,N}\le
\frac{(xyz)^\beta\sqrt{x}}{\sqrt{xyz}}
\bigl((y+z-x)^\beta-(x+z-y)^\beta\bigr)=G_3-G_2.
\end{align}
Since $\beta<1$ and $x+z-y>z/2$,
\begin{equation}\label{eq:ordered-coefficient-difference}
\begin{aligned}
G_3-G_2
&=\beta x^\beta(yz)^{\beta-1/2}
\int_{x+z-y}^{y+z-x}s^{\beta-1}\,ds\\
&\le2\beta x^\beta(yz)^{\beta-1/2}
(z/2)^{\beta-1}(y-x)\\
&=2^{2-\beta}\beta
x^\beta y^{\beta-1/2}z^{2\beta-3/2}(y-x).
\end{aligned}
\end{equation}
Consequently,
\begin{align}
|F_{\beta,N}[\phi]|
&\le(G_3-G_2)|\phi(x)-\phi(y)|
=|F_\beta[\phi]|\notag\\
&\le Cx^\beta y^{\beta-1/2}z^{2\beta-3/2}(y-x)
\le Cx^\beta y^{\beta+1/2}z^{2\beta-3/2}.
\end{align}
Here we used $|\phi(x)-\phi(y)|\le2\|\phi\|_\infty$.
Finally, $q\le\beta$, $3\beta-1\le2q$, and $z>1$ give
\begin{align*}
y^{\beta+1/2}z^{2\beta-3/2}
=y^q(y/z)^{\beta+1/2-q}z^{3\beta-1-q}\le y^qz^{3\beta-1-q}
\le y^qz^q
\le(1+y)^q(1+z)^q.
\end{align*}

We have therefore bounded the absolute value of the sum of
the second and third terms by
$Cx^\beta(1+y)^q(1+z)^q$ whenever $x\le L$,
for both the cutoff and uncut coefficients.
For the first term, Step~1 gives \eqref{eq:compact-envelope}
because $(x/(1+x))^\beta\le1$. It also gives
\eqref{eq:compact-small-frequency}, since
$(x/(1+x))^\beta(1+x)^q\le2^q x^\beta$ for $0<x\le1$.
For the other two terms, the preceding bound gives
\eqref{eq:compact-envelope} because
$x^\beta\le L^\beta(1+x)^q$ on $0<x\le L$,
and it gives \eqref{eq:compact-small-frequency} directly.
These terms vanish for $x\ge L$, so both estimates hold
throughout $\mathcal D_{\mathrm{ord}}$.
\end{proof}

\begin{proof}[Proof of \cref{cor:boundary-removal}]

\smallskip
\noindent\textit{Step 1. Continuity and local uniform convergence.}
Fix $R\ge1$ and $0<\varepsilon\le1$.
For $\varepsilon\le x\le y\le z\le R$, we have
\begin{equation}
x\le x+z-y\le y+z-x\le2R,
\qquad
\sup_{0\le s\le2R}S_\beta(x,y,z,s)
\le(2R)^{4\beta+1/2}.
\end{equation}
Thus, if $N\varepsilon\ge1$ and $N>(2R)^{4\beta+1/2}$,
\eqref{eq:cutoff} gives
\begin{equation}
G_{2,N}=G_2,\qquad G_{3,N}=G_3,\qquad
G_{1,N}=\ell_N((x+y-z)_+)G_1.
\end{equation}
Consequently,
\begin{equation}
|F_{\beta,N}[\phi]-F_\beta[\phi]|
=\bigl(1-\ell_N((x+y-z)_+)\bigr)G_1|\Delta_1\phi|.
\end{equation}
Using $\ell_N(s)=1$ for $s\ge N^{-1}$,
$G_1=(xyz)^{\beta-1/2}(x+y-z)_+^{\beta+1/2}$,
and $|\Delta_1\phi|\le4\|\phi\|_\infty$, we obtain
\begin{equation}
\sup_{\varepsilon\le x\le y\le z\le R}
|F_{\beta,N}[\phi]-F_\beta[\phi]|
\le C_{\varepsilon,R,\phi,\beta}N^{-\beta-1/2}.
\end{equation}
For $0\le x\le\varepsilon$ and $x\le y\le z\le R$,
\eqref{eq:compact-small-frequency} gives
\begin{equation}
|F_{\beta,N}[\phi]-F_\beta[\phi]|
\le |F_{\beta,N}[\phi]|+|F_\beta[\phi]|
\le C_{\phi,\beta}(1+R)^{2q}\varepsilon^\beta.
\end{equation}
Therefore,
\begin{equation}
\sup_{0\le x\le y\le z\le R}
|F_{\beta,N}[\phi]-F_\beta[\phi]|
\le C_{\varepsilon,R,\phi,\beta}N^{-\beta-1/2}
+C_{\phi,\beta}(1+R)^{2q}\varepsilon^\beta.
\end{equation}
Letting $N\to\infty$ and then $\varepsilon\downarrow0$
proves local uniform convergence on $\mathcal D_{\mathrm{ord}}$.

Fix $N\ge1$, one can check $F_{\beta,N}[\phi]\in C(\mathcal D_{\mathrm{ord}})$.
Local uniform convergence now gives
$F_\beta[\phi]\in C(\mathcal D_{\mathrm{ord}})$.

\smallskip
\noindent\textit{Step 2. Integral estimates.}
Let $g\ge0$ satisfy $\|g\|_{\mathcal M_1}\le M$, and set
$\mu=\operatorname{sort}_\#g^{\otimes3}$. Hence
\begin{equation}
\int_{\mathcal D_{\mathrm{ord}}}
(1+x)(1+y)(1+z)\,d\mu
=
\left(\int_{[0,\infty)}(1+\omega)\,dg\right)^3
\le M^3.
\end{equation}

Applying Tonelli's theorem, we obtain
\begin{align}
\iiint_{[0,\infty)^3}
K_\beta|\Delta\phi|\,dg^{\otimes3}
&=\frac13\int_{\mathcal D_{\mathrm{ord}}}
\sum_{i=1}^3G_i|\Delta_i\phi|\,d\mu\notag\\
&\le C_{\phi,\beta}
\int_{\mathcal D_{\mathrm{ord}}}
(1+x)^p(1+y)^p(1+z)^p\,d\mu\notag\\
&\le C_{\phi,\beta}
\int_{\mathcal D_{\mathrm{ord}}}
(1+x)(1+y)(1+z)\,d\mu\notag\\
&\le C_{\phi,\beta}M^3,
\end{align}
where we used \eqref{eq:raw-test-bound} and $p\le1$.
This proves \eqref{eq:absolute-collision-bound}.

For $z>R$, since $0\le q\le1$,
\begin{equation}
(1+x)^q(1+y)^q(1+z)^q
\le (1+R)^{q-1}(1+x)(1+y)(1+z).
\end{equation}
Hence \eqref{eq:compact-envelope} yields
\begin{align}
\int_{\mathcal D_{\mathrm{ord}}\cap\{z>R\}}
|F_\beta[\phi]|\,d\mu
\le C_{\phi,\beta}(1+R)^{q-1}
\int_{\mathcal D_{\mathrm{ord}}}
(1+x)(1+y)(1+z)\,d\mu\le C_{\phi,\beta}M^3(1+R)^{q-1}.
\end{align}
This proves \eqref{eq:compact-high-tail}.

Finally, $0\le K_{\beta,N}\le K_\beta$ gives the same absolute
integrability bound for $K_{\beta,N}$.
Estimate \eqref{eq:compact-high-tail} for $F_{\beta,N}[\phi]$
follows by applying \eqref{eq:compact-envelope}
in the preceding argument.
\end{proof}

\begin{proof}[Proof of \cref{lem:cond-kernel-upper}]
Fix $0<\beta<1/4$ and $d=1-4\beta$.
For $0<x\le y\le z$ with $x+y>z$, we have
$0<x+y-z\le x$ and $x+y-z\le x+z-y$. Hence
\begin{equation}
 \frac{G_1(x,y,z)}{G_2(x,y,z)}
 =\left(\frac{x+y-z}{x}\right)^{1/2}
  \left(\frac{x+y-z}{x+z-y}\right)^\beta\le1.
\end{equation}
If $x+y\le z$, then $G_1(x,y,z)=0$.
Since $x+z-y\le y+z-x\le2z$,
\eqref{eq:cond-G2} and \eqref{eq:cond-G3} give
\begin{equation}\label{eq:cond-ordered-kernel-upper}
 0\le G_1(x,y,z)\le G_2(x,y,z)\le G_3(x,y,z)
 \le2^\beta x^\beta y^{\beta-1/2}z^{2\beta-1/2}.
\end{equation}

For part \textup{(i)}, set
$(x,y,z)=\operatorname{sort}(\omega_1,\omega_2,\omega_3)$.
The symmetry of $K_\beta$ in its first two arguments gives
$K_\beta(\omega_1,\omega_2,\omega_3)
\in\{G_1(x,y,z),G_2(x,y,z),G_3(x,y,z)\}$.
Since $2\beta-1/2<0$ and $x\le y\le z$,
\begin{equation}\label{eq:cond-median-power-bound}
 2^\beta x^\beta y^{\beta-1/2}z^{2\beta-1/2}
 \le2^\beta x^\beta y^{3\beta-1}
 \le2^\beta y^{4\beta-1}=2^\beta y^{-d}.
\end{equation}
Combining \eqref{eq:cond-median-power-bound} with \eqref{eq:cond-ordered-kernel-upper}
and $0\le K_{\beta,N}\le K_\beta$ proves
\eqref{eq:condensation-kernel-bound}.

For part \textup{(ii)}, use $x^\beta\le y^\beta$ in
\eqref{eq:cond-ordered-kernel-upper}. Since $2\beta-1/2=-d/2$
and $\omega_i\omega_j\le yz$ for distinct $i,j\in\{1,2,3\}$,
\begin{equation}
 K_\beta(\omega_1,\omega_2,\omega_3)
 \le2^\beta(yz)^{-d/2}
 \le2^\beta(\omega_i\omega_j)^{-d/2}.
\end{equation}
To prove the estimate with $\omega_4^{-d/2}$, use
$xyz=\omega_1\omega_2\omega_3$ in \eqref{eq:kernel}:
\begin{equation}\label{eq:cond-negative-moment-ratio}
 \frac{K_\beta(\omega_1,\omega_2,\omega_3)\omega_4^{-d/2}}
      {(\omega_1\omega_2\omega_3)^{-d/2}}
 =\frac{\omega_4^{3\beta-1/2}\sqrt{x \wedge \omega_4}}
       {(xyz)^\beta}.
\end{equation}
If $0<\omega_4\le x$, then
\begin{equation}\label{eq:cond-negative-moment-small-output}
 \frac{\omega_4^{3\beta-1/2}\sqrt{x \wedge \omega_4}}
      {(xyz)^\beta}
 =\left(\frac{\omega_4^3}{xyz}\right)^\beta\le1.
\end{equation}
If $\omega_4\ge x$ and $3\beta\le1/2$, then
\begin{equation}\label{eq:cond-negative-moment-low-beta}
 \frac{\omega_4^{3\beta-1/2}\sqrt{x}}{(xyz)^\beta}
 \le\frac{x^{3\beta}}{(xyz)^\beta}\le1.
\end{equation}
If $\omega_4\ge x$ and $3\beta>1/2$, use
$\omega_4=\omega_1+\omega_2-\omega_3\le2z$ to obtain
\begin{equation}\label{eq:cond-negative-moment-high-beta}
 \frac{\omega_4^{3\beta-1/2}\sqrt{x}}{(xyz)^\beta}
 \le2^{3\beta-1/2}
 \left(\frac{x}{y}\right)^{1/2-\beta}
 \left(\frac{y}{z}\right)^{1/2-2\beta}
 \le2^{3\beta-1/2}.
\end{equation}
Here $1/2-\beta>0$, $1/2-2\beta>0$, and $2^{3\beta-1/2}<2$.
Substituting \eqref{eq:cond-negative-moment-small-output}, \eqref{eq:cond-negative-moment-low-beta}, and \eqref{eq:cond-negative-moment-high-beta}
into \eqref{eq:cond-negative-moment-ratio} proves the estimate
with $\omega_4^{-d/2}$ in \eqref{eq:cond-negative-moment-kernel}.

For part \textup{(iii)}, set
$(x,y,z)=\operatorname{sort}(\omega,\omega_1,\omega_2)$.
The assumption $\omega_1 \vee \omega_2\ge\omega/2$ gives
$y\ge\omega/2$, and $z\ge\omega$.
Both $K_\beta(\omega,\omega_1,\omega_2)$ and
$K_\beta(\omega,\omega_2,\omega_1)$ are at most $G_3(x,y,z)$.
As $\beta-1/2<0$ and $2\beta-1/2<0$,
\begin{align}
 &K_\beta(\omega,\omega_1,\omega_2)
  +K_\beta(\omega,\omega_2,\omega_1)\notag\\
 &\quad\le2^{\beta+1}x^\beta y^{\beta-1/2}z^{2\beta-1/2}
 \le2^{\beta+1}x^\beta(\omega/2)^{\beta-1/2}
                  \omega^{2\beta-1/2}\notag\\
 &\quad=2^{3/2}x^\beta\omega^{3\beta-1}
 =2^{3/2}\omega^{-d}(x/\omega)^\beta.
\end{align}
Since $x=\min\{\omega,\omega_1,\omega_2\}$, this proves
\eqref{eq:cond-potential-negative-kernel}.
Finally, $K_{\beta,N}\le K_\beta$ gives
\eqref{eq:cond-negative-moment-kernel} and \eqref{eq:cond-potential-negative-kernel}
for $K_{\beta,N}$.
\end{proof}

\begin{proof}[Proof of \cref{lem:kernel}]
Let $0<2x\le y\le z$. Since $2(y-x)\ge y$ and
$x+z-y\le y+z-x\le2z$, we have
\begin{align}
 (y+z-x)^\beta-(x+z-y)^\beta
 &=\beta\int_{x+z-y}^{y+z-x}u^{\beta-1}\,du\notag\\
 &\ge2\beta(y-x)(2z)^{\beta-1}
 \ge\beta2^{\beta-1}yz^{\beta-1}.
 \label{eq:cond-kernel-power-difference}
\end{align}
Using \eqref{eq:cond-G2} and \eqref{eq:cond-G3}, we obtain
\begin{align}
 G_3(x,y,z)-G_2(x,y,z)
 &=x^\beta(yz)^{\beta-1/2}
   \bigl[(y+z-x)^\beta-(x+z-y)^\beta\bigr]\notag\\
 &\ge\beta2^{\beta-1}x^\beta y^{\beta+1/2}z^{2\beta-3/2}.
\end{align}
This proves \eqref{eq:cond-kernel-gain-difference}.

Now assume $Lx\le y\le z$.
If $z-y\ge x$, then $G_1(x,y,z)=0$ by \eqref{eq:cond-G1},
and \eqref{eq:cond-high-potential} follows from
\eqref{eq:cond-kernel-gain-difference}.
If $z-y<x$, then
$0<x+y-z\le x$, $x+z-y<2x$, and $y+z-x\ge y$.
By \eqref{eq:cond-G1}, \eqref{eq:cond-G2}, and \eqref{eq:cond-G3},
\begin{equation}\label{eq:cond-kernel-near-diagonal}
 \begin{aligned}
 G_1(x,y,z)&\le x^{2\beta}(yz)^{\beta-1/2},\\
 G_2(x,y,z)&\le2^\beta x^{2\beta}(yz)^{\beta-1/2},\\
 G_3(x,y,z)&\ge x^\beta y^\beta(yz)^{\beta-1/2}.
 \end{aligned}
\end{equation}
Since $y\ge Lx$ and $L^\beta=2(2^\beta+1)$,
\begin{align}
 &G_3(x,y,z)-G_2(x,y,z)-G_1(x,y,z)\notag\\
 &\quad\ge x^\beta(yz)^{\beta-1/2}
       \bigl[y^\beta-(2^\beta+1)x^\beta\bigr]\notag\\
 &\quad\ge\tfrac12x^\beta y^{2\beta-1/2}z^{\beta-1/2}
 =\tfrac12(z/y)^{1-\beta}
        x^\beta y^{\beta+1/2}z^{2\beta-3/2}\notag\\
 &\quad\ge\tfrac12x^\beta y^{\beta+1/2}z^{2\beta-3/2}.
\end{align}
As $0<\beta<1/4$ implies $\beta2^{\beta-1}<1/2$,
\eqref{eq:cond-high-potential} follows.
\end{proof}

\begin{proof}[Proof of \cref{lem:cond-kernel-small-inputs}]
Let $0<x\le y<z/2$. By \eqref{eq:cutoff} and the symmetry of
$S_{\beta,N}$,
\begin{align}
 \sqrt z\,K_{\beta,N}(z,y,x)
 &=\frac{S_{\beta,N}(z,y,x,z+y-x)}{\sqrt{xy}}\notag\\
 &=\frac{S_{\beta,N}(z+y-x,x,y,z)}{\sqrt{xy}}
 =\sqrt{z+y-x}\,K_{\beta,N}(z+y-x,x,y),\\
 \sqrt z\,K_{\beta,N}(z,x,y)
 &=\frac{S_{\beta,N}(z,x,y,z+x-y)}{\sqrt{xy}}\notag\\
 &=\frac{S_{\beta,N}(z+x-y,y,x,z)}{\sqrt{xy}}
 =\sqrt{z+x-y}\,K_{\beta,N}(z+x-y,y,x).
\end{align}
This proves \eqref{eq:cond-kernel-pairing}.

Now assume $N\ge2/z \vee (3z/2)^{4\beta+1/2}$.
Since $0\le y-x<z/2$,
$z/2<z+x-y\le z+y-x<3z/2$.
Thus
\begin{equation}
 \begin{aligned}
 S_\beta(x,y,z,z+y-x)&\le(3z/2)^{4\beta+1/2}\le N,\\
 S_\beta(x,y,z,z+x-y)&\le(3z/2)^{4\beta+1/2}\le N.
 \end{aligned}
\end{equation}
Also, $N\ge2/z$ gives
$\ell_N(z)=\ell_N(z+y-x)=\ell_N(z+x-y)=1$.
The bounds $z/2>y\ge x$ and $z+x-y>z/2$ imply
$\min\{x,y,z,z+y-x\}=\min\{x,y,z,z+x-y\}=x$.
Consequently, \eqref{eq:cutoff} gives
\begin{align}
 K_{\beta,N}(z,y,x)
 &=\frac{(xyz)^\beta(z+y-x)^\beta\sqrt{x}}{\sqrt{xyz}}
   \ell_N(x)\ell_N(y)\notag\\
 &=z^{2\beta-1/2}\left(1+\frac{y-x}{z}\right)^\beta
   x^\beta y^{\beta-1/2}\ell_N(x)\ell_N(y),\\
 K_{\beta,N}(z,x,y)
 &=\frac{(xyz)^\beta(z+x-y)^\beta\sqrt{x}}{\sqrt{xyz}}
   \ell_N(x)\ell_N(y)\notag\\
 &=z^{2\beta-1/2}\left(1-\frac{y-x}{z}\right)^\beta
   x^\beta y^{\beta-1/2}\ell_N(x)\ell_N(y).
\end{align}
This proves \eqref{eq:cond-kernel-small-inputs}.

Finally, $0\le y-x\le y<z$ and $\beta>0$ give
\begin{align}
 x^\beta y^{\beta-1/2}(y-x)
 &\le x^\beta y^{\beta+1/2}\le z^{2\beta+1/2},\\
 x^\beta y^{\beta-1/2}(y-x)^2
 &\le x^\beta y^{\beta+3/2}\le z^{2\beta+3/2}.
\end{align}
These are \eqref{eq:cond-kernel-first-difference} and
\eqref{eq:cond-kernel-second-difference}.
\end{proof}

\section*{Acknowledgments}
S. Luo thanks the Department of Mathematics at Duke University
for its hospitality during his visit in the summer of 2026,
when part of this work was completed.

\medskip
\noindent\textbf{Declaration of AI use.}
Parts of the proofs in Sections~4 and~5 were generated by
OpenAI's ChatGPT 6 Astra.
ChatGPT was also used for literature searches and
English-language editing.
The authors take full responsibility for the correctness
of all mathematical arguments and the entire content
of this manuscript.

\bibliographystyle{amsalpha}
\bibliography{references}

@article{DNPZ92,
  author        = {Dyachenko, S. and Newell, A. C. and Pushkarev, A. and Zakharov, V. E.},
  title         = {{Optical turbulence: weak turbulence, condensates and collapsing filaments in the nonlinear Schr{\"{o}}dinger equation}},
  journal       = {Physica D: Nonlinear Phenomena},
  year          = {1992},
  volume        = {57},
  number        = {1--2},
  pages         = {96--160},
  doi           = {10.1016/0167-2789(92)90090-a},
  url           = {https://doi.org/10.1016/0167-2789(92)90090-a},
}

@book{ZLF92,
  author        = {Zakharov, V. and Lvov, V. and Falkovich, G.},
  title         = {{Kolmogorov-Zakharov Spectra of Turbulence: Wave Turbulence}},
  series        = {Graduate Texts in Physics},
  publisher     = {Springer},
  address       = {Cham},
  year          = {2025},
  edition       = {2},
  doi           = {10.1007/978-3-031-82729-7},
  url           = {https://doi.org/10.1007/978-3-031-82729-7},
}

@book{Nazarenko11,
  author        = {Nazarenko, S.},
  title         = {{Wave Turbulence}},
  series        = {Lecture Notes in Physics},
  publisher     = {Springer},
  address       = {Berlin, Heidelberg},
  year          = {2011},
  volume        = {825},
  doi           = {10.1007/978-3-642-15942-8},
  url           = {https://doi.org/10.1007/978-3-642-15942-8},
}

@article{MMT97,
  author        = {Majda, A. J. and McLaughlin, D. W. and Tabak, E. G.},
  title         = {{A one-dimensional model for dispersive wave turbulence}},
  journal       = {Journal of Nonlinear Science},
  year          = {1997},
  volume        = {7},
  number        = {1},
  pages         = {9--44},
  doi           = {10.1007/bf02679124},
  url           = {https://doi.org/10.1007/bf02679124},
}

@article{AL25,
  author        = {Ampatzoglou, I. and L{\'{e}}ger, T.},
  title         = {{On the ill-posedness of kinetic wave equations}},
  journal       = {Nonlinearity},
  year          = {2025},
  volume        = {38},
  number        = {11},
  pages         = {115004},
  doi           = {10.1088/1361-6544/ae19bf},
  url           = {https://doi.org/10.1088/1361-6544/ae19bf},
}

@misc{Zhu26,
  author        = {Zhu, X.},
  title         = {{Sharp decay thresholds in weighted $L^\infty$ for wave kinetic equations with power-law dispersion}},
  year          = {2026},
  note          = {arXiv preprint arXiv:2607.14892},
  eprint        = {2607.14892},
  archivePrefix = {arXiv},
  primaryClass  = {math.AP},
  url           = {https://arxiv.org/abs/2607.14892},
}

@article{EV15,
  author        = {Escobedo, M. and Vel{\'{a}}zquez, J. J. L.},
  title         = {{On the theory of Weak Turbulence for the Nonlinear Schr{\"{o}}dinger Equation}},
  journal       = {Memoirs of the American Mathematical Society},
  year          = {2015},
  volume        = {238},
  number        = {1124},
  doi           = {10.1090/memo/1124},
  url           = {https://doi.org/10.1090/memo/1124},
}

@article{MerinoAceituno16,
  author        = {Merino-Aceituno, S.},
  title         = {{Isotropic wave turbulence with simplified kernels: Existence, uniqueness, and mean-field limit for a class of instantaneous coagulation-fragmentation processes}},
  journal       = {Journal of Mathematical Physics},
  year          = {2016},
  volume        = {57},
  number        = {12},
  pages         = {121501},
  doi           = {10.1063/1.4968814},
  url           = {https://doi.org/10.1063/1.4968814},
}

@article{ST26,
  author        = {Staffilani, G. and Tran, M.-B.},
  title         = {{Condensation and Non-condensation Times for 4-Wave Kinetic Equations}},
  journal       = {Archive for Rational Mechanics and Analysis},
  year          = {2026},
  volume        = {250},
  number        = {4},
  pages         = {58},
  doi           = {10.1007/s00205-026-02217-0},
  url           = {https://doi.org/10.1007/s00205-026-02217-0},
}

@misc{STMixed25,
  author        = {Staffilani, G. and Tran, M.-B.},
  title         = {{Finite time energy cascade for mixed 3- and 4-wave kinetic equations}},
  year          = {2025},
  note          = {arXiv preprint arXiv:2512.19531},
  eprint        = {2512.19531},
  archivePrefix = {arXiv},
  primaryClass  = {math-ph},
  url           = {https://arxiv.org/abs/2512.19531},
}

@article{Peierls29,
  author        = {Peierls, R.},
  title         = {{Zur kinetischen Theorie der W{\"{a}}rmeleitung in Kristallen}},
  journal       = {Annalen der Physik},
  year          = {1929},
  volume        = {395},
  number        = {8},
  pages         = {1055--1101},
  doi           = {10.1002/andp.19293950803},
  url           = {https://doi.org/10.1002/andp.19293950803},
}

@article{Phillips60,
  author        = {Phillips, O. M.},
  title         = {{On the dynamics of unsteady gravity waves of finite amplitude Part 1. The elementary interactions}},
  journal       = {Journal of Fluid Mechanics},
  year          = {1960},
  volume        = {9},
  number        = {2},
  pages         = {193--217},
  doi           = {10.1017/s0022112060001043},
  url           = {https://doi.org/10.1017/s0022112060001043},
}

@article{Hasselmann62,
  author        = {Hasselmann, K.},
  title         = {{On the non-linear energy transfer in a gravity-wave spectrum Part 1. General theory}},
  journal       = {Journal of Fluid Mechanics},
  year          = {1962},
  volume        = {12},
  number        = {4},
  pages         = {481--500},
  doi           = {10.1017/s0022112062000373},
  url           = {https://doi.org/10.1017/s0022112062000373},
}

@article{BS66,
  author        = {Benney, D. J. and Saffman, P. G.},
  title         = {{Nonlinear interactions of random waves in a dispersive medium}},
  journal       = {Proceedings of the Royal Society of London. Series A. Mathematical and Physical Sciences},
  year          = {1966},
  volume        = {289},
  number        = {1418},
  pages         = {301--320},
  doi           = {10.1098/rspa.1966.0013},
  url           = {https://doi.org/10.1098/rspa.1966.0013},
}

@article{BN69,
  author        = {Benney, D. J. and Newell, A. C.},
  title         = {{Random Wave Closures}},
  journal       = {Studies in Applied Mathematics},
  year          = {1969},
  volume        = {48},
  number        = {1},
  pages         = {29--53},
  doi           = {10.1002/sapm196948129},
  url           = {https://doi.org/10.1002/sapm196948129},
}

@article{Zakharov65,
  author        = {Zakharov, V. E.},
  title         = {{Weak turbulence in media with a decay spectrum}},
  journal       = {Journal of Applied Mechanics and Technical Physics},
  year          = {1965},
  volume        = {6},
  number        = {4},
  pages         = {22--24},
  doi           = {10.1007/bf01565814},
  url           = {https://doi.org/10.1007/bf01565814},
}

@article{Zakharov68,
  author        = {Zakharov, V. E.},
  title         = {{Stability of periodic waves of finite amplitude on the surface of a deep fluid}},
  journal       = {Journal of Applied Mechanics and Technical Physics},
  year          = {1968},
  volume        = {9},
  number        = {2},
  pages         = {190--194},
  doi           = {10.1007/bf00913182},
  url           = {https://doi.org/10.1007/bf00913182},
}

@article{GIT20,
  author        = {Germain, P. and Ionescu, A. D. and Tran, M.-B.},
  title         = {{Optimal local well-posedness theory for the kinetic wave equation}},
  journal       = {Journal of Functional Analysis},
  year          = {2020},
  volume        = {279},
  number        = {4},
  pages         = {108570},
  doi           = {10.1016/j.jfa.2020.108570},
  url           = {https://doi.org/10.1016/j.jfa.2020.108570},
}

@article{STUV26,
  author        = {Staffilani, G. and Tran, M.-B.},
  title         = {{On the Energy Transfer toward Large Values of Wavenumbers for Solutions of 4-Wave Kinetic Equations}},
  journal       = {SIAM Journal on Mathematical Analysis},
  year          = {2026},
  volume        = {58},
  number        = {4},
  pages         = {4276--4310},
  doi           = {10.1137/25m1748032},
  url           = {https://doi.org/10.1137/25m1748032},
}

@article{GLZ25,
  author        = {Germain, P. and La, J. and Zhang, K. Z.},
  title         = {{Local Well-posedness for the Kinetic MMT Model}},
  journal       = {Communications in Mathematical Physics},
  year          = {2025},
  volume        = {406},
  number        = {1},
  pages         = {18},
  doi           = {10.1007/s00220-024-05205-1},
  url           = {https://doi.org/10.1007/s00220-024-05205-1},
}

@article{Lu04,
  author        = {Lu, X.},
  title         = {{On Isotropic Distributional Solutions to the Boltzmann Equation for Bose-Einstein Particles}},
  journal       = {Journal of Statistical Physics},
  year          = {2004},
  volume        = {116},
  number        = {5--6},
  pages         = {1597--1649},
  doi           = {10.1023/b:joss.0000041750.11320.9c},
  url           = {https://doi.org/10.1023/b:joss.0000041750.11320.9c},
}

@article{Lu13,
  author        = {Lu, X.},
  title         = {{The Boltzmann Equation for Bose-Einstein Particles: Condensation in Finite Time}},
  journal       = {Journal of Statistical Physics},
  year          = {2013},
  volume        = {150},
  number        = {6},
  pages         = {1138--1176},
  doi           = {10.1007/s10955-013-0725-9},
  url           = {https://doi.org/10.1007/s10955-013-0725-9},
}

@article{EV14,
  author        = {Escobedo, M. and Vel{\'{a}}zquez, J. J. L.},
  title         = {{On the Blow Up and Condensation of Supercritical Solutions of the Nordheim Equation for Bosons}},
  journal       = {Communications in Mathematical Physics},
  year          = {2014},
  volume        = {330},
  number        = {1},
  pages         = {331--365},
  doi           = {10.1007/s00220-014-2034-9},
  url           = {https://doi.org/10.1007/s00220-014-2034-9},
}

@article{EVN15,
  author        = {Escobedo, M. and Vel{\'{a}}zquez, J. J. L.},
  title         = {{Finite time blow-up and condensation for the bosonic Nordheim equation}},
  journal       = {Inventiones mathematicae},
  year          = {2015},
  volume        = {200},
  number        = {3},
  pages         = {761--847},
  doi           = {10.1007/s00222-014-0539-7},
  url           = {https://doi.org/10.1007/s00222-014-0539-7},
}

@article{CL19,
  author        = {Cai, S. and Lu, X.},
  title         = {{The Spatially Homogeneous Boltzmann Equation for Bose--Einstein Particles: Rate of Strong Convergence to Equilibrium}},
  journal       = {Journal of Statistical Physics},
  year          = {2019},
  volume        = {175},
  number        = {2},
  pages         = {289--350},
  doi           = {10.1007/s10955-019-02258-8},
  url           = {https://doi.org/10.1007/s10955-019-02258-8},
}

@article{CL26,
  author        = {Cai, S. and Lu, X.},
  title         = {{Convergence and non-convergence to Bose--Einstein condensation}},
  journal       = {Journal of Mathematical Physics},
  year          = {2026},
  volume        = {67},
  number        = {6},
  pages         = {063301},
  doi           = {10.1063/5.0258701},
  url           = {https://doi.org/10.1063/5.0258701},
}

@misc{LL26BEC,
  author        = {Luo, S. and Liu, J.-G.},
  title         = {{Bose--Einstein Condensation without an Initial Low-Energy Concentration Assumption}},
  year          = {2026},
  note          = {arXiv preprint arXiv:2609.01436},
  eprint        = {2609.01436},
  archivePrefix = {arXiv},
  primaryClass  = {math.AP},
  url           = {https://arxiv.org/abs/2609.01436},
}

@article{Spohn06,
  author        = {Spohn, H.},
  title         = {{The Phonon Boltzmann Equation, Properties and Link to Weakly Anharmonic Lattice Dynamics}},
  journal       = {Journal of Statistical Physics},
  year          = {2006},
  volume        = {124},
  number        = {2--4},
  pages         = {1041--1104},
  doi           = {10.1007/s10955-005-8088-5},
  url           = {https://doi.org/10.1007/s10955-005-8088-5},
}

@article{GST20,
  author        = {Gamba, I. M. and Smith, L. M. and Tran, M.-B.},
  title         = {{On the wave turbulence theory for stratified flows in the ocean}},
  journal       = {Mathematical Models and Methods in Applied Sciences},
  year          = {2020},
  volume        = {30},
  number        = {1},
  pages         = {105--137},
  doi           = {10.1142/s0218202520500037},
  url           = {https://doi.org/10.1142/s0218202520500037},
}

@article{NT18,
  author        = {Nguyen, T. T. and Tran, M.-B.},
  title         = {{On the Kinetic Equation in Zakharov's Wave Turbulence Theory for Capillary Waves}},
  journal       = {SIAM Journal on Mathematical Analysis},
  year          = {2018},
  volume        = {50},
  number        = {2},
  pages         = {2020--2047},
  doi           = {10.1137/17m1125042},
  url           = {https://doi.org/10.1137/17m1125042},
}

@article{CTSB20,
  author        = {Tran, M.-B. and Craciun, G. and Smith, L. M. and Boldyrev, S.},
  title         = {{A reaction network approach to the theory of acoustic wave turbulence}},
  journal       = {Journal of Differential Equations},
  year          = {2020},
  volume        = {269},
  number        = {5},
  pages         = {4332--4352},
  doi           = {10.1016/j.jde.2020.03.034},
  url           = {https://doi.org/10.1016/j.jde.2020.03.034},
}

@article{SoTr18,
  author        = {Soffer, A. and Tran, M.-B.},
  title         = {{On the dynamics of finite temperature trapped Bose gases}},
  journal       = {Advances in Mathematics},
  year          = {2018},
  volume        = {325},
  pages         = {533--607},
  doi           = {10.1016/j.aim.2017.12.007},
  url           = {https://doi.org/10.1016/j.aim.2017.12.007},
}

@article{SoTr20,
  author        = {Soffer, A. and Tran, M.-B.},
  title         = {{On the Energy Cascade of 3-Wave Kinetic Equations: Beyond Kolmogorov--Zakharov Solutions}},
  journal       = {Communications in Mathematical Physics},
  year          = {2020},
  volume        = {376},
  number        = {3},
  pages         = {2229--2276},
  doi           = {10.1007/s00220-019-03651-w},
  url           = {https://doi.org/10.1007/s00220-019-03651-w},
}

@article{LS11,
  author        = {Lukkarinen, J. and Spohn, H.},
  title         = {{Weakly nonlinear Schr{\"{o}}dinger equation with random initial data}},
  journal       = {Inventiones mathematicae},
  year          = {2011},
  volume        = {183},
  number        = {1},
  pages         = {79--188},
  doi           = {10.1007/s00222-010-0276-5},
  url           = {https://doi.org/10.1007/s00222-010-0276-5},
}

@article{BGHS21,
  author        = {Buckmaster, T. and Germain, P. and Hani, Z. and Shatah, J.},
  title         = {{Onset of the wave turbulence description of the longtime behavior of the nonlinear Schr{\"{o}}dinger equation}},
  journal       = {Inventiones mathematicae},
  year          = {2021},
  volume        = {225},
  number        = {3},
  pages         = {787--855},
  doi           = {10.1007/s00222-021-01039-z},
  url           = {https://doi.org/10.1007/s00222-021-01039-z},
}

@article{DH21,
  author        = {Deng, Y. and Hani, Z.},
  title         = {{On the derivation of the wave kinetic equation for NLS}},
  journal       = {Forum of Mathematics, Pi},
  year          = {2021},
  volume        = {9},
  pages         = {e6},
  doi           = {10.1017/fmp.2021.6},
  url           = {https://doi.org/10.1017/fmp.2021.6},
}

@article{CGHom,
  author        = {Collot, C. and Germain, P.},
  title         = {{On the derivation of the homogeneous kinetic wave equation}},
  journal       = {Communications on Pure and Applied Mathematics},
  year          = {2025},
  volume        = {78},
  number        = {4},
  pages         = {856--909},
  doi           = {10.1002/cpa.22232},
  url           = {https://doi.org/10.1002/cpa.22232},
}

@article{CGLong,
  author        = {Collot, C. and Germain, P.},
  title         = {{Derivation of the homogeneous kinetic wave equation: Longer time scales}},
  journal       = {Journal of Functional Analysis},
  year          = {2026},
  volume        = {290},
  number        = {1},
  pages         = {111179},
  doi           = {10.1016/j.jfa.2025.111179},
  url           = {https://doi.org/10.1016/j.jfa.2025.111179},
}

@article{DH23,
  author        = {Deng, Y. and Hani, Z.},
  title         = {{Full derivation of the wave kinetic equation}},
  journal       = {Inventiones mathematicae},
  year          = {2023},
  volume        = {233},
  number        = {2},
  pages         = {543--724},
  doi           = {10.1007/s00222-023-01189-2},
  url           = {https://doi.org/10.1007/s00222-023-01189-2},
}

@article{DH26Scaling,
  author        = {Deng, Y. and Hani, Z.},
  title         = {{Derivation of the Wave Kinetic Equation: Full Range of Scaling Laws}},
  journal       = {Memoirs of the American Mathematical Society},
  year          = {2026},
  volume        = {320},
  number        = {1631},
  doi           = {10.1090/memo/1631},
  url           = {https://doi.org/10.1090/memo/1631},
}

@misc{DHLong23,
  author        = {Deng, Y. and Hani, Z.},
  title         = {{Long time justification of wave turbulence theory}},
  year          = {2023},
  note          = {arXiv preprint arXiv:2311.10082},
  eprint        = {2311.10082},
  archivePrefix = {arXiv},
  primaryClass  = {math.AP},
  url           = {https://arxiv.org/abs/2311.10082},
}

@misc{STKdV21,
  author        = {Staffilani, G. and Tran, M.-B.},
  title         = {{On the wave turbulence theory for a stochastic KdV type equation}},
  year          = {2021},
  note          = {arXiv preprint arXiv:2106.09819},
  eprint        = {2106.09819},
  archivePrefix = {arXiv},
  primaryClass  = {math.AP},
  url           = {https://arxiv.org/abs/2106.09819},
}

@article{ACGInhom,
  author        = {Ampatzoglou, I. and Collot, C. and Germain, P.},
  title         = {{Derivation of the kinetic wave equation for quadratic dispersive problems in the inhomogeneous setting}},
  journal       = {American Journal of Mathematics},
  year          = {2025},
  volume        = {147},
  number        = {4},
  pages         = {1053--1158},
  doi           = {10.1353/ajm.2025.a966292},
  url           = {https://doi.org/10.1353/ajm.2025.a966292},
}

@misc{HRSTInhom,
  author        = {Hannani, A. and Rosenzweig, M. and Staffilani, G. and Tran, M.-B.},
  title         = {{On the wave turbulence theory for a stochastic KdV type equation -- Generalization for the inhomogeneous kinetic limit}},
  year          = {2022},
  note          = {arXiv preprint arXiv:2210.17445},
  eprint        = {2210.17445},
  archivePrefix = {arXiv},
  primaryClass  = {math.AP},
  url           = {https://arxiv.org/abs/2210.17445},
}

@article{DKKinetic,
  author        = {Dymov, A. and Kuksin, S.},
  title         = {{Formal Expansions in Stochastic Model for Wave Turbulence 1: Kinetic Limit}},
  journal       = {Communications in Mathematical Physics},
  year          = {2021},
  volume        = {382},
  number        = {2},
  pages         = {951--1014},
  doi           = {10.1007/s00220-021-03955-w},
  url           = {https://doi.org/10.1007/s00220-021-03955-w},
}

@article{GHDriven,
  author        = {Grande, R. and Hani, Z.},
  title         = {{Rigorous Derivation of Damped--Driven Wave Turbulence Theory}},
  journal       = {Archive for Rational Mechanics and Analysis},
  year          = {2026},
  volume        = {250},
  number        = {3},
  pages         = {27},
  doi           = {10.1007/s00205-026-02181-9},
  url           = {https://doi.org/10.1007/s00205-026-02181-9},
}

@article{Lu14,
  author        = {Lu, X.},
  title         = {{The Boltzmann Equation for Bose--Einstein Particles: Regularity and Condensation}},
  journal       = {Journal of Statistical Physics},
  year          = {2014},
  volume        = {156},
  number        = {3},
  pages         = {493--545},
  doi           = {10.1007/s10955-014-1026-7},
  url           = {https://doi.org/10.1007/s10955-014-1026-7},
}
\end{document}